\documentclass[a4paper,12pt]{article}

\usepackage{amsmath,amssymb,amsfonts,amsthm}
\usepackage{multirow,bigdelim}
\usepackage{amscd}
\usepackage{verbatim}
\usepackage{latexsym}
\usepackage{array}
\usepackage{enumerate}
\usepackage[all]{xy}
\usepackage{graphicx}
\usepackage{ulem}
\usepackage{ascmac}
\usepackage{mathtools}
\usepackage{layout}
\usepackage{bbm,enumitem,empheq,nicematrix,relsize,calligra}
\usepackage[initials,alphabetic]{amsrefs}
   \BibSpec{arXiv}{
   +{}{\PrintAuthors}{author}
   +{,}{ \textit}{title}
   +{,}{ }{date}
   +{,}{ arXiv:}{eprint}}
\usepackage{mleftright}
   \mleftright
\usepackage{tikz}
   \usetikzlibrary{cd,decorations,decorations.pathreplacing,arrows,calc,intersections,hobby}

\numberwithin{equation}{section}

\newcommand{\rHom}{\mathrm{RHom}}
\newcommand{\BDC}{{\mathbf{D}}^{\mathrm{b}}}

\newcommand{\Mod}{\mathrm{Mod}}

\newcommand{\rsect}{{\mathrm{R}}\Gamma}
\newcommand{\CC}{\mathbb{C}}

\newcommand{\RR}{\mathbb{R}}
\newcommand{\QQ}{\mathbb{Q}}
\newcommand{\ZZ}{\mathbb{Z}}

\newcommand{\M}{\mathcal{M}}

\renewcommand{\SS}{\mathcal{S}}

\newcommand{\PP}{{\mathbb P}}

\newcommand{\Cal}{\cal}

\newcommand{\an}{{\rm an}}

\newcommand{\e}{\varepsilon}

\newcommand{\id}{{\rm id}}

\newcommand{\Sol}{{\rm Sol}}

\newcommand{\Dbc}{{\bf D}_{c}^{b}}
\newcommand{\Kbc}{{\bf K}_{c}^{b}}
\newcommand{\tl}[1]{\widetilde{#1}}
\newcommand{\uotimes}[1]{\otimes_{#1}}

\newcommand{\simto}{\overset{\sim}{\longrightarrow}}

\newcommand{\dsum}{\displaystyle \sum}

\newcommand{\CF}{{\rm CF}}

\newcommand{\op}{\mbox{\scriptsize op}}
\newcommand{\SD}{\mathcal{D}}
\newcommand{\SDop}{\mathcal{D}^{\mbox{\scriptsize op}}}
\newcommand{\SO}{\mathcal{O}}
\newcommand{\SA}{\mathcal{A}}
\newcommand{\SM}{\mathcal{M}}
\newcommand{\SN}{\mathcal{N}}

\newcommand{\SE}{\mathcal{E}}
\newcommand{\SF}{\mathcal{F}}

\newcommand{\Modhol}{\mathrm{Mod}_{\mathrm{hol}}}

\newcommand{\BDCcoh}{{\mathbf{D}}^{\mathrm{b}}_{\mbox{\scriptsize coh}}}

\newcommand{\BDChol}{{\mathbf{D}}^{\mathrm{b}}_{\mathrm{hol}}}
\newcommand{\BDCrh}{{\mathbf{D}}^{\mathrm{b}}_{\mbox{\scriptsize rh}}}
\newcommand{\DD}{\mathbb{D}}
\newcommand{\Lotimes}[1]{\overset{L}{\otimes}_{#1}}
\newcommand{\Dotimes}{\overset{D}{\otimes}}

\newcommand{\Potimes}{\overset{+}{\otimes}}

\newcommand{\rhom}{{\rm R}{\mathcal{H}}om}
\newcommand{\rihom}{{\rm R}{\mathcal{I}}hom}
\newcommand{\Prihom}{{\rm R}{\mathcal{I}}hom^+}
\newcommand{\Prhom}{\rhom^+}
\newcommand{\I}{{\rm I}}
\newcommand{\che}[1]{\overset{\vee}{#1}}
\newcommand{\var}[1]{\overline{#1}}
\newcommand{\BEC}{{\mathbf{E}}^{\mathrm{b}}}
\newcommand{\Q}{\mathbf{Q}}
\newcommand{\EE}{\mathbb{E}}

\newcommand{\T}{{\rm T}}
\newcommand{\bfR}{\mathbf{R}}
\newcommand{\bfL}{\mathbf{L}}
\newcommand{\bfD}{\mathbf{D}}
\newcommand{\rmR}{{\mathrm{R}}}
\newcommand{\rmE}{{\mathrm{E}}}
\newcommand{\rmD}{{\mathrm{D}}}
\newcommand{\rmt}{{\mathrm{t}}}
\newcommand{\bfE}{\mathbf{E}}

\renewcommand{\Re}{\operatorname{Re}}

\newcommand{\SP}{\mathcal{P}}

\newcommand{\rmI}{\mathrm{I}}

\newcommand{\pt}{\mathrm{pt}}

\newcommand{\dR}{\mathrm{dR}}
\newcommand{\rmb}{\mathrm{b}}

\newcommand{\reg}{\mathrm{reg}}
\newcommand{\irr}{\mathrm{irr}}
\newcommand{\red}{\mathrm{red}}

\newcommand{\DRalg}{DR^{\mathrm{alg}}}
\newcommand{\ssupp}{\operatorname{sing}\!.\!\operatorname{supp}} 

\DeclareMathOperator{\supp}{supp}
\DeclareMathOperator{\rank}{rank}

\DeclareMathOperator{\CCyc}{CC} 
\DeclareMathOperator{\Eu}{Eu} 

\newcommand{\rcEC}{\mathbf{E}_{\mathbb{R}{\text -}\mathrm c}}
\newcommand{\CCirr}{\mathrm{CC}_{\mathrm{irr}}} 
\newcommand{\bs}{\left.\right\backslash} 
\newcommand{\vbar}{\left.\right|} 

\DeclarePairedDelimiter{\abs}{\lvert}{\rvert} 
 
\DeclarePairedDelimiterX{\Set}[2]{\lbrace}{\rbrace}{#1\ \delimsize\vert\ #2}

\newtheorem{theorem}{Theorem}[section]

\newtheorem{corollary}[theorem]{Corollary}
\newtheorem{lemma}[theorem]{Lemma}
\newtheorem{proposition}[theorem]{Proposition}

\theoremstyle{definition}
\newtheorem{definition}[theorem]{Definition}
\theoremstyle{remark}
\newtheorem{remark}[theorem]{\sc Remark}
\newtheorem{example}[theorem]{\sc Example}

\title{On characteristic cycles of
irregular holonomic {$\SD$}-modules 
}

\author{Kazuki KUDOMI
\footnote{Mathematical Institute, Tohoku University,
Aramaki Aza-Aoba 6-3, Aobaku, Sendai, 980-8578, Japan.
E-mail: kazuki.kudomi.q3@dc.tohoku.ac.jp}
and Kiyoshi TAKEUCHI
\footnote{Mathematical Institute, Tohoku University,
Aramaki Aza-Aoba 6-3, Aobaku, Sendai, 980-8578, Japan.
E-mail: takemicro@nifty.com} }

\begin{document}

\maketitle

\begin{abstract}
Based on the recent progress in the irregular 
Riemann-Hilbert correspondence for holonomic D-modules, 
we show that the characteristic cycles of some standard 
irregular holonomic D-modules can be expressed 
as in the classical theorem of Ginsburg. For this purpose, 
we first prove a formula for the enhanced solution complexes of 
holonomic D-modules having a quasi-normal form, via 
which, to our surprise, their solution complexes can be 
calculated more easily by topological methods. 
In the formulation and the proof of our main theorems, 
not necessarily homogeneous Lagrangian cycles 
that we call irregular characteristic cycles 
will play a crucial role. 
\end{abstract}

\section{Introduction}\label{sec:1}

\indent In the theory of D-modules, the most basic 
and important objects we study are their solution complexes 
i.e. the complexes of their holomorphic solutions. 
Indeed, in the classical Riemann-Hilbert correspondence, 
they were used to establish an equivalence between the categories 
of regular holonomic D-modules and perverse sheaves. Moreover, 
by the solution complexes of holonomic D-modules, 
we obtain their most important geometric invariants i.e. 
their characteristic cycles. See e.g. Kashiwara-Schapira 
\cite[Chapters X and XI]{KS90}. However for irregular 
holonomic D-modules, we know only very little about their 
solution complexes. 

\medskip 
\indent The aim of this paper is to study the solution complexes 
and the characteristic cycles of irregular 
holonomic D-modules in the light of the irregular 
Riemann-Hilbert correspondence established by 
D'Agnolo and Kashiwara in \cite{DK16}. To our surprise, 
the solution complexes are sometimes calculated more easily 
by topological methods via the enhanced ones introduced in 
\cite{DK16} (see Proposition \ref{prop-1} and 
Corollary \ref{cor-1}). Taking this advantage, for some standard 
holonomic D-modules, we define (not necessarily homogeneous) 
Lagrangian cycles that we call irregular characteristic cycles 
and use them to obtain Ginsburg type 
formulas for their (usual) characteristic cycles 
similar to the one in Ginsburg 
\cite[Theorem 3.3]{Gin86}. 

\medskip 
\indent In order to explain our results more precisely, 
let $X$ be a complex manifold and consider a holonomic 
$\SD_X$-module $\SM$ on it. Then by Kashiwara's 
constructibility theorem its solution complex 
\begin{equation}
Sol_X ( \SM ):= \rhom_{\SD_X}(\SM, \SO_X) \qquad 
\in \BDC(\CC_X)
\end{equation}
is constructible. In fact, Kashiwara also proved that 
it is a perverse sheaf on $X$ up to some shift. First,  
for simplicity, let us consider the case where $\SM$ is a 
meromorphic connection on $X$ along a closed hypersurface 
$D \subset X$. In this case, if moreover $\SM$ is 
regular then after some fundamental works by 
Deligne \cite{Del70} and Kashiwara-Kawai \cite{KK81} 
we know that for the inclusion map $j: X \setminus D 
\hookrightarrow X$ there exists an isomorphism 
\begin{equation}
Sol_X ( \SM ) \simeq j_!j^{-1} Sol_X ( \SM )
\end{equation}
and hence $Sol_X ( \SM )|_D \simeq 0$. Namely 
for a regular meromorphic connection $\SM$ 
noting interesting can happen over the 
hypersurface $D \subset X$. But for an irregular 
meromorphic connection $\SM$ this does not hold 
in general. Indeed, if the dimension $\dim X$ of $X$ 
is one and $D$ is a point i.e. $D= \{ x \}$ for 
some $x \in X$, the local Euler-Poincar{\'e} index 
\begin{equation}
\chi (Sol_X ( \SM ))(x):= \dsum_{j \in \ZZ} (-1)^j 
\dim H^j Sol_X ( \SM )_x \qquad \in \ZZ
\end{equation}
of the solution complex $Sol_X ( \SM )$ at 
the point $x \in X$ is equal to the minus of the 
irregularity of $\SM$, which is equal to zero 
if and only if $\SM$ is regular (see 
e.g. Sabbah \cite{Sab93} for an introduction 
to this subject). Now, to introduce 
our results in higher dimensions, we first consider 
the case where $D \subset X$ is a normal crossing divisor 
in $X$ and the meromorphic connection $\SM$ has 
a quasi-normal form along it in the 
sense of Mochizuki \cite[Chapter 5]{Moc11}. In this case, 
in Proposition \ref{hda-dk} for any point 
$x \in D$ we obtain a formula for the enhanced 
solution complex $Sol_X^\rmE ( \SM )$ of $\SM$ 
over a neighborhood $U \subset X$ of $x$ in $X$. 
This is a higher-dimensional analogue of 
the result of \cite[Proposition 5.4.5]{DK18}. 
By the proof of Proposition \ref{hda-dk}, 
we see also that the solution complex $Sol_X ( \SM )$ 
can be calculated by topological methods 
via the enhanced one $Sol_X^\rmE ( \SM )$. For 
such calculations, see Proposition \ref{prop-1} and 
Corollary \ref{cor-1}. With Proposition \ref{hda-dk} 
and its proof at hands, we define a (not necessarily homogeneous) 
Lagrangian cycle $\CCirr(\SM)$ in the open subset 
$T^\ast U\subset T^\ast X$ that we call the irregular characteristic cycle of 
the meromorphic connection $\SM$ 
as in \cite{Tak22} and \cite{KT23} and use them to prove 
the following Ginsburg type formula for the (usual) 
characteristic cycle $\CCyc(\SM)$ of $\SM$. We shall say 
that a holomorphic function $g\colon U\longrightarrow\CC$ on $U$ is a 
defining holomorphic function of the divisor $D\cap U\subset U$ 
if we have $g^{-1}(0)= D\cap U$ set-theoretically. 

\begin{theorem}\label{intro-thm-5-2-1}
In the situation as above, let $g\colon U\longrightarrow\CC$ 
be a defining holomorphic 
function of the normal crossing divisor $D\cap U\subset U$. 
Then in the open subset $T^\ast U\subset T^\ast X$ we have 
\begin{equation}
\CCyc(\SM)=\lim_{t\rightarrow+0}
t\Bigl\{\CCirr(\SM)+d\log g\Bigr\},
\end{equation}
where the limit in the right hand side stands for that of Lagrangian cycles
(see \cite[Section 2.2]{FKT26}).
\end{theorem}
Next we consider the following holonomic $\SD_X$-modules. 

\begin{definition}\label{intro-def-expD} 
(cf. \cite{Tak22}) Let $X$ be a complex manifold. 
Then we say that a holonomic $\SD_X$-module $\SM$ is an exponentially twisted holonomic
$\SD_X$-module if there exist a meromorphic function $f\in\SO_X(\ast Y)$ along a
closed hypersurface $Y\subset X$ and a regular holonomic $\SD_X$-module $\SN$ such 
that we have an isomorphism
\begin{equation}
\SM\simeq \SE_{X\setminus Y\vbar X}^f\Dotimes\SN.
\end{equation}
\end{definition}
In view of the recent progress in Kedlaya \cite{Ked10, Ked11}, 
Mochizuki \cite{Moc11} and D'Agnolo-Kashiwara \cite{DK16}, 
the exponentially twisted holonomic D-modules in Definition 
\ref{intro-def-expD} can be considered as natural prototypes or building blocks 
of general holonomic D-modules, and their Fourier transforms 
were studied precisely in \cite{Tak22}. 
For an exponentially twisted holonomic $\SD_X$-module $\SM$ in Definition 
\ref{intro-def-expD} we define a (not necessarily homogeneous) Lagrangian cycle $\CCirr(\SM)$
in $T^\ast (X \setminus Y) \subset T^\ast X$ by
\begin{equation}
\CCirr(\SM)\coloneq \CCyc(\SN\vbar_{X \setminus Y}) +df.
\end{equation} 
We call it the irregular characteristic cycle of $\SM$. 
Note that if $X$ is not compact it depends not only on $\SM$ itself  but also on
the decomposition $\SM\simeq \SE_{X\setminus Y\vbar X}^f\Dotimes\SN$ of $\SM$.
Then we obtain the following  Ginsburg type formula for the (usual) 
characteristic cycle $\CCyc(\SM)$ of $\SM$. 

\begin{theorem}\label{mthm-intro} 
Let $\SM, f, \SN$ etc. be as in Definition \ref{intro-def-expD} and 
$g\colon X\longrightarrow\CC$ a (local) defining holomorphic function of the divisor $Y\subset X$.
Then we have
\begin{equation}
\CCyc(\SM)=\lim_{t\rightarrow+0}t\Bigl\{\CCirr(\SM)+d\log g\Bigr\}.
\end{equation}
\end{theorem}
Recall that the formula of Ginsburg 
in \cite[Theorem 3.3]{Gin86} describes the characteristic cycles of 
the localizations of regular holonomic $\SD_X$-modules 
along closed hypersurfaces $Y \subset X$ and it was 
generalized later to real constructible sheaves (including 
also the o-minimal ones) by Schmid and Vilonen 
in \cite{SV96}. For the proof of Theorems \ref{intro-thm-5-2-1} 
and \ref{mthm-intro}, we use the methods in \cite{SV96}. 
Note also that some of the intermediate 
results in this paper e.g. the assertions (ii) and (iii) 
of Proposition \ref{thm-1} and Corollary \ref{thm-CCquasi} have been 
obtained previously by Hu and Teyssier in \cite{HT25} by a totally 
different method. Whereas our proof relies on the 
theories of ind-sheaves and the irregular Riemann-Hilbert correspondence, 
Hu and Teyssier use Sabbah's study of irregularity sheaves 
in \cite{Sab17}. It is remarkable that in the two dimensional case 
$\dim X=2$ they proved their formula of (usual) 
characteristic cycles for all holonomic D-modules. 

\medskip 
\indent This paper is organized as follows. First, in 
Section \ref{uni-sec:2}, we recall some basic notions and 
results which will be used in this paper. In Section 
\ref{sec:qNor}, we study the enhanced and usual solution complexes of  
irregular holonomic D-modules having a quasi-normal form 
in the sense of Mochizuki \cite[Chapter 5]{Moc11}. 
In Section \ref{sec-index}, using the results in Section 
\ref{sec:qNor} we obtain an index formula for 
irregular integrable connections, which expresses the global 
Euler-Poincar{\'e} indices of their algebraic de Rham complexes. 
Then finally in Section \ref{sec-Gins}, we prove 
the Ginsburg type formulas for characteristic cycles in 
Theorems \ref{intro-thm-5-2-1} and \ref{mthm-intro}.

\section{Preliminary Notions and Results}\label{uni-sec:2}
In this section, we briefly recall some basic notions 
and results which will be used in this paper. 
We assume here that the reader is familiar with the 
theory of sheaves and functors in the framework of 
derived categories. For them we follow the terminologies 
in \cite{KS90} etc. For a topological space $M$ 
denote by $\BDC(\CC_M)$ the derived category 
consisting of bounded 
complexes of sheaves of $\CC$-vector spaces on it.

\subsection{Ind-sheaves}\label{sec:3}
We recall some basic notions 
and results on ind-sheaves. References are made to 
Kashiwara-Schapira \cite{KS01} and \cite{KS06}. 
Let $M$ be a good topological space (which is locally compact, 
Hausdorff, countable at infinity and has finite soft dimension). 
We denote by $\Mod(\CC_M)$ the abelian category of sheaves 
of $\CC$-vector spaces on it and by $\I\CC_M$ 
that of ind-sheaves. Then there exists a 
natural exact embedding $\iota_M : \Mod(\CC_M)\to\I\CC_M$ 
of categories. We sometimes omit it.
It has an exact left adjoint $\alpha_M$, 
that has in turn an exact fully faithful
left adjoint functor $\beta_M$: 
 \[\xymatrix@C=60pt{\Mod(\CC_M)  \ar@<1.0ex>[r]^-{\iota_{M}} 
 \ar@<-1.0ex>[r]_- {\beta_{M}} & \I\CC_M 
\ar@<0.0ex>[l]|-{\alpha_{M}}}.\]

The category $\I\CC_M$ does not have enough injectives. 
Nevertheless, we can construct the derived category $\BDC(\I\CC_M)$ 
for ind-sheaves and the Grothendieck six operations among them. 
We denote by $\otimes$ and $\rihom$ the operations 
of tensor products and internal homs respectively. 
If $f : M\to N$ be a continuous map, we denote 
by $f^{-1}, \rmR f_\ast, f^!$ and $\rmR f_{!!}$ the 
operations of inverse images,
direct images, proper inverse images and proper direct images 
respectively. 
We set also $\rhom := \alpha_M\circ\rihom$. 
We thus obtain the functors
\begin{align*}
\iota_M &: \BDC(\CC_M)\to \BDC(\I\CC_M),\\
\alpha_M &: \BDC(\I\CC_M)\to \BDC(\CC_M),\\
\beta_M &: \BDC(\CC_M)\to \BDC(\I\CC_M),\\
\otimes &: \BDC(\I\CC_M)\times\BDC(\I\CC_M)\to\BDC(\I\CC_M), \\
\rihom &: \BDC(\I\CC_M)^{\op}\times\BDC(\I\CC_M)\to\BDC(\I\CC_M), \\
\rhom &: \BDC(\I\CC_M)^{\op}\times\BDC(\I\CC_M)\to\BDC(\CC_M), \\
\rmR f_\ast &: \BDC(\I\CC_M)\to\BDC(\I\CC_N),\\
f^{-1} &: \BDC(\I\CC_N)\to\BDC(\I\CC_M),\\
\rmR f_{!!} &: \BDC(\I\CC_M)\to\BDC(\I\CC_N),\\
f^! &: \BDC(\I\CC_N)\to\BDC(\I\CC_M).
\end{align*}

Note that $(f^{-1}, \rmR f_\ast)$ and 
$(\rmR f_{!!}, f^!)$ are pairs of adjoint functors.
We may summarize the commutativity of the various functors 
we have introduced in the table below. 
Here, $``\circ"$ means that the functors commute,
and $``\times"$ they do not.
\begin{table}[h]
\begin{equation*}
   \begin{tabular}{l||c|c|c|c|c|c|c}
    {} & $\otimes$ & $f^{-1}$ & $\rmR f_\ast$ & $f^!$ & 
$\rmR f_{!!}$ & $\underset{}{\varinjlim}$ &  $\varprojlim$ \\ \hline \hline 
    $\underset{}{\iota}$ & $\circ$ & $\circ$ & $\circ$ & 
$\circ$ & $\times$ & $\times$ & $\circ$  \\ \hline
    $\underset{}{\alpha}$ & $\circ$ & $\circ$ & $\circ$ & 
$\times$ & $\circ$ & $\circ$ & $\circ$ \\ \hline
    $\underset{}{\beta}$ & $\circ$ & $\circ$ & $\times$  
& $\times$ & $\times$ & $\circ$& $\times$\\ \hline 
     $\underset{}{\varinjlim}$ & $\circ$ & $\circ$ & 
$\times$ & $\circ$ & $\circ$ &\multicolumn{2}{|c}{} \\\cline{1-6}
     $\underset{}{\varprojlim}$ & $\times$ & $\times$ & 
$\circ$ & $\times$ & $\times$ &\multicolumn{2}{|c}{} \\\cline{1-6}
   \end{tabular}
   \end{equation*}
\end{table}

\subsection{Ind-sheaves on Bordered Spaces}\label{sec:4} 
For the results in this subsection, we refer to 
D'Agnolo-Kashiwara \cite{DK16}. 
A bordered space is a pair $M_{\infty} = (M, \che{M})$ of
a good topological space $\che{M}$ and an open subset $M\subset\che{M}$.
A morphism $f : (M, \che{M})\to (N, \che{N})$ of bordered spaces
is a continuous map $f : M\to N$ such that the first projection
$\che{M}\times\che{N}\to\che{M}$ is proper on
the closure $\var{\Gamma}_f$ of the graph $\Gamma_f$ of $f$ 
in $\che{M}\times\che{N}$.
If also the second projection $\var{\Gamma}_f\to\che{N}$ is proper, 
we say that $f$ is semi-proper. 
The category of good topological spaces embeds into that
of bordered spaces by the identification $M = (M, M)$. 
We define the triangulated category of ind-sheaves on 
$M_{\infty} = (M, \che{M})$ by 
\[\BDC(\I\CC_{M_\infty}) := 
\BDC(\I\CC_{\che{M}})/\BDC(\I\CC_{\che{M}\backslash M}).\]
The quotient functor
\[\mathbf{q} : \BDC(\I\CC_{\che{M}})\to\BDC(\I\CC_{M_\infty})\]
has a left adjoint $\mathbf{l}$ and a right 
adjoint $\mathbf{r}$, both fully faithful, defined by 
\[\mathbf{l}(\mathbf{q} F) := \CC_M\otimes F,\hspace{25pt} 
\mathbf{r}(\mathbf{q} F) := \rihom(\CC_M, F). \]
For a morphism $f : M_\infty\to N_\infty$ 
of bordered spaces, 
the Grothendieck's operations 
\begin{align*} 
\otimes &: \BDC(\I\CC_{M_\infty})\times
\BDC(\I\CC_{M_\infty})\to\BDC(\I\CC_{M_\infty}), \\
\rihom &: \BDC(\I\CC_{M_\infty})^{\op}\times
\BDC(\I\CC_{M_\infty})\to\BDC(\I\CC_{M_\infty}), \\
\rmR f_\ast &: \BDC(\I\CC_{M_\infty})\to\BDC(\I\CC_{N_\infty}),\\
f^{-1} &: \BDC(\I\CC_{N_\infty})\to\BDC(\I\CC_{M_\infty}),\\
\rmR f_{!!} &: \BDC(\I\CC_{M_\infty})\to\BDC(\I\CC_{N_\infty}),\\
f^! &: \BDC(\I\CC_{N_\infty})\to\BDC(\I\CC_{M_\infty}) \\
\end{align*}
are defined by
\begin{align*} 
\mathbf{q}(F)\otimes\mathbf{q}(G) &:=
\mathbf{q}(F\otimes G), \\
\rihom(\mathbf{q}(F), \mathbf{q}(G)) &:=
\mathbf{q}\big(\rihom(F, G)\big), \\
\rmR f_\ast(\mathbf{q}(F)) &:=
\mathbf{q}\big(\rmR {\rm pr}_{2\ast}\rihom(
\CC_{\Gamma_f}, {\rm pr}_1^!F)\big),\\
f^{-1} (\mathbf{q}(G))&:=\mathbf{q}\big(\rmR{\rm pr_{1!!}}
(\CC_{\Gamma_f}\otimes {\rm pr_2}^{-1}G)\big),\\
\rmR f_{!!}(\mathbf{q}(F)) &:=
\mathbf{q}\big(\rmR{\rm pr_2}_{!!}(
\CC_{\Gamma_f}\otimes{\rm pr_1}^{-1}F)\big),\\
f^!(\mathbf{q}(G)) &:=\mathbf{q}
\big(\rmR {\rm pr}_{1\ast}\rihom(\CC_{\Gamma_f}, {\rm pr}_2^!G)\big)
\end{align*}
respectively, where ${\rm pr_1} : \che{M}\times\che{N}\to\che{M}$ and 
${\rm pr_2} : \che{M}\times\che{N}\to\che{N}$ 
are the projections. Moreover, there exists a natural embedding 
\[\xymatrix@C=30pt@M=10pt{\BDC(\CC_M)\ar@{^{(}->}[r] & 
\BDC(\I\CC_{M_\infty}).}\]

\subsection{Enhanced Sheaves}\label{sec:5}
For the results in this subsection, see 
D'Agnolo-Kashiwara \cite{DK16} and 
Kashiwara-Schapira \cite{KS16}. 
Let $M$ be a good topological space. 
We consider the maps 
\[M\times\RR^2\xrightarrow{p_1, p_2, \mu}M
\times\RR\overset{\pi}{\longrightarrow}M\]
where $p_1, p_2$ are the first and the second projections 
and we set $\pi (x,t):=x$ and 
$\mu(x, t_1, t_2) := (x, t_1+t_2)$. 
Then the convolution functors for 
sheaves on $M \times \RR$ are defined by
\begin{align*}
F_1\Potimes F_2 &:= \rmR \mu_!(p_1^{-1}F_1\otimes p_2^{-1}F_2),\\
\Prhom(F_1, F_2) &:= \rmR p_{1\ast}\rhom(p_2^{-1}F_1, \mu^!F_2).
\end{align*}
We define the triangulated category of enhanced sheaves on $M$ by 
\[\BEC(\CC_M) := \BDC(\CC_{M\times\RR})/\pi^{-1}\BDC(\CC_M). \] 
Then the quotient functor
\[\Q : \BDC(\CC_{M \times \RR} )  \to\BEC(\CC_M)\]
has fully faithful left and right adjoints $\bfL^\rmE, \bfR^\rmE$ defined by 
\[\bfL^\rmE(\Q F) := (\CC_{\{t\geq0\}}\oplus\CC_{\{t\leq 0\}})
\Potimes F ,\hspace{20pt} \bfR^\rmE(\Q G) :
=\Prhom(\CC_{\{t\geq0\}}\oplus\CC_{\{t\leq 0\}}, G), \]
where $\{t\geq0\}$ stands for $\{(x, t)\in 
M\times\RR\ |\ t\geq0\}$ and $\{t\leq0\}$ is 
defined similarly. The convolution functors 
are defined also for enhanced sheaves. We denote them 
by the same symbols $\Potimes$, $\Prhom$. 
For a continuous map $f : M \to N $, we 
can define naturally the operations 
$\bfE f^{-1}$, $\bfE f_\ast$, $\bfE f^!$, $\bfE f_{!}$ 
for enhanced sheaves. 
We have also a 
natural embedding $\e : \BDC( \CC_M) \to \BEC( \CC_M)$ defined by 
\[ \e(F) := \Q(\CC_{\{t\geq0\}}\otimes\pi^{-1}F). \] 
For a continuous function $\varphi : U\to \RR$ 
defined on an open subset $U \subset M$ of $M$  we define 
the exponential enhanced sheaf by 
\[ {\rm E}_{U|M}^\varphi := 
\Q(\CC_{\{t+\varphi\geq0\}} ), \]
where $\{t+\varphi\geq0\}$ stands for 
$\{(x, t)\in M\times{\RR}\ |\ x\in U, t+\varphi(x)\geq0\}$.

\subsection{Enhanced Ind-sheaves}\label{sec:6}
We recall some basic notions 
and results on enhanced ind-sheaves. References are made to 
D'Agnolo-Kashiwara \cite{DK16} and 
Kashiwara-Schapira \cite{KS16}. 
Let $M$ be a good topological space.
Set $\RR_\infty := (\RR, \var{\RR})$ for 
$\var{\RR} := \RR\sqcup\{-\infty, +\infty\}$,
and let $t\in\RR$ be the affine coordinate. 
We consider the maps 
\[M\times\RR_\infty^2\xrightarrow{p_1, p_2, \mu}M
\times\RR_\infty\overset{\pi}{\longrightarrow}M\]
where $p_1, p_2$ and $\pi$ are morphisms
of bordered spaces induced by the projections.
And $\mu$ is a morphism of bordered spaces induced by the map 
$M\times\RR^2\ni(x, t_1, t_2)\mapsto(x, t_1+t_2)\in M\times\RR$.
Then the convolution functors for 
ind-sheaves on $M \times \RR_\infty$ are defined by
\begin{align*}
F_1\Potimes F_2 &:= \rmR\mu_{!!}(p_1^{-1}F_1\otimes p_2^{-1}F_2),\\
\Prihom(F_1, F_2) &:= \rmR p_{1\ast}\rihom(p_2^{-1}F_1, \mu^!F_2).
\end{align*}
Now we define the triangulated category 
of enhanced ind-sheaves on $M$ by 
\[\BEC(\I\CC_M) := \BDC(\I\CC_{M 
\times\RR_\infty})/\pi^{-1}\BDC(\I\CC_M).\]
Note that we have a natural embedding of categories
\[\BEC(\CC_M) \xhookrightarrow{\ \ \ }\BEC(\I\CC_M).\]
The quotient functor
\[\Q : \BDC(\I\CC_{M\times\RR_\infty})\to\BEC(\I\CC_M)\]
has fully faithful left and right adjoints $\bfL^\rmE,\bfR^\rmE$ defined by 
\[\bfL^\rmE(\Q K) := (\CC_{\{t\geq0\}}\oplus\CC_{\{t\leq 0\}})
\Potimes K ,\hspace{20pt} \bfR^\rmE(\Q K) :
=\Prihom(\CC_{\{t\geq0\}}\oplus\CC_{\{t\leq 0\}}, K), \]
where $\{t\geq0\}$ stands for 
$\{(x, t)\in M\times\var{\RR}\ |\ t\in\RR, t\geq0\}$ 
and $\{t\leq0\}$ is defined similarly.

The convolution functors 
are defined also for enhanced ind-sheaves. We denote them 
by the same symbols $\Potimes$, $\Prihom$. 
For a continuous map $f : M \to N $, we 
can define also the operations 
$\bfE f^{-1}$, $\bfE f_\ast$, $\bfE f^!$, $\bfE f_{!!}$ 
for enhanced ind-sheaves. For example, 
by the natural morphism $\tl{f}: M \times \RR_{\infty} \to 
N \times \RR_{\infty}$ of bordered spaces associated to 
$f$ we set $\bfE f_\ast ( \Q K)= \Q(\rmR \tl{f}_{\ast}(K))$. 
The other operations are defined similarly. 
We thus obtain the six operations $\Potimes$, $\Prihom$,
$\bfE f^{-1}$, $\bfE f_\ast$, $\bfE f^!$, $\bfE f_{!!}$ 
for enhanced ind-sheaves .
Moreover we denote by $\rmD_M^\rmE$ 
the Verdier duality functor for enhanced ind-sheaves.
We have outer hom functors
\begin{align*}
\rihom^\rmE(K_1, K_2) &:= \rmR\pi_\ast\rihom(\bfL^\rmE K_1, \bfL^\rmE K_2)
\simeq \rmR\pi_\ast\rihom(\bfL^\rmE K_1, \bfR^\rmE K_2),\\
\rhom^\rmE(K_1, K_2) &:= \alpha_M\rihom^\rmE(K_1, K_2),\\
\rHom^\rmE(K_1, K_2) &:=\rmR\Gamma(M; \rhom^\rmE(K_1, K_2)),
\end{align*}
with values in $\BDC(\I\CC_M), 
\BDC(\CC_M)$ and $\BDC(\CC)$, respectively. 
Moreover for $F\in\BDC(\I\CC_M)$ and $K\in\BEC(\I\CC_M)$ the objects 
\begin{align*}
\pi^{-1}F\otimes K &:=\Q(\pi^{-1}F\otimes \bfL^\rmE K),\\
\rihom(\pi^{-1}F, K) &:=\Q\big(\rihom(\pi^{-1}F, \bfR^\rmE K)\big). 
\end{align*}
in $\BEC(\I\CC_M)$ are well-defined. 
Set $\CC_M^\rmE := \Q 
\Bigl(``\underset{a\to +\infty}{\varinjlim}"\ \CC_{\{t\geq a\}}
\Bigr)\in\BEC(\I\CC_M)$. 
Then we have 
natural embeddings $\e, e : \BDC(\I\CC_M) \to \BEC(\I\CC_M)$ defined by 
\begin{align*}
\e(F) & := \Q(\CC_{\{t\geq0\}}\otimes\pi^{-1}F) \\ 
e(F) &:=  \CC_M^\rmE\otimes\pi^{-1}F
\simeq \CC_M^\rmE\Potimes\e(F). 
\end{align*}

For a continuous function $\varphi : U\to \RR$ 
defined on an open subset $U \subset M$ of $M$  we define 
the exponential enhanced ind-sheaf by 
\[\EE_{U|M}^\varphi := 
\CC_M^\rmE\Potimes {\rm E}_{U|M}^\varphi
=
\CC_M^\rmE\Potimes
\Q\CC_{\{t+\varphi\geq0\}}
=
 \Q 
\Bigl(``\underset{a\to +\infty}{\varinjlim}"\ \CC_{\{t+\varphi\geq a\}}
\Bigr)\]
where $\{t+\varphi\geq0\}$ stands for 
$\{(x, t)\in M\times\var{\RR}\ |\ t\in\RR, x\in U, t+\varphi(x)\geq0\}$.

\subsection{$\SD$-Modules}\label{sec:7}
In this subsection we recall some basic notions 
and results on $\SD$-modules. 
References are made to \cite{Bjo93}, 
\cite{HTT08}, \cite[\S 7]{KS01}, 
\cite[\S 8, 9]{DK16},
\cite[\S 3, 4, 7]{KS16} and 
\cite[\S 4, 5, 6, 7, 8]{Kas16}. 
For a complex manifold $X$ we denote by 
$d_X$ its complex dimension. 
Denote by $\SO_X, \Omega_X$ and $\SD_X$ 
the sheaves of holomorphic functions, 
holomorphic differential forms of top degree 
and holomorphic differential operators, respectively. 
Let $\BDC(\SD_X)$ be the bounded derived category 
of left $\SD_X$-modules and $\BDC(\SDop_X)$ 
be that of right $\SD_X$-modules. 
Moreover we denote by $\BDCcoh(\SD_X)$, $\BDC_{\rm good}(\SD_X)$,
$\BDChol(\SD_X)$ and $\BDCrh(\SD_X)$ the full triangulated subcategories
of $\BDC(\SD_X)$ consisting of objects with coherent, good, 
holonomic and regular holonomic cohomologies, 
respectively.
For a morphism $f : X\to Y$ of complex manifolds, 
denote by $\Dotimes, \rhom_{\SD_X}, \bfD f_\ast, \bfD f^\ast$ 
the standard operations for $\SD$-modules. 
We define also the duality functor $\DD_X : \BDCcoh(\SD_X)^{\op} 
\simto \BDCcoh(\SD_X)$ 
by 
\[\DD_X(\SM):=\rhom_{\SD_X}(\SM, \SD_X)
\uotimes{\SO_X}\Omega^{\otimes-1}_X[d_X].\]
Note that there exists 
an equivalence of categories 
$( \cdot )^{\rm r} : \Mod(\SD_X)\simto\Mod(\SDop_X)$ given by 
\[\SM^{\rm r}:=\Omega_X\uotimes{\SO_X}\SM.\]
The classical de Rham and solution functors are defined by  
\begin{align*}
DR_X &:  \BDCcoh (\SD_X)\to\BDC(\CC_X),
\hspace{40pt}\SM \longmapsto \Omega_X\Lotimes{\SD_X}\SM, \\
Sol_X &: \BDCcoh (\SD_X)^{\op}\to\BDC(\CC_X),
\hspace{30pt}\SM \longmapsto \rhom_{\SD_X}(\SM, \SO_X).
\end{align*}
Then for $\SM\in\BDCcoh(\SD_X)$ 
we have an isomorphism $Sol_X(\SM)[d_X]\simeq DR_X(\DD_X\SM)$. 
For a closed hypersurface $D\subset X$ 
in $X$ we denote by $\SO_X(\ast D)$ 
the sheaf of meromorphic functions on $X$ with poles in $D$. 
Then for $\SM\in\BDC(\SD_X)$ we set 
\[\SM(\ast D) := \SM\Dotimes\SO_X(\ast D).\]
For $f\in\SO_X(\ast D)$ and $U := X\bs D$, set 
\begin{align*}
\SD_Xe^f &:= \SD_X/\{P \in \SD_X  \ |\ Pe^f|_U = 0\}, \\
\SE_{U|X}^f &:= \SD_Xe^f(\ast D).
\end{align*}
Note that $\SE_{U|X}^f$ is holonomic and there exists an isomorphism 
\[\DD_X(\SE_{U|X}^f)(\ast D) \simeq \SE_{U|X}^{-f}.\]
Namely $\SE_{U|X}^f$ is a meromorphic connection associated 
to $d+df$. 

One defines the ind-sheaf $\SO_X^\rmt$ of tempered 
holomorphic functions 
as the Dolbeault complex with coefficients in the ind-sheaf 
of tempered distributions. 
More precisely, denoting by $\overline{X}$ the complex conjugate manifold 
to $X$ and by $X_{\RR}$ the underlying real analytic manifold of $X$,
we set
\[\SO_X^\rmt := \rihom_{\SD_{\overline{X}}}(
\SO_{\overline{X}}, \mathcal{D}b_{X_\RR}^\rmt), \]
where $\mathcal{D}b_{X_\RR}^\rmt$ is the ind-sheaf of 
tempered distributions on $X_\RR$
(for the definition see \cite[Definition 7.2.5]{KS01}). 
Moreover, we set 
 \[\Omega_X^\rmt := \beta_X\Omega_X\otimes_{\beta_X\SO_X}\SO_X^\rmt.\] 
Then the tempered de Rham and solution functors 
are defined by 
\begin{align*}
DR_X^\rmt &: \BDCcoh (\SD_X)\to\BDC(\I\CC_X),
\hspace{40pt} 
\SM \longmapsto \Omega_X^\rmt\Lotimes{\SD_X}\SM, 
\\
Sol_X^\rmt &: \BDCcoh (\SD_X)^{\op}\to\BDC(\I\CC_X), 
\hspace{40pt}
\SM \longmapsto \rihom_{\SD_X}(\SM, \SO_X^\rmt). 
\end{align*}
Note that we have isomorphisms
\begin{align*}
Sol_X(\SM) &\simeq \alpha_XSol_X^\rmt(\SM), \\
DR_X(\SM) &\simeq \alpha_XDR_X^\rmt(\SM), \\
Sol_X^\rmt(\SM)[d_X] &\simeq DR_X^\rmt(\DD_X\SM).
\end{align*}

Let $\PP$ be the one dimensional complex 
projective space $\PP^1$ and 
$i : X\times\RR_\infty\to X\times\PP$ 
the natural morphism of bordered spaces and
$\tau\in\CC\subset\PP$ the affine coordinate 
such that $\tau|_\RR$ is that of $\RR$. 
We then define objects $\SO_X^\rmE\in\BEC(\I\SD_X)$ and 
$\Omega_X^\rmE\in\BEC(\I\SD_X^{\op})$ by 
\begin{align*}
\SO_X^\rmE &:= \rihom_{\SD_{\overline{X}}}(
\SO_{\overline{X}}, \mathcal{D}b_{X_\RR}^{\T})\\
&\simeq i^!\bigl((\SE_{\CC|\PP}^{-\tau})^{\rm r} 
\Lotimes{\SD_\PP}\SO_{X\times\PP}^\rmt\bigr)[1]
\simeq i^!\rihom_{\SD_\PP}(\SE_{\CC|\PP}^{\tau}, 
\SO_{X\times\PP}^\rmt)[2],\\
\Omega_X^\rmE &:= \Omega_X\Lotimes{\SO_X}\SO_X^\rmE\simeq 
i^!(\Omega_{X\times\PP}^\rmt\Lotimes{\SD_\PP}
\SE_{\CC|\PP}^{-\tau})[1], 
\end{align*}
where $\mathcal{D}b_{X_\RR}^{\T}$ 
stand for the enhanced ind-sheaf 
of tempered distributions on $X_\RR$ 
(for the definition see \cite[Definition 8.1.1]{DK16}). 
We call $\SO_X^\rmE$ the enhanced ind-sheaf 
of tempered holomorphic functions. 
Note that there exists an isomorphism
\[i_0^!\bfR^\rmE\SO_X^\rmE\simeq\SO_X^\rmt, \]
where $i_0 : X\to X\times\RR_\infty$ is the 
inclusion map of bordered spaces induced by $x\mapsto (x, 0)$. 
The enhanced de Rham and solution functors 
are defined by 
\begin{align*}
DR_X^\rmE &: \BDCcoh (\SD_X)\to\BEC(\I\CC_X),
\hspace{40pt} 
\SM \longmapsto \Omega_X^\rmE\Lotimes{\SD_X}\SM,
\\
Sol_X^\rmE &: \BDCcoh (\SD_X)^{\op}\to\BEC(\I\CC_X), 
\hspace{40pt} 
\SM \longmapsto \rihom_{\SD_X}(\SM, \SO_X^\rmE). 
\end{align*}
Then for $\SM\in\BDCcoh(\SD_X)$ 
we have isomorphism $Sol_X^\rmE(\SM)[d_X]\simeq DR_X^\rmE(\DD_X\SM)$ 
and 
$Sol^{\rmt}_X(\M)\simeq
i_0^!{\bfR^\rmE}Sol_X^{\rmE}(\M).$ 
We recall the following results of \cite{DK16}.
\begin{theorem}
\begin{enumerate}\label{thm-4}
\item[\rm{(i)}] For $\SM\in\BDC_{\rm hol}(\SD_X)$ there
is an isomorphism in $\BEC(\I\CC_X)$
\[\rmD_X^\rmE\bigl(DR_X^\rmE(\SM)\bigr) \simeq Sol_X^\rmE(\SM)[d_X].\]

\item[\rm{(ii)}] Let $f : X\to Y$ be a morphism of complex manifolds.
Then for $\SN\in\BDC_{\rm hol}(\SD_Y)$ 
there is an isomorphism in $\BEC(\I\CC_X)$
\[Sol_X^\rmE({\bfD} f^\ast\SN) \simeq \bfE f^{-1}Sol_Y^\rmE(\SN).\]

\item[\rm{(iii)}] Let $f : X\to Y$ be a morphism of complex manifolds
and $\SM\in\BDC_{\rm good}(\SD_X)\cap\BDC_{\rm hol}(\SD_X)$.
If $\supp(\SM)$ is proper over Y then there 
is an isomorphism in $\BEC(\I\CC_Y)$
\[Sol_Y^\rmE({\bfD} f_\ast\SM)[d_Y] \simeq 
\bfE f_\ast Sol_X^\rmE(\SM )[d_X].\]

\item[\rm{(iv)}] For $\SM_1, \SM_2\in\BDC_{\rm hol}(\SD_X)$, 
there exists an isomorphism in $\BEC(\I\CC_X)$
\[Sol_X^\rmE(\SM_1\Dotimes\SM_2)\simeq Sol_X^\rmE(\SM_1)
\Potimes Sol_X^\rmE(\SM_2).\]

\item[\rm{(v)}] If $\SM\in\BDC_{\rm hol}(\SD_X)$ and $D\subset X$
is a closed hypersurface, then there 
are isomorphisms in $\BEC(\I\CC_X)$
\begin{align*}
Sol_X^\rmE(\SM(\ast D)) &\simeq \pi^{-1}
\CC_{X\bs D}\otimes Sol_X^\rmE(\SM),\\
DR_X^\rmE(\SM(\ast D)) &\simeq \rihom(
\pi^{-1}\CC_{X\bs D}, DR_X^\rmE(\SM)).
\end{align*}

\item[\rm{(vi)}] Let $D$ be a closed hypersurface in $X$ and 
$f\in\SO_X(\ast D)$ a meromorphic function along $D$.
Then there exists an isomorphism in $\BEC(\I\CC_X)$
\[Sol_X^\rmE \big( \SE_{X \setminus D | X}^\varphi\big) 
\simeq \EE_{X\setminus D | X}^{\Re\varphi}.\]
\end{enumerate}
\end{theorem}

Finally, we recall the following fundamental theorem of \cite{DK16}.

\begin{theorem}[{\cite[Theorem 9.5.3 (Irregular 
Riemann-Hilbert Correspondence)]{DK16}}]\label{cor-4}
The enhanced solution functor on a complex manifold $X$ 
\begin{align}
Sol_X^\rmE\colon\BDChol(\SD_X)^{\op}
\longrightarrow\rcEC^{\rmb}(\rmI\CC_X)
\end{align}
is fully faithful. 
\end{theorem}

\subsection{Constructible functions and 
constructible sheaves}\label{subsec:7}
In this subsection we recall some basic notions 
on constructible functions and their relationship 
with constructible sheaves. For a complex analytic 
space $X$ we denote by $\Dbc (\CC_X) \subset \BDC(\CC_X)$ 
the full triangulated subcategory of $\BDC(\CC_X)$ 
consisting of constructible objects. 
For an abelian group $G$ and a complex analytic space $X$,  
we shall say that a $G$-valued function 
$\varphi :X \longrightarrow G$ on $X$ is constructible if 
there exists a stratification $X=\bigsqcup_{\alpha}X_{\alpha}$ of $X$ 
such that $\varphi |_{X_{\alpha}}$ is 
constant for any $\alpha$. We denote by $\CF_G(X)$ 
the abelian group of $G$-valued constructible functions on $X$. 
In this paper, we consider $\CF_G(X)$ only for 
the additive group $G= \ZZ$. 
For a $G$-valued constructible function $\varphi \colon X \longrightarrow G$ on 
a complex analytic space $X$, by taking 
a stratification $X=\bigsqcup_{\alpha}X_{\alpha}$ of 
$X$ such that $\varphi |_{X_{\alpha}}$ is 
constant for any $\alpha$, we set 
\begin{equation}
\int_X \varphi :=\sum_{\alpha}\chi(X_{\alpha}) \cdot 
\varphi (x_{\alpha}) \quad \in G, 
\end{equation}
where $\chi ( \cdot )$ stands for the 
topological Euler characteristic and $x_{\alpha}$ is a reference point in 
the stratum $X_{\alpha}$. 
By the following lemma $\int_X \varphi \in G$ does not depend on the choice of 
the stratification $X=\bigsqcup_{\alpha} X_{\alpha}$ of $X$. 
We call it the topological (or Euler) integral of 
$\varphi$ over $X$. 

\begin{lemma}\label{lem:2-str-1}
Let $Y$ be a complex analytic space and 
$Y=\bigsqcup_{\alpha}Y_{\alpha}$ a stratification of $Y$. 
Then we have 
\begin{equation}
\chi_{{\rm c}}(Y) = \dsum_{\alpha} \chi_{{\rm c}}(Y_{\alpha}), 
\end{equation}
where $\chi_{{\rm c}}( \cdot )$ stands for the Euler 
characteristic with compact supports. Moreover, 
for any $\alpha$ we have 
$\chi_{{\rm c}}(Y_{\alpha})=\chi (Y_{\alpha})$. 
\end{lemma}

More generally, 
for any morphism $\rho \colon Z \longrightarrow W$ of complex analytic spaces 
and any $G$-valued constructible function 
$\varphi \in \CF_G(Z)$ on $Z$, we define the 
push-forward $\int_{\rho} \varphi \in \CF_G(W)$ of 
$\varphi$ by 
\begin{equation}
\Bigl( \int_{\rho} \varphi \Bigr) (w) :=\int_{\rho^{-1}(w)} \varphi \qquad 
(w \in W). 
\end{equation}
We thus obtain a homomorphism 
\begin{equation}
\int_{\rho} : \CF_G(Z) \longrightarrow \CF_G(W)
\end{equation}
of abelian groups. 
Let $X$ be a complex analytic space and $F \in \Dbc (\CC_X)$ 
a constructible object on it. Then it is easy to see that 
the $\ZZ$-valued function $\chi_X(F): X \longrightarrow \ZZ$ on 
$X$ defined by 
\begin{equation}
\chi_X(F)(x):= \dsum_{j \in \ZZ} (-1)^j \dim (H^jF)_x 
\qquad (x \in X) 
\end{equation}
is constructible. For a complex analytic space $X$ 
we define the Grothendieck group $\Kbc (\CC_X)$ of 
$\Dbc (\CC_X)$ to be the quotient of the free abelian group 
generated by the objects of $\Dbc (\CC_X)$ by the 
relations 
\begin{equation}
F=F^{\prime}+ F^{\prime \prime} \qquad 
(F^{\prime} \longrightarrow F \longrightarrow F^{\prime \prime}
 \overset{+1}{\longrightarrow} \ \text{is 
a distinguished triangle}). 
\end{equation}
Since for a distinguished triangle 
$F^{\prime} \longrightarrow F \longrightarrow F^{\prime \prime}
 \overset{+1}{\longrightarrow}$ 
in $\Dbc (\CC_X)$ we have 
$\chi_X(F)=\chi_X(F^{\prime})+ \chi_X(F^{\prime \prime})$, 
we thus obtain a group homomorphism 
\begin{equation}
\chi_X: \Kbc (\CC_X) \longrightarrow \CF_{\ZZ}(X). 
\end{equation}
The following lemma is well-known to the specialists 
(see e.g. Kashiwara-Schapira \cite[Theorem 9.7.1]{KS90}). 

\begin{lemma}\label{lem:comut-1}
Let $\rho \colon Z \longrightarrow W$ be a proper 
morphism of complex analytic spaces. 
Then the diagram 
\begin{equation}\label{comm-dia}
\begin{CD}
\Kbc (\CC_Z)   @>{\chi_Z}>> \CF_{\ZZ}(Z)
\\
@V{\Phi ( \rho )}VV   @VV{\int_{\rho}}V
\\
\Kbc (\CC_W) @>{\chi_W}>>  \CF_{\ZZ}(W)
\end{CD}
\end{equation}
commutes, where $\Phi ( \rho ) : \Kbc (\CC_Z) 
\longrightarrow \Kbc (\CC_W)$ is the morphism 
induced by the direct image functor 
$\rmR \rho_\ast: \Dbc (\CC_Z) 
\longrightarrow \Dbc (\CC_W)$. 
\end{lemma}

\section{Solution complexes to irregular holonomic D-modules 
with a quasi-normal form}\label{sec:qNor}

In this section, we study the enhanced and usual solution complexes to 
irregular holonomic D-modules with a quasi-normal form 
in the sense of Mochizuki \cite[Chapter 5]{Moc11}. First of all, 
as a prototype of our study, we have the following result. 

\begin{proposition}\label{prop-1}
Let $X \subset \CC^n$ be a neighborhood of the origin 
$0 \in \CC^n$ and $z=(z_1, \ldots, z_n)$ the standard coordinate of 
$X \subset \CC^n$ and for $1 \leq l \leq n$ set 
$D:= \{ z \in X \subset \CC^n \ | \ z_1 \cdots z_l=0 \} 
\subset X$ and $U:= X \setminus D$. For a 
holomorphic function $c:X \longrightarrow \CC$ on $X$ 
such that $c(z) \not= 0$ for any $z \in X$ and non-negative 
integers $k_1, k_2, \ldots, k_l \in \ZZ_{\geq 0}$ we 
define a meromorphic function $\varphi$ on $X$ having 
a pole along the normal crossing divisor $D \subset X$ by 
\begin{equation} 
\varphi (z):= \frac{c(z)}{z_1^{k_1} z_2^{k_2} \cdots z_l^{k_l}} 
\qquad (z \in U=X \setminus D). 
\end{equation} 
For $a\gg0$ we set
\begin{equation} 
U_a :=  \{z \in U = X \setminus D  \ |\  
{\Re}(\varphi(z)) < a\}  \ \subset X \subset \CC^n.
\end{equation} 
Then we have the following results. 
\begin{enumerate}
\item[\rm{(i)}] 
Assume that $l=1$. Then we have isomorphisms
\begin{equation} 
H^jSol_X^\rmt\big(\SE_{U | X}^{\varphi}\big)
\simeq
\begin{cases}
``\underset{a\to +\infty}{\varinjlim}"\ \CC_{U_a} & (j=0)\\
\\
\hspace{30pt} \CC_{D}^{\oplus k_1} & (j=1)\\
\\
\hspace{35pt}0 & (\mbox{\rm otherwise}).
\end{cases}
\end{equation} 
\item[\rm{(ii)}] 
Assume that $l \geq 2$. Then we have an isomorphism
\begin{equation} 
H^0 Sol_X^\rmt\big(\SE_{U | X}^{\varphi}\big)
\simeq
``\underset{a\to +\infty}{\varinjlim}"\ \CC_{U_a} 
\end{equation} 
and $H^j Sol_X^\rmt (\SE_{U | X}^{\varphi})$ 
($j \geq 1$) are usual sheaves supported by 
the normal crossing divisor $D \subset X$. Moreover, 
they are constructible with respect to the standard 
stratification of $D$ and we have isomorphisms 
\begin{equation} 
H^jSol_X^\rmt\big(\SE_{U | X}^{\varphi}\big)_0
\simeq
\begin{cases}
\CC^{d \cdot \binom{l-1}{j-1}} & (1 \leq j \leq l)\\
\\
0 & (\mbox{\rm otherwise}), 
\end{cases}
\end{equation}
where we set 
\begin{equation} 
d:= 
\begin{cases}
{\rm gcd}(k_1,k_2, \ldots, k_l) & (k_1 k_2 \cdots k_l \not= 0)\\
\\
\hspace{35pt}0 & (\mbox{\rm otherwise}).
\end{cases}
\end{equation} 
\end{enumerate}
\end{proposition}

\begin{proof}
Since (i) is well-known and follows immediately also 
from (the proof of) \cite[Proposition 3.14]{IT20}, 
we prove only (ii). The proof of (ii) is 
similar to that of \cite[Proposition 3.14]{IT20}, 
but we need more careful attention. We prove 
only the assertions in (ii) on a neighborhood of the 
origin $0 \in X \subset \CC^n$, as we will see that 
the other parts can be proved similarly. Then, 
after a suitable change of coordinates, shrinking 
$X$ if necessary, we may assume also that the 
holomorphic function $c(z)$ is a non-zero 
constant and denote it by $c \in \CC^*:= \CC 
\setminus \{ 0 \}$ for simplicity. 
We set $\SN:=\SE_{U | X}^{\varphi}\in\Modhol(\SD_{X})$.
Recall that by (the proof of) 
\cite[Lemma 3.13]{IT20} we have isomorphisms
\begin{equation}
Sol_X^\rmt(\SN) \simeq
i_0^!\bfR^\rmE Sol_X^\rmE(\SN)
\simeq
\rmR\pi_\ast\rihom
(\CC_{\{t\geq0\}}\oplus\CC_{\{t\leq0\}}, Sol_X^\rmE(\SN)).
\end{equation}
Moreover by Theorem \ref{thm-4} (vi) 
there exists also an isomorphism
\begin{equation}
Sol_X^\rmE(\SN)
\simeq
``\underset{a\to +\infty}{\varinjlim}"
\ \CC_{\{t\geq-{\Re}\varphi+a\}}.
\end{equation}
For any sufficiently large $a\gg0$ we can easily show that
\begin{equation}
\rhom(\CC_{\{t\geq0\}}, \CC_{\{t\geq-{\Re}\varphi+a\}})
\simeq
\rmR\Gamma_{\{t\geq0\}}\CC_{\{t\geq-{\Re}\varphi+a\}}
\simeq
 \CC_{\{t>0,\ t\geq-{\Re}\varphi+a\}}.
\end{equation}
Similarly for $a\gg0$ we have
\begin{equation}
\rhom(\CC_{\{t \leq0 \}}, \CC_{\{t\geq-{\Re}\varphi+a\}})
\simeq
 \CC_{\{t<0,\ t\geq-{\Re}\varphi+a\}}. 
\end{equation}
Let $\var{\pi} : X\times\var{\RR}\to X$ be the projection
and $j : X\times \RR\xhookrightarrow{\ \ \ }
 X\times \var{\RR}$ the inclusion map. 
Then for $a\gg0$ it is easy to see that 
\begin{align*}
& \rmR\pi_\ast\rihom
( \CC_{\{t\leq0\}},  \CC_{\{t\geq-{\Re}\varphi+a\}} )
\\
&\simeq
\rmR\var{\pi}_\ast\rhom_{\CC_{X\times\var{\RR}}}(
\CC_{X\times \RR}, \CC_{\{t<0,\ t\geq-{\Re}\varphi+a\}})
\\
&\simeq
\rmR\var{\pi}_\ast \rmR j_\ast 
\CC_{\{t<0,\ t\geq-{\Re}\varphi+a\}} \simeq 0. 
\end{align*}
We thus obtain an isomorphism 
\begin{equation}
 \rmR\pi_\ast\rihom
( \CC_{\{t\leq0\}},  Sol_X^\rmE(\SN) ) \simeq 0. 
\end{equation}
For $a\gg0$ let us calculate 
\begin{equation}
\rmR\pi_\ast\rihom
( \CC_{\{t\geq0\}},  \CC_{\{t\geq-{\Re}\varphi+a\}} ) 
\simeq \rmR\var{\pi}_\ast \rmR j_\ast 
\CC_{\{t>0,\ t\geq-{\Re}\varphi+a\}}. 
\end{equation}
The stalk of this complex at a point $x \in X$ 
is isomorphic to $\CC$ (resp. $0$) if 
$x \in U_a$ (resp. $x \in X \setminus U_a$). 
Moreover the stalk at the origin $0 \subset X= \CC^n$ is isomorphic to 
\begin{equation}
\rmR\Gamma( \{ 0 \} \times\var{\RR}; 
\rmR j_\ast\CC_{\{t>0,\ t\geq-{\Re}\varphi+a\}})
\simeq 
( \rmR j_\ast\CC_{\{t>0,\ t\geq-{\Re}\varphi+a\}})_{(0, + \infty )}. 
\end{equation}
Then it suffices to calculate the cohomology groups of 
the complex $( \rmR j_\ast\CC_{\{t>0,\ t\geq-{\Re}\varphi+a\}})_{(0, + \infty )}$. 
For a sufficiently small $0< \e \ll 1$ 
we set $D(0; \e ):= \{ \tau \in \CC \ | \ | \tau | \leq \e \}$, 
$D(0; \e )^*:= \{ \tau \in \CC \ | \ 0<| \tau | \leq \e \}$ and 
\begin{equation}
K^*:= \big( D(0; \e )^* \big)^l \times \big( D(0; \e ) \big)^{n-l}
\ \subset \ K:= \big( D(0; \e ) \big)^n \qquad \subset X \subset \CC^n. 
\end{equation}
For a sufficiently large $b \gg 0$ 
we set also 
\begin{equation}
K^*_b:= \{ z \in K^* \ | \ {\Re}(\varphi(z))\geq -b \}
\qquad \subset K^*
\end{equation}
and 
\begin{equation}
L^*_{a+b}:= \{ (z,t) \in X \times \overline{\RR} \ | \ 
b+a \leq t < + \infty, z \in K^*, t \geq - {\Re}(\varphi(z))+a \} 
\ \subset X \times \overline{\RR}
\end{equation}
so that under the natural identification 
$\{ t= a+b \} := X \times \{ a+b \} \simeq X$ we have 
\begin{equation}
L^*_{a+b} \cap \{ t= a+b \} \simeq K^*_b. 
\end{equation}
For $0< \e \ll 1$ and $b \gg 0$ there exists an isomorphism 
\begin{equation}
( \rmR j_\ast\CC_{\{t>0,\ t\geq-{\Re}\varphi+a\}})_{(0, + \infty )}
\simeq \rmR \Gamma ( X \times \RR ; \CC_{L^*_{a+b}}) 
\end{equation}
and $L^*_{a+b}$ is homotopic to 
$L^*_{a+b} \cap \{ t= a+b \} \simeq K^*_b$. 
We fix such $0< \e \ll 1$ and $b \gg 0$ once and for all. 
Then our remaining task is to calculate the cohomology groups of 
the complex $\rmR \Gamma ( X  ; \CC_{K^*_b})$. 
For this purpose, let $\varpi: \widetilde{X} \longrightarrow X$ 
be the real oriented blow-up of $X$ along the normal 
crossing divisor $D \subset X$ and by the isomorphism 
$\varpi^{-1}(U) \simto U$ induced by $\varpi$ regard 
$K^*_b \subset U= X \setminus D$ as a locally 
closed subset of $\widetilde{X}$. We denote by 
$\overline{K^*_b} \subset \widetilde{X}$ the closure of 
$K^*_b$ in $\widetilde{X}$ and set 
\begin{equation}
T^*_b:= \overline{K^*_b} \setminus K^*_b = 
\overline{K^*_b} \cap \varpi^{-1}(D). 
\end{equation}
Then we obtain an exact sequence 
\begin{equation}\label{exa-seq}
0 \longrightarrow \CC_{K^*_b} \longrightarrow 
\CC_{\overline{K^*_b}} \longrightarrow 
\CC_{T^*_b} \longrightarrow 0 
\end{equation}
of $\RR$-constructible sheaves on $\widetilde{X}$. 
First, let us consider the case where 
$k_1 k_2 \cdots k_l=0$ i.e. there exist $1 \leq i \leq l$ 
such that $k_i=0$. Then for 
the subanalytic subset $(K^*_b)_{{\rm red}}:= 
K_b^* \cap \{ z_i= \delta \}$ ($0< \delta < \e$) we see 
that $K^*_b$ is isomorphic to 
$(K^*_b)_{{\rm red}} \times (S^1 \times (0,\e ])$. 
Since we have $\rmR \Gamma ( \RR  ; \CC_{(0,\e ]})
\simeq 0$, by the K{\"u}nneth formula we thus 
obtain the desired vanishing 
\begin{equation}
 \rmR \Gamma ( X  ; \CC_{K_b^*}) \simeq 0. 
\end{equation}
So, we may assume that $k_1 k_2 \cdots k_l \not= 0$ i.e. 
for any $1 \leq i \leq l$ we have $k_i>0$. 
Take $R>0$ and $\lambda \in \RR$ such that 
$c=R e^{\sqrt{-1} \lambda }$ and for each $1 \leq i \leq l$ 
set $z_i=r_i e^{\sqrt{-1} \theta_i}$ 
($r_i >0, \theta_i \in \RR$). Then for $z \in K^* 
\subset U= X \setminus D$ we have 
\begin{equation}
 {\Re}(\varphi(z)) = 
R r_1^{-k_1} \cdots r_l^{-k_l} 
\cos ( \lambda + k_1 \theta_1 + \cdots + k_l \theta_l). 
\end{equation}
This implies that the condition ${\Re}(\varphi(z))\geq -b$ 
for $z=(r_1 e^{\sqrt{-1} \theta_1}, \ldots, 
r_l e^{\sqrt{-1} \theta_l}, z_{l+1}, \ldots, z_n) \in 
K^*$ is equivalent to the one 
\begin{equation}
r_1^{k_1} \cdots r_l^{k_l} \geq - \frac{R}{b} 
\cos ( \lambda + k_1 \theta_1 + \cdots + k_l \theta_l). 
\end{equation}
Note that this condition is satisfied 
for any $r_1, \ldots, r_l \geq 0$ if and only 
if $\cos ( \lambda + k_1 \theta_1 + \cdots + k_l \theta_l) 
\geq 0$. In particular, for the complex submanifold 
$Y:= \{ z \in X= \CC^n \ | \ z_1= \cdots =z_l=0 \} 
\subset X= \CC^n$ such that $Y \subset D$ and the 
closed subset 
\begin{equation}
 \varpi^{-1}(Y) \ ( \simeq (S^1)^l \times Y) 
= 
\{ (e^{\sqrt{-1} \theta_1}, \ldots, e^{\sqrt{-1} \theta_l}, z_{l+1}, \ldots, z_n) 
 \ | \ 
\theta_i \in \RR, z_i \in \CC \} 
\end{equation}
of $\varpi^{-1}(D) \subset \widetilde{X}$ we have 
\begin{align*}
& \overline{K^*_b} \cap \varpi^{-1}(Y) 
\\ 
& = 
\{ (e^{\sqrt{-1} \theta_1}, \ldots, e^{\sqrt{-1} \theta_l}, z_{l+1}, \ldots, z_n) 
 \ | \ 
\theta_i \in \RR, |z_i| \leq \e, 
\cos ( \lambda + k_1 \theta_1 + \cdots + k_l \theta_l) \geq 0 \}.
\end{align*}
Moreover, we can easily see that $T^*_b= 
\overline{K^*_b} \cap \varpi^{-1}(D)$ is 
homotopic to $\overline{K^*_b} \cap \varpi^{-1}(0) 
\subset \varpi^{-1}(0) \simeq (S^1)^l$. 
On the other hand, we see also that the closed 
subset $\overline{K^*_b} \subset \widetilde{X}$ is 
homotopic to $\varpi^{-1}(Y)$ and hence to 
$\varpi^{-1}(0) \simeq (S^1)^l$. By the exact 
sequence \eqref{exa-seq} and Proposition \ref{prop-appA}, we thus 
obtain the desired isomorphisms 
\begin{equation} 
H^j(X; \CC_{K_b^*} )
\simeq
\begin{cases}
\CC^{d \cdot \binom{l-1}{j-1}} & (1 \leq j \leq l)\\
\\
0 & (\mbox{\rm otherwise}).
\end{cases}
\end{equation}
This completes the proof.
\end{proof}

\begin{remark}
Let $\SS$ be the standard stratification of the 
normal crossing divisor $D \subset X$ in Proposition 
\ref{prop-1}. Then, although the holomorphic 
solutions to $\SE_{U | X}^{\varphi}$ on 
$U=X \setminus D \subset X$ are 
constant multiples of the single-valued function 
$e^{\varphi (z)}$ and hence $Sol_X (\SE_{U | X}^{\varphi})|_U$
is a constant sheaf on $U$, by the proof of Proposition 
\ref{prop-1} we see that for some strata $S \in \SS$ in 
$\SS$ and $j \in \ZZ$ the local systems 
$H^jSol_X (\SE_{U | X}^{\varphi})|_S$ on $S$ may not be 
constant i.e. have some non-trivial monodromies. 
\end{remark} 

Note that some special cases of Proposition \ref{prop-1} were 
obtained in Kashiwara-Schapira 
\cite[Proposition 7.3 and Remark 7,4]{KS03} 
and Ito-Takeuchi \cite[Proposition 3.14]{IT20}. 
By Proposition \ref{prop-1} and the isomorphism 
$\alpha_X Sol_X^\rmt 
(\SE_{U | X}^{\varphi}) 
\simeq Sol_X (\SE_{U | X}^{\varphi})$ we obtain the 
following corollary. 

\begin{corollary}\label{cor-1}
In the situation of Proposition \ref{prop-1} 
we have the following results. 
\begin{enumerate}
\item[\rm{(i)}] 
If $l=1$ we have 
\begin{equation} 
\dim H^jSol_X \big(\SE_{U | X}^{\varphi}\big)_0
=
\begin{cases}
k_1 & (j=1)\\
\\
0 & (\mbox{\rm otherwise})
\end{cases}
\end{equation} 
and hence 
\begin{equation} 
\chi \bigl(
Sol_X \big(\SE_{U | X}^{\varphi}\big)
\bigr)(0)= -k_1. 
\end{equation} 
\item[\rm{(ii)}] 
If $l \geq 2$ we have 
\begin{equation} 
\dim H^jSol_X \big(\SE_{U | X}^{\varphi}\big)_0
=
\begin{cases}
d \cdot \binom{l-1}{j-1} & (1 \leq j \leq l)\\
\\
0 & (\mbox{\rm otherwise})
\end{cases}
\end{equation} 
and hence 
\begin{equation} 
\chi \bigl(
Sol_X \big(\SE_{U | X}^{\varphi}\big)
\bigr)(0)=0. 
\end{equation} 
\end{enumerate}
\end{corollary}

From now, we shall extend Corollary \ref{cor-1} to 
holonomic D-modules with a quasi-normal form 
in the sense of Mochizuki \cite[Chapter 5]{Moc11}. 
First of all, we recall some notions and results in \cite[\S 7]{DK16}. 
Let $X$ be a complex manifold and $D \subset X$ a 
normal crossing divisor in it. Denote by 
$\varpi_X : \tl{X}\to X$ the real oriented blow-up of $X$ along 
$D$ (sometimes we denote it simply by $\varpi$). 
Then we set
\begin{align*}
\SO_{\tl{X}}^\rmt &:= \rhom_{\varpi^{-1}
\SD_{\overline{X}}}(\varpi^{-1}\SO_{\overline{X}}, 
\mathcal{D}b_{\tl{X_\RR}}^\rmt),\\
\SA_{\tl{X}} &:= \alpha_{\tl{X}}\SO_{\tl{X}}^\rmt,\\
\SD_{\tl{X}}^\SA &:= \SA_{\tl{X}}\otimes_{\varpi^{-1}\SO_X}
\varpi^{-1}\SD_X, 
\end{align*}
where $\mathcal{D}b_{\tl{X}}^\rmt$ stands for the ind-sheaf 
of tempered distributions on $\tl{X}$ 
(for the definition see \cite[Notation 7.2.4]{DK16}). 
Recall that a section of $\SA_{\tl X}$ is a holomorphic function having
moderate growth at $\varpi_X^{-1}(D)$.
Note that $\SA_{\tl{X}}$ and 
$\SD_{\tl{X}}^\SA$ are sheaves of rings on $\tl{X}$. 
For $\SM\in\BDC(\SD_X)$ we define 
an object $\SM^{\SA}\in\BDC(\SD_{\tl{X}}^\SA)$ by 
\begin{align*}
\SM^{\SA} &:= \SD_{\tl{X}}^\SA\Lotimes
{\varpi^{-1}\SD_X}\varpi^{-1}\SM
\simeq\SA_{\tl{X}}
\Lotimes{\varpi^{-1}\SO_X}\varpi^{-1}\SM.
\end{align*}
Note that if $\SM$ is a holonomic $\SD_X$-module such that 
$\SM\simto\SM(\ast D)$ and 
${\rm sing.supp} (\SM)\subset D$,
then one has $\SM^\SA\simeq
\SD_{\tl{X}}^\SA\otimes_{\varpi^{-1}\SD_X}\varpi^{-1}\SM$
(see \cite[Lemma 7.3.2]{DK16}). 
Let us take local coordinates 
$(u,v)=(u_1, \ldots, u_l, v_1, \ldots, v_{n-l})$ 
of $X$ such that $D= \{ u_1 u_2 \cdots u_l=0 \}$. 
We define a partial order $\leq$ on the 
set $\ZZ^l$ by 
\[ a=(a_1, \ldots, a_l) \leq a^{\prime}=(a^{\prime}_1, \ldots, a^{\prime}_l) 
 \ \Longleftrightarrow 
\ a_i \leq a_i^{\prime} \ (1 \leq i \leq l).\] 
Then for a meromorphic function $\varphi\in\SO_X(\ast D)$
on $X$ having a pole along $D$ by using its Laurent expansion 
\[ \varphi = \sum_{a \in \ZZ^l} c_a( \varphi )(v) \cdot 
u^a \ \in \SO_X(\ast D) \]
with respect to $u_1, \ldots, u_{l}$
we define its order 
${\rm ord}( \varphi ) \in \ZZ^l$ to be 
the minimum 
\[ \min \Big( \{ a \in \ZZ^l \ | \
c_a( \varphi ) \not= 0 \} \cup \{ 0 \} \Big) \]
if it exists. In \cite[Chapter 5]{Moc11} 
Mochizuki defined the notion of 
good sets of irregular values on $(X,D)$ to be 
finite subsets $S \subset \SO_X(\ast D)/ \SO_X$ such that 
\begin{enumerate}[noitemsep]
\item [\rm (i)] ${\rm ord}( \varphi )$ exists for any 
$\varphi \in S$ and if $\varphi \not= 0$ then its  
leading term $c_{{\rm ord}( \varphi )} ( \varphi ) (v)$ 
does not vanish 
at any point $v \in Y:= \{ u_1= \cdots =u_l=0 \} 
\subset D$. 
\item [\rm (ii)] ${\rm ord}( \varphi - \psi )$ exists for any 
$\varphi \not= \psi$ in $S$ and then ${\rm ord}( \varphi - \psi ) 
\in \ZZ^l_{\leq 0} \setminus \{ 0 \}$ and the  
leading term $c_{{\rm ord}( \varphi - \psi )} ( \varphi - \psi ) (v)$ 
does not vanish 
at any point $v \in Y= \{ u_1= \cdots =u_l=0 \} 
\subset D$. 
\item [\rm (iii)] the subset $\{ {\rm ord}( \varphi - \psi ) \ | \ 
\varphi, \psi \in S, \varphi \not= \psi \} \subset \ZZ^l$  
is totally ordered with respect to the order $\leq$ on $\ZZ^l$. 
\end{enumerate}
\begin{definition}\label{def-A}
Let $X$ be a complex manifold and $D \subset X$ a 
normal crossing divisor in it. 
Then we say that a holonomic $\SD_X$-module 
$\SM$ has a normal form along $D$ if
\begin{enumerate}[noitemsep]
\item [\rm (i)] $\SM\simto\SM(\ast D)$
\item [\rm (ii)] $\rm{sing.supp}(\SM)\subset D$
\item [\rm (iii)] for any $\theta\in\varpi^{-1}(D)\subset\tl{X}$,
there exist an open neighborhood $U\subset X$
of $\varpi({\theta}) \in D$ in $X$, a good set 
$S= \{ [ \varphi_1], [ \varphi_2], \ldots, 
[ \varphi_k] \} \subset \SO_X(\ast D)/ \SO_X$ 
($\varphi_i \in \SO_X(\ast D)$) 
of irregular values on $(U, D \cap U)$, 
positive integers 
$m_i >0$ ($1 \leq i \leq k$) and 
an open neighborhood $W$ of $\theta$ 
with $W \subset\varpi^{-1}(U)$
such that
\begin{equation}
\SM^\SA|_W
\simeq
\bigoplus_{i=1}^k 
\Bigl( \bigl(\SE_{U\bs D|U}^{\varphi_i}\bigr)^\SA |_W
\Bigr)^{\oplus m_i} .
\end{equation}   
\end{enumerate}
\end{definition}

By \cite[Proposition 3.19]{IT20} the good set 
$S \subset \SO_X(\ast D)/ \SO_X$ of 
irregular values for $\SM$ 
in this definition does not depend on 
the point $\theta \in \varpi^{-1}(D)$. Moreover 
by \cite[Proposition 3.5]{IT20} for any 
$\theta \in \varpi^{-1}(D \cap U)$ there exists its  
sectorial open neighborhood $V \subset U \setminus D$ 
such that 
\begin{equation}
\pi^{-1} \CC_V \otimes 
Sol_X^\rmE ( \SM ) 
\simeq
\bigoplus_{i=1}^k 
\Bigl(  \EE_{V|X}^{\Re \varphi_i} 
\Bigr)^{\oplus m_i}.
\end{equation}

A ramification of $X$ along the normal crossing divisor 
$D \subset X$ on a neighborhood $U$ 
of $x \in D$ is a finite map $ \rho : X' \longrightarrow U$
of complex manifolds 
of the form $w \longmapsto 
z=(z_1,z_2, \ldots, z_n)= 
 \rho (w) = (w^{d_1}_1,\ldots, w^{d_l}_l, w_{l+1},\ldots, w_n)$ 
for some $(d_1, \ldots, d_l)\in (\ZZ_{>0})^l$, where 
$(w_1,\ldots, w_n)$ is a local coordinate system of $X'$ and 
$(z_1, \ldots, z_n)$ is that of 
$U$ such that $D \cap U=\{z_1\cdots z_l=0\}$. 

\begin{definition}\label{def-B}
Let $X$ be a complex manifold and $D \subset X$ a 
normal crossing divisor in it. 
Then we say that a holonomic $\SD_X$-module $\SM$ has a 
quasi-normal form along $D$ if it satisfies 
the conditions (i) and (ii) of Definition \ref{def-A},
and if for any point $x \in D$ 
there exists a ramification $\rho  : X'\to U$ 
on a neighborhood $U$ of it such that $\bfD \rho^\ast(\SM|_U)$
has a normal form along the normal crossing divisor 
$\rho^{-1}(D\cap U)$.
\end{definition}

Note that $\bfD \rho^\ast(\SM|_U)$ as well 
as $\bfD \rho_\ast\bfD \rho^\ast(\SM|_U)$
is concentrated in degree zero and $\SM|_U$ is a 
direct summand of $\bfD \rho_\ast\bfD \rho^\ast(\SM|_U)$. 
Now let $\SM$ be a holonomic $\SD_X$-module having a 
quasi-normal form along the normal crossing divisor 
$D \subset X$. Then for any point $x \in D$ 
there exists a ramification $\rho  : X'\to U$ 
on a neighborhood $U$ of it such that $\bfD \rho^\ast(\SM|_U)$
has a normal form along the normal crossing divisor 
$D^{\prime} :=\rho^{-1}(D\cap U) \subset X^{\prime}$. 
Note that $\rho^{-1}(x) \subset D^{\prime}$ is a point and 
denote it by $x^{\prime}$. 
Let $\varpi^{\prime}: 
\widetilde{X^{\prime}} \to X^{\prime}$ be the real 
oriented blow-up of $X^{\prime}$ along 
$D^{\prime}$ and $\tl{\rho}: \widetilde{X^{\prime}} \to 
\tl{X}$ the morphism induced by $\rho$. 
Then by \cite[Propositions 3.5 and 3.19]{IT20} 
there exist a unique good set 
$S= \{ [ \varphi_1], [ \varphi_2], \ldots, 
[ \varphi_k] \} \subset \SO_{X^{\prime}}(\ast D^{\prime})
/ \SO_{X^{\prime}}$ 
($\varphi_i \in  \SO_{X^{\prime}}(\ast D^{\prime})$) 
of irregular values on a neighborhood of 
$x^{\prime} \in D^{\prime}$ in $X^{\prime}$ and 
positive integers 
$m_i >0$ ($1 \leq i \leq k$) such that 
for any $\theta^{\prime} \in ( \varpi^{\prime})^{-1}
(D^{\prime})$ and 
its sufficiently small sectorial open neighborhood 
$V^{\prime} \subset X^{\prime} \setminus D^{\prime}$ 
we have an isomorphism 
\begin{equation}
\pi^{-1} \CC_{V^{\prime}} \otimes 
Sol_{X^{\prime}}^\rmE ( \bfD \rho^\ast(\SM|_U) ) 
\simeq
\bigoplus_{i=1}^k 
\Bigl(  \EE_{V^{\prime}|X^{\prime}}^{\Re \varphi_i} 
\Bigr)^{\oplus m_i}.
\end{equation}
For a point $\theta\in\varpi^{-1}(D \cap U)$ and 
its sufficiently small sectorial open neighborhood 
$V \subset U \setminus D$ we take a point 
$\theta^{\prime} \in ( \varpi^{\prime})^{-1}
(D^{\prime})$ such that $\tl{\rho} ( \theta^{\prime}) = 
\theta$ and its sectorial open neighborhood 
$V^{\prime} \subset X^{\prime} \setminus D^{\prime}$ 
such that $\rho |_{V^{\prime}} : V^{\prime} \simto V$. 
Define holomorphic functions 
$f_i: V \to \CC$ ($1 \leq i \leq k$) by 
$f_i:= \varphi_i \circ ( \rho |_{V^{\prime}})^{-1}$. 
Then by \cite[Proposition 3.5]{IT20} we obtain 
an isomorphism 
\begin{equation}\label{mon-eq} 
\pi^{-1} \CC_V \otimes 
Sol_X^\rmE ( \SM ) 
\simeq
\bigoplus_{i=1}^k 
\Bigl(  \EE_{V|X}^{\Re f_i} 
\Bigr)^{\oplus m_i}.
\end{equation}
As $\tl{\rho}: \widetilde{X^{\prime}} \to 
\tl{X}$ is locally an isomorphism, then it is 
also clear that on an open neighborhood $W$ of $\theta$ 
in $\tl{X}$ we have an isomorphism 
\begin{equation}\label{mono-eqn} 
\SM^\SA|_W
\simeq
\bigoplus_{i=1}^k 
\Bigl( \bigl(\SE_{U\bs D|U}^{f_i}\bigr)^\SA |_W
\Bigr)^{\oplus m_i} .
\end{equation}
Moreover, in \cite[Lemma 7.4]{Tak22} we proved the 
following result. 

\begin{lemma}\label{DK-lem-2}
In the situation as above, 
there exists a  
sectorial open neighborhood $V \subset U \setminus D$ 
of $\theta \in \varpi^{-1}(D \cap U)$ 
such that for any $1 \leq i,j \leq k$ the natural morphism 
\begin{equation}
{\rm Hom}^\rmE ({\rm E}_{V|M}^{\Re f_i}, 
{\rm E}_{V|M}^{\Re f_j}) 
\longrightarrow 
{\rm Hom}^\rmE ( \EE_{V|X}^{\Re f_i}, \EE_{V|X}^{\Re f_j}) 
\end{equation}
is an isomorphism. 
\end{lemma}

In order to improve \eqref{mon-eq} and obtain a 
higher-dimensional analogue of \cite[Proposition 5.4.5]{DK18}, 
let us prepare some notations (see \cite[Section 5]{DK18} 
for the details in the one dimensional case).  For 
the real oriented blow-up 
$\varpi : \tl{X}\to X$ of $X$ along the normal crossing divisor 
$D \subset X$ consider the following commutative diagram
\begin{equation}
\vcenter{
\xymatrix@M=5pt{
\varpi^{-1}(D) \ar[r]^-{\tl{\imath}} & \tl{X} \ar[d]^-{\varpi}& \\
X\bs D \ar[r]^-{j} \ar[ur]^-{\tl{\jmath}} & X, 
}}\end{equation}
where $\tl{\imath},\tl{\jmath},j$ 
are the natural embeddings. For an open subset $\Omega\subset\tl{X}$, 
$f\in \Gamma(\Omega;\tl{\jmath}_{\ast} 
j^{-1}\SO_X) \simeq \Gamma( \tl{\jmath}^{-1} (\Omega);
\SO_{X \setminus D})$ and $\theta \in \Omega \cap \varpi^{-1}(D)$ 
we say that $f$ admits a Puiseux expansion 
along $D \subset X$ at $\theta$ if there exist  
a ramification $\rho  : X'\to U$ 
of a neighborhood $U$ of $\varpi ( \theta ) \in D$ 
along $D \cap U \subset U$, a sectorial 
neighborhood $V \subset U \setminus D$ of $\theta$ 
contained in $\tl{\jmath}^{-1} (\Omega) 
= \varpi ( \Omega \setminus \varpi^{-1}(D)) 
\subset X \setminus D$ and a 
meromorphic function $g \in \SO_{X^{\prime}}(* D^{\prime})$ 
along the normal crossing divisor $ D^{\prime}:= \rho^{-1}(D) 
\subset X^{\prime}$ defined on an open neighborhood 
$W$ of $\rho^{-1}( \overline{V} \cap D)= 
\overline{\rho^{-1}(V)} \cap D^{\prime}$ in $X^{\prime}$ such that 
the pull-back of $f|_V \in \SO_X(V)$ by $\rho$ conicides 
with $g$ on the open subset $W \cap \rho^{-1}(V) \subset W$. 
We denote by $\SP_{\tl{X}}$ the subsheaf of 
$\tl{\jmath}_{\ast} j^{-1}\SO_X$ whose sections are defined by 
\begin{multline*}
\Gamma(\Omega;\SP_{\tl{X}}) \coloneq 
\{f\in\Gamma(\Omega;\tl{\jmath}_{\ast} 
j^{-1}\SO_X)\mid \textit{For any $\theta\in\Omega\cap \varpi^{-1}(D)$,} \\ 
\textit{$f$ admits a Puiseux expansion along $D \subset X$ at $\theta$.}\}
\end{multline*}
for open subsets $\Omega\subset\tl{X}$.
Then we define the sheaf of 
Puiseux germs $\SP_{\varpi^{-1}(D)}$ on $\varpi^{-1}(D)$ by   
\begin{align}
\SP_{\varpi^{-1}(D)}\coloneq\tl{\imath}^{\,-1}\SP_{\tl{X}}.
\end{align}
For a point $\theta \in \varpi^{-1}(D)$ if we take a local coordinate 
$(u,v)=(u_1, \ldots, u_l, v_1, \ldots, v_{n-l})$ of $X$ 
on a neighborhood of $\varpi ( \theta ) \in D$ in $X$ such that 
$\varpi ( \theta )=(0,0) \in D= \{ u_1 u_2 \cdots u_l=0 \}$ 
then the stalk of $\SP_{\varpi^{-1}(D)}$ at $\theta$ 
is isomorphic to the ring 
\begin{equation}
\bigcup_{p\in\ZZ_{\geq1}} \ 
\CC \{ u_1^{\frac{1}{p}}, \ldots, u_l^{\frac{1}{p}}, 
v_1, \ldots, v_{n-l} \} 
\bigl[
u_1^{- \frac{1}{p}}, \ldots, u_l^{- \frac{1}{p}}
\bigr].
\end{equation}
of Puiseux series along $D \subset X$. We 
denote by $\SP_{\varpi^{-1}(D)}^{\leq 0}$ the subsheaf of 
$\SP_{\varpi^{-1}(D)}$ consisting of sections locally contained in 
the ring 
\begin{equation}
\bigcup_{p\in\ZZ_{\geq1}} \ 
\CC \{ u_1^{\frac{1}{p}}, \ldots, u_l^{\frac{1}{p}}, 
v_1, \ldots, v_{n-l} \} 
\end{equation}
for some (hence, any) local coordinate 
$(u,v)=(u_1, \ldots, u_l, v_1, \ldots, v_{n-l})$ of $X$ 
as above. By this definition, it is clear that 
for any point $x \in D$ there exist its neighborhood 
$U$ in $X$ and a subsheaf $\SP_{\varpi^{-1}(D \cap U)}^{\prime} 
\subset \SP_{\varpi^{-1}(D \cap U)}$ of $\CC_{\varpi^{-1}(D \cap U)}$-modules  
defined on 
the open subset $\varpi^{-1}(D \cap U) \subset \varpi^{-1}(D)$ 
such that the natural morphism 
\begin{equation}
\SP_{\varpi^{-1}(D \cap U)}^{\prime} \longrightarrow 
\SP_{\varpi^{-1}(D \cap U)}/ \SP_{\varpi^{-1}(D \cap U)}^{\leq 0}
\end{equation}
is an isomorphism. We call such $\SP_{\varpi^{-1}(D \cap U)}^{\prime}$ 
a representative subsheaf of $\SP_{\varpi^{-1}(D \cap U)}$. 
By slightly modifying the definition of the multiplicities in 
D'Agnolo-Kashiwara \cite[Section 5.3]{DK18}, we shall use 
the following one (cf. \cite[Definition 2.4]{KT23}). 

\begin{definition}\label{def-multi}
(cf. \cite[Section 5.3]{DK18} and \cite[Definition 2.4]{KT23}) 
In the situation as above, we say that a morphism 
$N: \SP_{\varpi^{-1}(D \cap U)}^{\prime} \longrightarrow 
( \ZZ_{\geq 0})_{\varpi^{-1}(D \cap U)}$ of sheaves of 
sets is a multiplicity along $D \cap U \subset U$ if 
there exists 
a ramification $\rho  : X'\to U$ 
of $U$ along $D \cap U \subset U$ such that 
for any $\theta \in \varpi^{-1}(D \cap U)$ the subset 
$N_{\theta}^{>0}:= N_{\theta}^{-1}( \ZZ_{>0}) \subset 
\SP_{\varpi^{-1}(D \cap U), \theta}^{\prime}$ is finite 
and the pull-backs of its elements 
$f \in N_{\theta}^{>0}$ by $\rho$ are meromorphic 
functions on $X^{\prime}$ along $D^{\prime}:= \rho^{-1}(D) 
\subset X^{\prime}$ and form a good set 
$\{ [f \circ \rho ] \ | \ f \in N_{\theta}^{>0} \} 
 \subset \SO_{X^{\prime}}(\ast D^{\prime})/ \SO_{X^{\prime}}$ 
of irregular values on $(X^{\prime}, D^{\prime})$ 
on a neighborhood of the point $\rho^{-1}( \varpi ( \theta ))
\in D^{\prime}$. 
\end{definition}

\begin{definition}\label{def-qnf}
(cf. \cite[Definition 5.3.1]{DK18} and \cite[Definition 2.5]{KT23}) 
In the situation as above, we say that an 
$\RR$-constructible enhanced sheaf 
$F \in \BEC( \CC_X)$ on $X$ 
has a quasi-normal form along the normal crossing 
divisor $D \cap U \subset U$ if there exists a 
multiplicity $N: \SP_{\varpi^{-1}(D \cap U)}^{\prime} \longrightarrow 
( \ZZ_{\geq 0})_{\varpi^{-1}(D \cap U)}$ 
such that any point $\theta\in\varpi^{-1}(D \cap U)$ has 
its sectorial open neighborhood 
$V_{\theta} \subset U \setminus D \subset \tl{X}$ for 
which we have an isomorphism 
\begin{equation}
\pi^{-1} \CC_{V_{\theta}} \otimes F 
\simeq
\bigoplus_{f \in N_{\theta}^{>0}} 
\Bigl(  {\rm E}_{V_{\theta} |X}^{\Re f} 
\Bigr)^{N(f)}.
\end{equation}
\end{definition}
Enhanced ind-sheaves having a quasi-normal form along the 
normal crossing divisor $D \cap U \subset U$ are defined similarly. 

\begin{lemma}[{\cite[Lemma 7.7]{Tak22}}]\label{lem-hda-dk}  
Assume that a holonomic $\SD_X$-module $\SM$ has a 
quasi-normal form along the normal crossing 
divisor $D \subset X$. Then for any point $x \in D$ 
there exist a subanalytic open neighborhood 
$U$ of $x$ in $X$ such that the 
$\RR$-constructible enhanced ind-sheaf 
\begin{equation}
\pi^{-1} \CC_{U} \otimes Sol_X^\rmE ( \SM ) 
\simeq
\pi^{-1} \CC_{U \setminus D} \otimes 
Sol_X^\rmE ( \SM ) 
\end{equation}
has a quasi-normal form along the normal crossing 
divisor $D \cap U \subset U$. 
\end{lemma}

\begin{proof}
The proof is similar to that of \cite[Lemma 5.4.4]{DK18}. 
With the representative subsheaf $\SP_{\varpi^{-1}(D \cap U)}^{\prime}$ 
of $\SP_{\varpi^{-1}(D \cap U)}$ at hands, it suffices to use 
\eqref{mon-eq}, \eqref{mono-eqn} and 
\cite[Propositions 3.10 and 3.19]{IT20}. 
\end{proof}
In the situation of Lemma \ref{lem-hda-dk}, let 
$N: \SP_{\varpi^{-1}(D \cap U)}^{\prime} \longrightarrow 
( \ZZ_{\geq 0})_{\varpi^{-1}(D \cap U)}$ be the multiplicity 
for which the enhanced ind-sheaf 
$\SF  \simeq \pi^{-1} \CC_{U} \otimes Sol_X^\rmE ( \SM ) 
\in \BEC( \I \CC_X)$ has a quasi-normal form along 
the normal crossing divisor $D \cap U \subset U$. Then 
by the proof of Lemma \ref{lem-hda-dk}, the sections of 
the subsheaf $N^{>0}= N^{-1}(( \ZZ_{>0})_{\varpi^{-1}(D \cap U)}) 
\subset \SP_{\varpi^{-1}(D \cap U)}^{\prime}$ are 
the exponential factors of $\SM$. Moreover, if 
the divisor $D \cap U \subset U$ is smooth and connected, 
then the non-negative rational number 
\begin{equation}
\dsum_{f \in N_{\theta}^{>0}} N_{\theta}(f) \cdot {\rm ord}_{D \cap U}(f) 
\quad \in \QQ_{\geq 0}  
\end{equation}
associated to a point $\theta \in \varpi^{-1}(D \cap U)$ 
is an integer and does not depend on the choice of 
$\theta \in \varpi^{-1}(D \cap U)$, where for 
the exponential factor $f \in N_{\theta}^{>0}$ of $\SM$ 
the rational number ${\rm ord}_{D \cap U}(f)  \geq 0$ stands for 
the pole order of $f$ along $D \cap U$. We call it the 
irregularity of $\SM$ along $D \cap U$ and 
denote it by ${\rm irr}_{D \cap U} ( \SM )$. 
If $D \subset X$ itself is smooth and connected, 
we define the irregularity ${\rm irr}_{D} ( \SM ) 
\in \ZZ_{\geq 0}$ of $\SM$ along $D \subset X$ similarly. 
By Lemmas \ref{DK-lem-2} and \ref{lem-hda-dk},  
we obtain the following higher-dimensional analogue 
of \cite[Proposition 5.4.5]{DK18}. 
For a precise explanation of the proof of 
\cite[Proposition 5.4.5]{DK18}, see \cite[Remark 2.10]{KT23}. 

\begin{proposition}[{\cite[Proposition 7.8]{Tak22}}]\label{hda-dk}  
Assume that a holonomic $\SD_X$-module $\SM$ has a 
quasi-normal form along the normal crossing 
divisor $D \subset X$. Then for any point $x \in D$ 
there exist a subanalytic open neighborhood 
$U$ of $x$ in $X$ and an 
$\RR$-constructible enhanced sheaf 
$F \in \BEC( \CC_X)$ on $X$ 
having a quasi-normal form along the normal crossing 
divisor $D \cap U \subset U$ 
such that 
\begin{equation}
\pi^{-1} \CC_{U} \otimes Sol_X^\rmE ( \SM ) 
\simeq
\pi^{-1} \CC_{U \setminus D} \otimes 
Sol_X^\rmE ( \SM ) 
\simeq
\CC_X^\rmE \Potimes F. 
\end{equation}
\end{proposition}

\begin{proposition}\label{thm-1}
Let $D \subset X$ be a normal crossing divisor on a 
complex manifold $X$, $x \in D$ a point on it and $\SM$ 
a holonomic $\SD_X$-module having a quasi-normal form along 
$D$. Let $(u,v)=(u_1, \ldots, u_l, v_1, \ldots, v_{n-l})$ 
be a local coordinate of $X$ such that 
$x=(0,0) \in D= \{ u_1 u_2 \cdots u_l=0 \}$. 
Then we have the following results. 
\begin{enumerate}
\item[\rm{(i)}] 
The solution complex $Sol_X ( \SM )$ of $\SM$ is 
constructible with respect to the standard 
stratification of $X$ associated to the normal 
crossing divisor $D \subset X$. 
\item[\rm{(ii)}] 
Assume moreover that $l=1$. Then we have 
\begin{equation} 
\chi \bigl(
Sol_X ( \SM )
\bigr)(x)= - {\rm irr}_D( \SM )(x), 
\end{equation} 
where ${\rm irr}_D( \SM )(x) \in \ZZ_{\geq 0}$ 
stands for the irregularity of $\SM$ 
along $D \subset X$ on a neighborhood of the point $x \in D$. 
\item[\rm{(iii)}] 
Assume moreover that $l \geq 2$. Then we have 
\begin{equation} 
\chi \bigl( Sol_X ( \SM ) \bigr)(x)=0. 
\end{equation} 
\item[\rm{(iv)}] 
In the situation of (ii), we set 
\begin{equation}
D_i^{\circ}:= D_i \setminus \Bigl( \bigcup_{j \not= i} D_j \Bigr) 
\qquad (1 \leq i \leq l)
\end{equation}
and assume also that there exists $1 \leq i \leq l$ such that 
${\rm irr}_{D_i^{\circ}}( \SM ) = 0$. Then we have 
$Sol_X ( \SM )_x \simeq 0$. 
\end{enumerate}
\end{proposition}

\begin{proof}
By Proposition \ref{hda-dk} we take an 
$\RR$-constructible enhanced sheaf 
$F \in \BEC( \CC_X)$ on $X$ 
having a quasi-normal form along the normal crossing 
divisor $D= \{ u_1 \cdots u_l=0 \} 
\subset X$ such that there exists an isomorphism 
\begin{equation}
Sol_X^\rmE ( \SM ) 
\simeq
\pi^{-1} \CC_{X \setminus D} \otimes 
Sol_X^\rmE ( \SM ) 
\simeq
\CC_X^\rmE \Potimes F
\end{equation}
on a neighborhood of $x \in D$ in $X$. Let $U \subset X$ 
be a sufficiently small open neighborhood of $x \in D$ in $X$ 
and $\rho : X^{\prime} \longrightarrow U$ a ramification of 
$U$ along $D \cap U \subset U$ such that 
$\bfD \rho^\ast(\SM|_U)$
has a normal form along the normal crossing divisor 
$D^{\prime}:= \rho^{-1}(D\cap U) \subset X^{\prime}$. 
Then the restriction $X^{\prime} \setminus D^{\prime} 
\longrightarrow U \setminus (D \cap U)$ of $\rho$ is an 
unramified covering and we denote its covering degree 
by $d>0$. Moreover by Theorem \ref{thm-4} (ii) 
and \cite[Proposition 4.7.14 (ii)]{DK16} 
for the enhanced sheaf $G:= \bfE \rho^{-1}(F|_{U \times \RR}) 
\in \BEC( \CC_{X^{\prime}})$ on $X^{\prime}$ we have an 
isomorphism 
\begin{equation}
Sol_{X^{\prime}}^\rmE \bigl(  \bfD \rho^\ast(\SM|_U) \bigr) 
\simeq
\CC_{X^{\prime}}^\rmE \Potimes G. 
\end{equation}
Let $x^{\prime} \in D^{\prime}$ be the unique point in the 
one-point set $\rho^{-1}(x) \subset D^{\prime}$. 
Then by the proof of Proposition \ref{prop-1} we can 
easily show that 
\begin{equation}
\chi \bigl( Sol_{X^{\prime}} \bigl( \bfD \rho^\ast(\SM|_U) \bigr) \bigr)
(x^{\prime})= d \cdot \chi \bigl( Sol_X ( \SM ) \bigr)(x). 
\end{equation}
By killing the monodromies of the enhanced sheaf 
$G \in \BEC( \CC_{X^{\prime}})$ with the help of 
the Mayer-Vietoris exact sequences associated to 
an open covering of $X^{\prime} \setminus D^{\prime}$ 
by some open sectors along $D^{\prime} \subset X^{\prime}$ 
as in the proof of \cite[Proposition 3.15 and Theorem 3.18]{IT20}, 
we can prove the assertions (i)--(iv) along the same lines as in 
the proof of Proposition \ref{prop-1}. 
This completes the proof. 
\end{proof}
Note that the assertions (ii) and (iii) 
of Proposition \ref{thm-1} have been 
obtained previously by Hu and Teyssier in \cite[Proposition 5.5]{HT25} by a totally 
different method. Whereas our proof relies on the 
theories of ind-sheaves and the irregular Riemann-Hilbert correspondence, 
Hu and Teyssier used Sabbah's study of irregularity sheaves 
in \cite{Sab17}. By Proposition \ref{thm-1} (ii) and (iii) 
we obtain the following 
formula of Hu-Teyssier \cite[Proposition 5.6]{HT25} 
for the characteristic cycles of holonomic $\SD$-modules 
having a quasi-normal form.

\begin{corollary}[{Hu-Teyssier \cite[Proposition 5.6]{HT25}}]\label{thm-CCquasi} 
In the situation of Proposition \ref{thm-1} let $U$ be an open neighborhood of 
the point $x\in D$ in $X$ on which the local coordinate 
$(u,v)=(u_1,\dots,u_l,v_1,\dots,v_{n-l})$ of $X$ such that 
$x=(0,0)\in D=\{u_1 u_2\cdots u_l=0\}$ $(1\leq l\leq n)$ is defined. 
For $1\leq i\leq l$ we set $D_i\coloneq\{(u,v)\in U\mid u_i=0\}\subset U$ and 
\begin{equation}
D_i^\circ\coloneq D_i\setminus \big(\bigcup_{j\neq i}D_j\big) \quad \subset D_i.
\end{equation}
Let $r \geq 0$ be the generic rank of $\SM$. 
Then on the open subset $U\subset X$ the characteristic cycle 
$\CCyc(\SM)$ of $\SM$ is given by the formula 
\begin{align}
\CCyc(\SM)&= r \cdot [T_X^\ast X] 
+\sum_{i=1}^l(\irr_{D_i^\circ}(\SM)+r)\cdot[T_{D_i}^\ast X] \\
&+\sum_{1\leq i<j\leq l}(\irr_{D_i^\circ}(\SM)+\irr_{D_j^\circ}(\SM)+r)
\cdot [T_{D_i\cap D_j}^\ast X] \\
&\cdots \cdots \cdots \\
&+(\irr_{D_1^\circ}(\SM)+\cdots+\irr_{D_l^\circ}(\SM)+r)\cdot[
T_{D_1\cap\cdots\cap D_l}^\ast X] \\
&=\sum_{r=0}^l\Bigl\{\sum_{1\leq i_1<\cdots<i_r\leq l}(\irr_{D_{i_1}^\circ}(\SM)
+\cdots+\irr_{D_{i_r}^\circ}(\SM)+r)
\cdot[T_{D_{i_1}\cap\cdots\cap D_{i_r}}^\ast X]\Bigr\}.
\end{align}
\end{corollary}

\begin{proof}
By Proposition \ref{thm-1} we can easily show that on 
$U\subset X$ we have an equality 
\begin{align}
\chi(Sol_X(\SM))&= r \cdot \mathbbm{1}_X
+\sum_{i=1}^l(\irr_{D_i^\circ}(\SM)+r)\cdot \mathbbm{1}_{D_i} \\
&\cdots \cdots \cdots \\
&+(\irr_{D_1^\circ}(\SM)+\cdots+\irr_{D_l^\circ}(\SM)+r)
\cdot \mathbbm{1}_{D_1\cap\cdots\cap D_l}.
\end{align}
Indeed, it suffices to use the binomial identities 
\begin{equation}
(1-1)^m=1-\binom{m}{1}+\binom{m}{2}-\cdots+(-1)^m\binom{m}{m}=0
\end{equation} 
for positive integers $m\geq1$.
Then the assertion immediately follows from 
Kashiwara's index theorem for holonomic $\SD$-modules (see \cite{Kas83}).
\end{proof}

\begin{definition}
Let $X$ be a complex manifold and $Y\subset X$ a closed hypersurface in it.
Then we say that a holonomic $\SD_X$-module $\SM$ is 
a meromorphic connection along $Y\subset X$ if 
\begin{enumerate}[noitemsep]
\item [\rm (i)] $\SM\simto\SM(\ast Y)$
\item [\rm (ii)] $\ssupp(\SM)\subset Y$.
\end{enumerate}
\end{definition}

We recall the following fundamental result obtained by 
Kedlaya and Mochizuki.

\begin{theorem}[\cite{Ked10, Ked11, Moc11}]\label{thm-MK}
For a holonomic $\SD_X$-module $\SM$ and $x\in X$,
there exist an open neighborhood $U$ of $x$, 
a closed hypersurface $Y\subset U$,
a complex manifold $X'$ and
a projective morphism $\nu : X'\to U$ such that
\begin{enumerate}[noitemsep]
\item [\rm (i)] $\rm{sing.supp} (\SM)\cap U\subset Y$,
\item [\rm (ii)] $D:= \nu^{-1}(Y)$ is a normal crossing divisor in $X'$,
\item [\rm (iii)] $\nu$ induces an isomorphism $X'\bs D\simto U\bs Y$,
\item [\rm (iv)] $(\bfD \nu^\ast\SM)(\ast D)$ has a quasi-normal form along $D$.
\end{enumerate}
\end{theorem}
This is a generalization of the classical 
Hukuhara-Levelt-Turrittin theorem to higher dimensions. 
By Proposition \ref{thm-1} and Theorem \ref{thm-MK} we obtain a method to 
calculate the characteristic cycles of meromorphic connections as follows.
Let $X$ be a complex manifold and $\SM$ a meromorphic connection along 
a closed hypersurface $Y\subset X$  in $X$.
The problem being local, we may replace $X$ by a neighborhood of 
a point $x\in Y$ and assume that there exists 
a projective morphism $\nu\colon X^\prime\longrightarrow X$ of 
a complex manifold $X^\prime$ such that 
\begin{enumerate}[noitemsep]
\item [\rm (i)] $D\coloneq \nu^{-1}(Y)\subset X^\prime$ is 
a normal crossing divisor in $X^\prime$,
\item [\rm(ii)] $\nu$ induces an isomorphism $X^\prime\setminus D\simto X\setminus Y$,
\item [\rm (iii)] the meromorphic connection 
$\bfD\nu^\ast\SM\simto (\bfD\nu^\ast\SM)(\ast D)$ on $X^\prime$ 
along $D$ has a quasi-normal form along $D$. 
\end{enumerate}
In this situation, by Proposition \ref{thm-1} we have a formula for 
$\chi(Sol_{X^\prime}(\bfD\nu^\ast\SM))\in \CF_\ZZ(X^\prime)$.
Moreover there exists an isomorphism 
$\SM\simeq\bfD\nu_\ast(\bfD^\ast\nu\SM)$ and hence 
by Lemma \ref{lem:comut-1} we obtain an equality 
\begin{equation}
\chi(Sol_X(\SM))=\int_\nu\chi(Sol_{X^\prime}(\bfD\nu^\ast\SM)),
\end{equation}
where $\displaystyle\int_\nu\colon\CF_\ZZ(X^\prime)
\longrightarrow\CF_\ZZ(X)$ stands for the push-forward of 
$\ZZ$-valued constructible functions by the morphism 
$\nu\colon X^\prime\longrightarrow X$.
Then we obtain a formula for the characteristic cycle 
$\CCyc(\SM)$ of $\SM$ by Kashiwara's index theorem (see \cite{Kas83}).

\begin{example}\label{ex-CC}
\begin{enumerate}[wide,labelwidth=!,labelindent=0pt]
\item [\rm (i)]
Consider the case where $X$ is the complex plane $\CC_z^2$ endowed with the 
standard coordinate $z=(z_1,z_2)=(x,y)$.
Then for the closed hypersurface $Y\coloneq\{x=0\}\subset X=\CC^2$ 
and the meromorphic function 
\begin{equation}
\varphi(x,y)\coloneq \frac{y}{x} \qquad \bigl((x,y)\in X\setminus Y\bigr)
\end{equation} 
on $X=\CC^2$ along $Y\subset X$ we set $\SM\coloneq 
\SE_{X\setminus Y\vbar X}^\varphi\in \Modhol(\SD_X)$.
In this case, $\varphi$ has a point of indeterminacy at the origin 
$\{0\}=\{x=y=0\}\subset X$ and hence $\SM$ does not have a quasi-normal form.
Let $\nu\colon X^\prime\longrightarrow X$ be the blow-up of $X$ along the origin 
and $E\coloneq \nu^{-1}(0)\simeq \PP^1$ its exceptional 
divisor and set $D\coloneq \nu^{-1}(Y)\subset X^\prime$. 
We denote by $\tl{Y}\subset D\subset X^\prime$ (resp. $\{y=0\}^\sim \subset 
X^\prime$) the proper transform of $Y=\{x=0\}\subset X$ (resp. $\{y=0\}\subset X$) in $X^\prime$. 
Then $D=\nu^{-1}(Y)=\tl{Y}\cup E$ and $D\cup\{y=0\}^\sim$ are normal 
crossing divisors in $X^\prime$ and the meromorphic function $\varphi\circ\nu\colon 
X^\prime\setminus D\longrightarrow\CC$ on $X^\prime$ along $D\subset X^\prime$ 
has no point of indeterminacy on the whole $X^\prime$. 
Indeed, the pole orders of $\varphi\circ\nu$ along the smooth divisors 
$\tl{Y}, E, \{y=0\}^\sim\subset X^\prime$ are $1,0,-1$ respectively.
By Proposition \ref{prop-1}, this implies that for the meromorphic 
connection $\bfD\nu^\ast\SM\simeq\SE_{X^\prime\setminus 
D\vbar X^\prime}^{\varphi\circ\nu}$ and a point $z^\prime\in X^\prime$ on $X^\prime$ we have 
\begin{equation}
\chi(Sol_{X^\prime}(\bfD\nu^\ast\SM))(z^\prime)=
\begin{cases}
1 & (z^\prime\in X^\prime\setminus D), \\
\\
-1 & (z^\prime\in D\setminus E=\tl{Y}\setminus E), \\
\\
0 & (z^\prime\in E).
\end{cases}
\end{equation}
We thus obtain 
\begin{equation}
\chi(Sol_X(\SM))(0)=\int_{\nu^{-1}(0)=E}\chi(Sol_{X^\prime}(\bfD\nu^\ast\SM))=0
\end{equation}
and hence 
\begin{equation}
\chi(Sol_X(\SM))= 1\cdot\mathbbm{1}_X -2\cdot\mathbbm{1}_Y +1\cdot\mathbbm{1}_{\{0\}}.
\end{equation}
Since $Y=\{x=0\}$ is smooth and its Euler obstruction $\Eu_Y$ is equal to 
$\mathbbm{1}_Y$, by Kashiwara's index theorem we finally obtain
\begin{equation}
\CCyc(\SM)= 1\cdot[T_X^\ast X] + 2\cdot[T_Y^\ast X] +1\cdot[T_{\{0\}}^\ast X].
\end{equation}
\item  [\rm (ii)]
In the situation of (i), for the closed hypersurface 
$Y=\{x=0\}\subset X=\CC^2$ and $k\in \ZZ_{>0}$ let us consider the meromorphic function 
\begin{equation}
\varphi(x,y)\coloneq \frac{y^k}{x} \qquad \bigl((x,y)\in X\setminus Y\bigr)
\end{equation}
on $X=\CC^2$ along $Y\subset X$ and set $\SM\coloneq 
\SE_{X\setminus Y\vbar X}^\varphi\in\Modhol(\SD_X)$.
Also in this case, as in the proof of \cite[Theorem 3.1 (i)]{MT11} 
we can construct a sequence of blow-ups along a point 
\begin{equation}
X_k \overset{\nu_k}{\longrightarrow} 
X_{k-1} \overset{\nu_{k-1}}{\longrightarrow}
\cdots
\overset{\nu_2}{\longrightarrow} X_1
\overset{\nu_1}{\longrightarrow} X
\end{equation} 
such that for their composition 
$\nu\coloneq\nu_k\circ\cdots\circ\nu_1\colon X^\prime\coloneq X_k
\longrightarrow X$ the meromorphic function $\varphi\circ\nu$ on $X^\prime$ 
has no point of indeterminacy on the whole $X^\prime$.
Moreover its pole order along the proper transform $\tl{E_1}(\simeq\PP^1)
\subset X^\prime$ of the first exceptional divisor $E_1\coloneq\nu_1^{-1}(0)
\subset X_1$ in $X^\prime$ is $0$.
We thus obtain $\chi(Sol_X(\SM))(0)=0$ and as in (i) we see that
\begin{equation}
\CCyc(\SM)= 1\cdot[T_X^\ast X] + 2\cdot[T_Y^\ast X] +1\cdot[T_{\{0\}}^\ast X].
\end{equation}
also in this case.   
\item [\rm (iii)]
In the situation of (i), for the closed hypersurface $Y=\{x=0\}\subset 
X=\CC^2$ and $k\in \ZZ_{>0}$ let us consider the meromorphic function 
\begin{equation}
\varphi(x,y)\coloneq \frac{y}{x^k} \qquad \bigl((x,y)\in X\setminus Y\bigr)
\end{equation}
on $X=\CC^2$ along $Y\subset X$ and set $\SM\coloneq \SE_{X\setminus 
Y\vbar X}^\varphi\in\Modhol(\SD_X)$.
In this case, for the morphism $\nu=\nu_k\circ\cdots\circ\nu_1\colon 
X^\prime\longrightarrow X$ in (ii) the meromorphic function $\varphi\circ\nu$ 
on $X^\prime$ has no point of indeterminacy on the whole $X^\prime$.
For $1\leq i\leq k$ we denote the proper transform of the exceptional 
divisor $E_i (\simeq\PP^1) \subset X_i$ of the $i$-th blow-up $\nu_i$ 
in $X^\prime$ by $\tl{E_i}(\simeq\PP^1)\subset X^\prime$. 
Then for $1\leq i\leq k$ the pole order of the meromorphic function 
$\varphi\circ\nu$ along $\tl{E_i}\simeq\PP^1$ is equal to $k-i$ in this case.
Nevertheless, by using the fact that the Euler characteristic of 
$\PP^1$ minus $2$ points is equal to 0, we can easily show that 
\begin{equation}
\chi(Sol_X(\SM))(0) =\int_{\nu^{-1}(0)}\chi(Sol_{X^\prime}(\bfD\nu^\ast\SM))=0.
\end{equation}
Then as in (i) and (ii), by Kashiwara's index theorem we finally obtain 
\begin{equation}
\CCyc(\SM)= 1\cdot[T_X^\ast X] + (k+1)\cdot[T_Y^\ast X] +k\cdot[T_{\{0\}}^\ast X].
\end{equation}
\item [\rm (iv)]
In the situation of (i), for the closed hypersurface $Y=\{x^2-y^3=0\}
\subset X=\CC^2$ let us consider the meromorphic function 
\begin{equation}
\varphi(x,y)\coloneq \frac{1}{x^2-y^3} \qquad \bigl((x,y)\in X\setminus Y\bigr)
\end{equation}
on $X=\CC^2$ along $Y\subset X$ and set $\SM\coloneq \SE_{X\setminus 
Y\vbar X}^\varphi\in\Modhol(\SD_X)$.
In this case, in a standard way we can construct a sequence of blow-ups along a point 
\begin{equation}
X_3 \overset{\nu_3}{\longrightarrow} 
X_2 \overset{\nu_2}{\longrightarrow}
X_1 \overset{\nu_1}{\longrightarrow} X
\end{equation} 
such that for their composition $\nu\coloneq \nu_3\circ\nu_2\circ\nu_1\colon 
X^\prime\coloneq X_3\longrightarrow X$ the divisor $D\coloneq\nu^{-1}(Y)
\subset X^\prime$ is normal crossing.
For $1\leq i\leq 3$ we denote the proper transform of the exceptional 
divisor $E_i (\simeq\PP^1) \subset X_i$ of the $i$-th blow-up $\nu_i$ in $X^\prime$ 
by $\tl{E_i}(\simeq\PP^1)\subset X^\prime$.
Then the pole orders of the meromorphic function $\varphi\circ\nu$ on 
$X^\prime$ along $\tl{E_1},\tl{E_2},\tl{E_3}$ are equal to $2,3,6$ respectively.
Moreover obviously its pole order along the proper transform 
$\tl{Y}\subset X^\prime$ of $Y\subset X$ in $X^\prime$ is equal to $1$.
In Figure \ref{fig:1} below we show how the irreducible components 
of the normal crossing divisor
\begin{equation}
D:=\nu^{-1}(Y)=\tl{E_1}\cup\tl{E_2}\cup\tl{E_3}\cup\tl{Y}
\end{equation}
intersect each other.
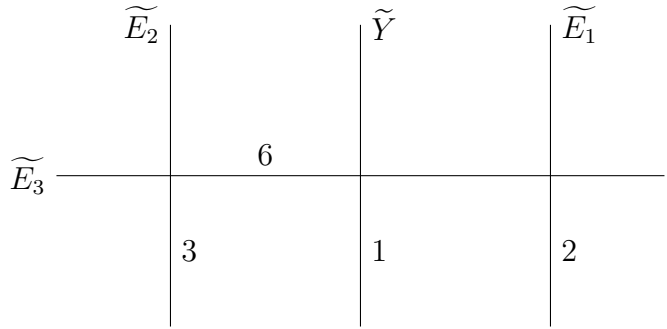
\begin{figure}[b] 
\centering
\begin{tikzpicture}[scale=0.5]
\coordinate (O) at (0,0);
\coordinate (Xm) at (-8,0);
\coordinate (XM) at (8,0);
\coordinate (Xp) at (-2.5,0);

\coordinate (Y1m) at (-5,-4);
\coordinate (Y1M) at (-5,4);
\coordinate (Y1p) at (-5,-2);

\coordinate (Y2m) at (0,-4);
\coordinate (Y2M) at (0,4);
\coordinate (Y2p) at (0,-2);

\coordinate (Y3m) at (5,-4);
\coordinate (Y3M) at (5,4);
\coordinate (Y3p) at (5,-2);

\draw (Xm)--(XM);
\draw (Y1m)--(Y1M);
\draw (Y2m)--(Y2M);
\draw (Y3m)--(Y3M);

\draw (Xm) node[left]{$\tl{E_3}$};
\draw (Y1M) node[left]{$\tl{E_2}$};
\draw (Y2M) node[right]{$\tl{Y}$};
\draw (Y3M) node[right]{$\tl{E_1}$};

\draw (Xp) node[above]{$6$};
\draw (Y1p) node[right]{$3$};
\draw (Y2p) node[right]{$1$};
\draw (Y3p) node[right]{$2$};
\end{tikzpicture} 
\caption{Irreducible components of $D$ and pole orders of $\varphi\circ\nu$.}
\label{fig:1}
\end{figure}
The number attached to each irreducible component in Figure 
\ref{fig:1} stands for the pole order of $\varphi\circ\nu$ along it.
Then by Proposition \ref{prop-1} for the closed subset 
$\nu^{-1}(0)=\tl{E_1}\cup\tl{E_2}\cup\tl{E_3}\subset D$ we obtain 
\begin{align}
\chi(Sol_X(\SM))(0)&=\int_{\nu^{-1}(0)}\chi(Sol_{X^\prime}(\bfD\nu^\ast\SM)) \notag \\
&= (-2)\cdot(2-1)+(-3)\cdot(2-1)+(-6)\cdot(2-1-1-1) \notag \\
&= 1.
\end{align}
On the other hand, by Kashiwara's formula for the Euler 
obstructions of complex hypersurfaces having only isolated 
singular points (see \cite{Kas83}) for a point $(x,y)\in X$ of $X$ we have 
\begin{equation}
\Eu_Y((x,y))=
\begin{cases}
0 & ((x,y) \notin Y), \\
\\
1 & ((x,y)\in Y\setminus\{0\}), \\
\\
2 & ((x,y)=0=(0,0)).
\end{cases}
\end{equation}
We thus obtain
\begin{equation}
\chi(Sol_X(\SM))=1\cdot\mathbbm{1}_X -2\cdot\Eu_Y + 4\cdot\mathbbm{1}_{\{0\}}
\end{equation}
hence 
\begin{equation}
\CCyc(\SM)= 1\cdot[T_X^\ast X] +2\cdot[\overline{T_{Y_{\reg}}^\ast X}] +4\cdot[T_{\{0\}}^\ast X]
\end{equation}
in this case.
\end{enumerate}
\end{example}

\section{An index formula for irregular connections}\label{sec-index}
In this section, by Proposition \ref{thm-1} we obtain an index formula for 
irregular integrable connections, which expresses the global 
Euler-Poincar{\'e} indices of their algebraic de Rham complexes.
This is a higher-dimensional analogue of a result of Bloch-Esnault 
\cite{BE04a}, but our proof here will be very different from that of \cite{BE04a}.
Whereas the proof of Bloch and Esnault relies on the rapid decay 
homology groups of irregular connections introduced by them 
(see Hien \cite{Hi09} for the higher-dimensional case), we prove our 
index formula by the results on the solution complexes to irregular 
holonomic $\SD$-modules with a quasi-normal form in Section \ref{sec:qNor}.
Let $U$ be a not necessarily complete and smooth algebraic variety 
over $\CC$ and $\SN$ an algebraic integrable connection on it.
Throughout this section we will be interested in its algebraic de Rham complex 
\begin{align}
& \DRalg_U(\SN) \coloneq \Omega_U\Lotimes{\SD_U} \SN \\
& \simeq \Bigl[\cdots \longrightarrow 0 \longrightarrow \SN
\longrightarrow \SN\otimes_{\SO_U}\Omega^1_U \longrightarrow \cdots 
\longrightarrow \SN\otimes_{\SO_U}\Omega^{\dim U}_U
\longrightarrow 0 \longrightarrow 
 \cdots\Bigr]
\end{align}
and its global Euler-Poincar{\'e} index 
\begin{equation}
\chi(U; \DRalg_U(\SN)) \coloneq \chi (\rsect(U;\DRalg_U(\SN))) \quad \in\ZZ.
\end{equation} 
For $j \in \ZZ$ we set 
\begin{equation}
H^j_{{\rm DR}}(U; \SN )
 \coloneq H^j (U; \DRalg_U(\SN)[- \dim U] ) 
\end{equation} 
and call it the $j$-th algebraic de Rham cohomology of the 
connection $\SN$. Then we define the index $\chi^{{\rm alg}}( \SN ) \in \ZZ$ of 
$\SN$ to be the global Euler-Poincar{\'e} index of 
the shifted algebraic de Rham complex $\DRalg_U(\SN) [- \dim U]$ as 
\begin{equation}
\chi^{{\rm alg}}( \SN ) \coloneq 
\dsum_{j \in \ZZ} (-1)^j \dim H^j_{{\rm DR}}(U; \SN )
\quad \in\ZZ. 
\end{equation} 

\begin{proposition}\label{pro-eit} 
Let $i_U\colon U\hookrightarrow X$ be a smooth compactification of $U$ such 
that $D\coloneq X\setminus U\subset X$ is a not necessarily normal crossing 
divisor in $X$ and $X^\an$ the underlying 
complex manifold of $X$ endowed with the classical topology.
Then for the algebraic meromorphic connection $\SM\coloneq 
i_{U\ast}\SN\simeq\displaystyle\int_{i_U}\SN\in\Modhol(\SD_X)$ on $X$ we have
\begin{equation}
\chi^{{\rm alg}}( \SN )
 =\int_{X^\an}\chi(Sol_X(\SM)),
\end{equation}
where $\displaystyle\int_{X^\an}\colon \CF_{\ZZ}(X^\an)\longrightarrow
\ZZ$ stands for the topological (Euler) integral 
of $\ZZ$-valued constructible functions on $X^\an$.
\end{proposition}

\begin{proof}
For the one point set $\{\pt\}$ 
let us consider the commutative diagram 
\begin{equation}
\vcenter{
\xymatrix@M=7pt@C=36pt@R=24pt{
U \ar@{^{(}->}[r]_-{i_U} \ar[rd]_-{a_U} & X \ar[d]_-{a_X} \\
& \{\pt\}.
}}
\end{equation}   
Then we have isomorphisms
\begin{align}
\rsect(U;\DRalg_U(\SN)[- \dim U]) 
&\simeq \rsect(U; \SD_{\{\pt\}\leftarrow U}\Lotimes{\SD_U}\SN \ [-\dim U]) \\
&\simeq \rmR a_{U\ast}(\SD_{\{\pt\}\leftarrow U}\Lotimes{\SD_U}\SN \ [-\dim U]) \\
&\simeq (\int_{a_U}\SN) \ [-\dim U].
\end{align}
Moreover by \cite[Thoerem 3.2.3 (i)]{HTT08} we have 
$\displaystyle\int_{a_U}\SN\in\BDChol(\SD_{\{\pt\}})$ i.e. the 
cohomology groups of the complex $\displaystyle\int_{a_U}\SN$ are 
finite dimensional vector spaces over $\CC$.
This implies that for the dual 
\begin{align}
(\rsect(U;\DRalg_U(\SN)[- \dim U]))^\ast 
&\coloneq
\rHom_{\CC}(\rsect(U;\DRalg_U(\SN)[- \dim U]),\CC) \\
&\simeq Sol_{\{\pt\}}(\int_{a_U}\SN) \ [\dim U]
\end{align}
of the complex $\rsect(U;\DRalg_U(\SN)[- \dim U])$ we have an equality 
\begin{equation}
\chi(\rsect(U;\DRalg_U(\SN)[- \dim U]))=\chi((\rsect(U;\DRalg_U(\SN)[- \dim U]))^\ast).
\end{equation}
We thus obtain 
\begin{align}
\chi^{{\rm alg}}( \SN )
&=\chi(Sol_{\{\pt\}}(\int_{a_U}\SN) \ [\dim U]) \\
&=\chi(Sol_{\{\pt\}}(\int_{a_X}\SM) \ [\dim X]). \label{eq:chiX}
\end{align}
On the other hand, by \cite[Proposition 4.7.5]{HTT08} for the holonomic 
$\SD_X$-module $\SM$ and the proper map 
$(a_X)^\an\colon X^\an\longrightarrow\{\pt\}$ there exists an isomorphism
\begin{equation}
DR_{\{\pt\}}(\int_{a_X}\SM)\simeq \rmR(a_X)^\an_\ast DR_X(\SM).
\end{equation}
Applying the Verdier dual functor $\bfD_{\{\pt\}}\simeq(\cdot)^\ast$ 
on the one point set $\{\pt\}$ to it, we obtain isomorphisms
\begin{align}
Sol_{\{\pt\}} (\int_{a_X}\SM) 
\simeq \rmR(a_X)^\an_\ast \bfD_{X^\an}(DR_X(\SM)) \\
\simeq \rsect(X^\an; Sol_X(\SM) \ [\dim X]).
\end{align}
Then by (\ref{eq:chiX}) we obtain the assertion as follows:
\begin{align}
\chi^{{\rm alg}}( \SN )
&=\chi(\rsect(X^\an; Sol_X(\SM) \ [2\dim X])) \\
&=\chi(\rsect(X^\an; Sol_X(\SM))) \\
&=\int_{X^\an}\chi(Sol_X(\SM)).
\end{align}
\end{proof}

By Proposition \ref{thm-1} we can rewrite 
Proposition \ref{pro-eit} more explicitly as follows.  
Recall that a hypersurface $D$ in a smooth algebraic variety $X$ is called 
a strict normal crossing divisor if its irreducible components are smooth. 

\begin{theorem}\label{th-exitm} 
Let $i_U\colon U\hookrightarrow X$ be a smooth compactification of $U$ such 
that $D\coloneq X\setminus U\subset X$ is a strict normal crossing 
divisor in $X$ and the algebraic meromorphic connection $\SM\coloneq 
i_{U\ast}\SN\simeq\displaystyle\int_{i_U}\SN\in\Modhol(\SD_X)$ on $X$ has 
a quasi-normal form along it. Such a smooth compactification of $U$ 
always exists thanks to the theory of Mochizuki \cite{Moc11} 
(see also Hien \cite{Hi09} for similar use of Mochizuki's results). 
Let $D= \displaystyle\bigcup_{i \in I} D_i$ be the 
irreducible decomposition of $D$ and for each $i \in I$ define 
an open subset $D_i^{\circ} \subset D_i$ of $D_i$ by 
\begin{equation}
D_i^{\circ}:= D_i \setminus \Bigl( \bigcup_{j \not= i} D_j \Bigr) 
\qquad \subset D_i. 
\end{equation}
Then we have
\begin{equation}
\chi^{{\rm alg}}( \SN ) = {\rm rk} \SN \cdot \chi (U^{\an}) - 
\dsum_{i \in I} {\rm irr}_{D_i^{\circ}} ( \SM ) \cdot 
\chi ( (D_i^{\circ})^{\an} ), 
\end{equation}
where ${\rm rk} \SN \in \ZZ_{\geq 0}$ stands for the 
rank of the integrable connection $\SN$. 
\end{theorem}
In the case $\dim U=1$ this theorem was first obtained by 
Bloch and Esnault in \cite{BE04a}.

\section{Ginsburg type formulas for characteristic cycles}\label{sec-Gins}
In this section, for some standard 
holonomic D-modules, we define (not necessarily homogeneous) 
Lagrangian cycles that we call irregular characteristic cycles 
and use them to prove Ginsburg type 
formulas for their (usual) characteristic cycles 
similar to the one in Ginsburg \cite[Theorem 3.3]{Gin86}. 
First of all, we recall the definition of the 
irregular characteristic cycles introduced
by \cite{Tak22} and \cite{KT23} and reformulate the result 
of Corollary \ref{thm-CCquasi} in terms of them.
Let $X$ be a complex manifold, $D\subset X$ a normal crossing divisor in it and $\SM$
a holonomic $\SD_X$-module having a quasi-normal form along $D\subset X$.
Then for any point $x\in X$ there exists its neighborhood $U \subset X$ in $X$ 
for which we can define a (not necessarily homogeneous)
Lagrangian cycle $\CCirr(\SM)$ in the open subset $T^\ast U\subset T^\ast X$ as follows.
First, for a (sufficiently small) open sector $V\subset U\setminus D$ along the normal
crossing  divisor $D\cap U\subset U$ we take the exponential factors 
$f_1,\ldots,f_p\in\SP_{\varpi^{-1}(D\cap U)}^\prime$ of $\SM$ in the 
representative subsheaf $\SP_{\varpi^{-1}(D\cap U)}^\prime\subset 
\SP_{\varpi^{-1}(D\cap U)}$ (see Section \ref{sec:qNor}) which are holomorphic on $V$ and set
\begin{equation}
\Lambda(\SM,V)_i\coloneq
\Set*{(x,df_i(x))}{x\in V}
\quad \subset T^\ast V
\quad (1\leq i\leq p)
\end{equation}
and  
\begin{equation}
\CCirr(\SM,V)\coloneq
\sum_{i=1}^p N(f_i)\cdot[\Lambda(\SM,V)_i],
\end{equation}
where $N\colon\SP_{\varpi^{-1}(D\cap U)}^\prime\longrightarrow
(\ZZ_{\geq0})_{\varpi^{-1}(D\cap U)}$ is the multiplicity for 
which the enhanced ind-sheaf $\pi^{-1}\CC_U\otimes Sol_X^\rmE(\SM)\in \BEC(\rmI\CC_X)$
has a quasi-normal form along $D\cap U\subset U$.
Then $\CCirr(\SM,V)$ is a (not necessarily homogeneous) Lagrangian cycle in 
$T^\ast V\subset T^\ast X$.
Denote the generic rank of the meromorphic connection $\SM$ by $r\geq0$ and let
$h_1\ldots,h_r\in\SP_{\varpi^{-1}(D\cap U)}^\prime$ be the exponential factors of $\SM$
holomorphic on $V$ and counted with multiplicities.
Then obviously we can define $\CCirr(\SM,V)$ also by
\begin{equation}
\CCirr(\SM,V)\coloneq\sum_{i=1}^r
\Bigl[\Set*{(x,dh_i(x))}{x\in V}\Bigr].
\end{equation}
Moreover, shrinking $V$ if necessary, by the condition (ii) of 
the good sets of irregular values in Section \ref{sec:qNor} we may assume also that 
for any $i\neq j$ we have 
\begin{equation}
\Lambda(\SM,V)_i\cap\Lambda(\SM,V)_j=\emptyset.
\end{equation}
Shrink $U$ and cover $U\setminus D$ by such sectors $V\subset U\setminus D$.
Then by the proof of Proposition \ref{hda-dk}, we see that the Lagrangian cycles
$\CCirr(\SM,V)$ for various $V$ patch together to form the one $\CCirr(\SM,V)$ in 
$T^\ast(U\setminus D)\subset T^\ast U\subset T^\ast X$.
We call it the irregular characteristic cycle of $\SM$ (over $U\subset X$).
Now let $g\colon U\longrightarrow\CC$ be a defining holomorphic 
function of the normal crossing divisor $D\cap U\subset U$.
Then we obtain the following reformulation of Corollary \ref{thm-CCquasi} in terms of 
$\CCirr(\SM)$.

\begin{theorem}\label{thm-5-2-1}
In the open subset $T^\ast U\subset T^\ast X$ we have 
\begin{equation}
\CCyc(\SM)=\lim_{t\rightarrow+0}
t\Bigl\{\CCirr(\SM)+d\log g\Bigr\},
\end{equation}
where the limit in the right hand side stands for that of Lagrangian cycles
(for the definition, see \cite[Section 2.2]{FKT26}).
\end{theorem}

\begin{proof}
Let $(u,v)=(u_1,\ldots,u_l,v_1,\ldots,v_{n-l})$ be the coordinate of $U$ 
such that $x=(0,0)\in D\cap U=\{u_1\cdots u_l=0\}$ 
that we used to define the representative subsheaf 
$\SP_{\varpi^{-1}(D\cap U)}^\prime\subset \SP_{\varpi^{-1}(D\cap U)}$
and $h_1,\ldots,h_r\in \SP_{\varpi^{-1}(D\cap U)}^\prime$ the 
exponential factors of $\SM$ (which are multi-valued holomorphic 
functions on $U\setminus D$) counted with multiplicities.
For $1\leq j\leq r$ let $(k_{i1},\ldots,k_{il})\in(\QQ_{\geq0})^l$ be the pole 
order of $h_i$ along the normal crossing divisor $D\cap U=\{u_1\cdots u_l=0\}\subset U$.
First we consider the case where for any $1\leq i\leq r$ 
the function $h_i$ is meromorphic and hence we have
$(k_{i1},\ldots,k_{il})\in (\ZZ_{\geq0})^l$. 
Then by Corollary \ref{thm-CCquasi} for the complex submanifold 
$Y\coloneq\{u_1=\cdots=u_l=0\}\subset U$ of $U$ it suffices to check that 
the multiplicity of the conormal bundle $[T_Y^\ast U]$ in the limit cycle
\begin{equation}
\lim_{t\rightarrow+0} t\Bigl\{\CCirr(\SM)+d \log g\Bigr\}
\end{equation}  
is equal to 
\begin{equation}
\sum_{i=1}^r\biggl\{
\Bigl(\sum_{j=1}^l k_{ij}\Bigr)+1
\biggr\}
=\biggl(\sum_{j=1}^l\irr_{D_j^\circ}(\SM)\biggr)+r
\end{equation}
(for the definition of $D_j^\circ\subset D_j=\{u_j=0\}$ see Corollary \ref{thm-CCquasi}).
Indeed, we can calculate the multiplicities of the other conormal bundles similarly.
For $1\leq i\leq r$ we set 
\begin{equation}
\Theta(\SM)_i\coloneq\Set*{(x,dh_i(x))}{x\in U\setminus D} \quad \subset T^\ast U
\end{equation}
so that we have $\CCirr(\SM)=\displaystyle\sum_{i=1}^r[\Theta(\SM)_i]$.
Then it suffices to show that for any $1\leq i\leq r$ the multiplicity of the conormal
bundle $[T_Y^\ast U]$ in the limit cycle
\begin{equation}
\lim_{t\rightarrow+0} t\Bigl\{[\Theta(\SM)_i]+d \log g\Bigr\}
\end{equation}
is equal to $\displaystyle\sum_{j=1}^l k_{ij}+1$.
For this purpose, let us recall some elementary methods in toric geometry.
First, for an integer vector 
\begin{equation}
\vec{a}=
\begin{pmatrix}
a_1 \\
\vdots \\
a_l 
\end{pmatrix}
\quad \in\ZZ^l
\end{equation}
and a point $u=(u_1,\ldots,u_l)\in T\coloneq (\CC^\ast)^l$ we set 
$u^{\vec{a}}\coloneq u_1^{a_1}\cdots u_l^{a_l}\in\CC^\ast$.
Next, for an integer square matrix 
\begin{equation}
A=(\vec{a_1} \cdots \vec{a_l}) \quad \in M_l(\ZZ)
\end{equation}
with column vectors $\vec{a_1},\ldots,\vec{a_l}\in\ZZ^l$ we define a morphism
$\Phi_A\colon T\longrightarrow T$ by 
\begin{equation}
\Phi_A(u)\coloneq u^A\coloneq (u^{\vec{a_1}},u^{\vec{a_2}},\ldots,u^{\vec{a_l}})
\quad \in T=(\CC^\ast)^l.
\end{equation}
Then we can easily see that for any $A,B\in M_l(\ZZ)$
we have 
\begin{equation}
u^{AB}=(u^A)^B \quad (u\in T=(\CC^\ast)^l)
\end{equation}
and hence $\Phi_{AB}=\Phi_B\circ\Phi_A$.
This implies that for $A\in M_l(\ZZ)$ the morphism $\Phi_A\colon T\longrightarrow T$ is
an automorphism if and only if $A$ is unimodular i.e. $\det A=\pm1$.
For $A\in M_l(\ZZ)$ we define also a morphism $\Psi_A\colon 
T\times\CC^{n-l}\longrightarrow T\times\CC^{n-l}$ by $\Psi_A\coloneq \Phi_A\times \id_{\CC^{n-l}}$.
Now, for the pole order $(k_1,\ldots,k_l)\coloneq(k_{i1},\ldots,k_{il})\in(\ZZ_{\geq0})^l$ 
of the meromorphic function $h_i$ along $D\cap U\subset U$ we set
\begin{equation}
A\coloneq -
\begin{pNiceMatrix}
k_1+1 & k_1 & \Cdots & \Cdots & k_1 \\
k_2 & k_2+1 & & & \Vdots \\
\Vdots & k_3 & \Ddots & & \Vdots \\
& \Vdots & & \Ddots & k_{l-1} \\
k_l & k_l & \Hdotsfor{2} & k_l+1 \\
\end{pNiceMatrix}
\quad \in M_l(\ZZ).
\end{equation}
Then by a simple calculation, we can replace the complex Lagrangian submanifold 
$[\Theta(\SM)_i]+d\log g$ by the graph of the morphism
$\Psi_A\colon T\times\CC^{n-l}\longrightarrow T\times\CC^{n-l}$ associated to $A$.
Hence for a give point $\alpha\in T=(\CC^\ast)^l$ it suffices to calculate the number
of the solutions $u\in T=(\CC^\ast)^l$ of the equation $t\Phi_A(u)=\alpha$ that tend to
the origin $0\in\CC^l$ as $t\rightarrow+0$.
Since there exist unimodular matrices $B_1, B_2 \in M_l(\ZZ)$ such that 
\begin{equation}
B_1AB_2=
\begin{pNiceMatrix}
1 & & & & \\
& \Ddots & & & \text{\huge{0}} \\
& & \Ddots & & \\
& \text{\huge{0}} & & 1 & \\ 
& & & & k_1+\cdots+k_l+1 \\   
\end{pNiceMatrix}
,
\end{equation}
we see that $\Phi_A\colon T\longrightarrow T$ is an unramified covering of degree 
$k_1+\cdots+k_l+1=(\displaystyle\sum_{j=1}^l k_{ij})+1$.
Moreover, for the matrix
\begin{equation}
C\coloneq
\begin{pNiceMatrix}
\displaystyle -\sum_{j\neq1} k_j-1 & k_1 & \Cdots & k_1 \\
k_2 & \displaystyle  -\sum_{j\neq2} k_j-1 &  & k_2 \\
\Vdots & \Vdots & \Ddots & \Vdots \\
k_l & k_l & \Cdots & \displaystyle  -\sum_{j\neq l}k_j-1 \\   
\end{pNiceMatrix}
\quad \in M_l(\ZZ) 
\end{equation}
we have
\begin{equation}
AC=CA=(k_1+\cdots+k_l+1)\cdot E_l,
\end{equation}
where $E_l\in M_l(\ZZ)$ stands for the unit matrix of size $l$.
So, for a solution $u\in T=(\CC^\ast)^l$ of the equation 
\begin{equation}
t\Phi_A(u)=\alpha \quad \iff \quad
\Phi_A(u)=\frac{\alpha}{t}
\end{equation}
there exists a point $\beta=(\beta_1,\ldots,\beta_l)\in T=(\CC^\ast)^l$
such that 
\begin{align}
(u_1^{k_1+\cdots+k_l+1},\ldots \ldots ,u_l^{k_1+\cdots+k_l+1})&=\Phi_{AC}(u) \\
&=\Phi_C\biggl(\frac{\alpha}{t}\biggr) \\
&=(t\beta_1,\ldots,t\beta_l).
\end{align}
This implies that $u\in T=(\CC^\ast)^l$ tends to the origin $0\in\CC^l$ 
as $t\rightarrow+0$ and hence we obtain the assertion. 
Next we consider the general case. 
Let $\rho\colon X^\prime\longrightarrow U$ 
$\bigl((w,v)\longmapsto(u,v)=(w_1^{d_1},\ldots,w_l^{d_l},v)\bigr)$
be a ramification of $U$ along $D\cap U\subset U$ such that 
$h_i\circ\rho$ is a meromorphic function on $X^\prime$ along the normal crossing 
divisor $D^\prime\coloneq\rho^{-1}(D\cap U)\subset X^\prime$
and hence $(d_i k_{i1},\ldots,d_i k_{il})\in(\ZZ_{\geq0})^l$ for any $1\leq i\leq r$. 
By the morphism $\rho \times {\rm id}_{\CC^n}: 
X^\prime \times \CC^n \longrightarrow U \times \CC^n$ we 
take the pull-back $( \rho \times {\rm id}_{\CC^n})^* ( \CCirr(\SM)+d \log g )$ 
of the cycle $\CCirr(\SM)+d \log g$. 
Then we can similarly show that for the complex submanifold 
$Y^\prime\coloneq\{w_1=\cdots=w_l=0\}\subset X^\prime$ of $X^\prime$ the multiplicity
of the conormal bundle $[T_{Y^\prime}^\ast X^\prime]$ in the limit cycle 
\begin{equation}
\lim_{t\rightarrow+0}t \Bigl\{
( \rho \times {\rm id}_{\CC^n})^* ( \CCirr(\SM)+d \log g )
 \Bigr\}
\end{equation}
is equal to 
\begin{equation}
\sum_{i=1}^r d_1\cdots d_l\cdot \biggl\{\Bigl(\sum_{j=1}^l k_{ij}\Bigr)+1\biggr\}.
\end{equation}
As the degree of the covering 
$X^\prime\setminus D^\prime\longrightarrow U\setminus(U\cap D)$
induced by $\rho$ is equal to $d_1\cdots d_l$, this implies that 
the multiplicity of $[T_Y^\ast U]$ in the limit cycle 
\begin{equation}
\lim_{t\rightarrow+0}t\Bigl\{\CCirr(\SM)+d \log g\Bigr\}
\end{equation}   
is equal to 
\begin{align}
\frac{1}{d_1\cdots d_l}\cdot\sum_{i=1}^r d_1\cdots d_l\cdot
\biggl\{\Bigl(\sum_{j=1}^l k_{ij}\Bigr)+1\biggr\}
=\biggl(\sum_{j=1}^l \irr_{D_j^\circ}(\SM)\biggr)+r
\end{align}
as we expect from Corollary \ref{thm-CCquasi}.
This completes the proof.
\end{proof}

\begin{definition}\label{def-expmero}
Let $X$ be a complex manifold.
Then we say that a holonomic $\SD_X$-module $\SM$ is an exponentially twisted 
meromorphic connection if there exists a meromorphic function $f\in\SO_X(\ast Y_1)$
(resp. a regular meromorphic connection $\SN$) on $X$
along a closed hypersurface $Y_1\subset X$ (resp. $Y_2\subset X$) such that we have an 
isomorphism
\begin{equation}
\SM\simeq \SE_{X\setminus Y_1\vbar X}^f \Dotimes \SN.
\end{equation}
In this case, for the divisor $D\coloneq Y_1\cup Y_2$ we say also that $\SM$ 
is an exponentially twisted meromorphic connection along $D\subset X$.
\end{definition}

\begin{remark}
In the situation of Definition \ref{def-expmero}, for the open subset 
$U\coloneq X\setminus D\subset X$ there exists an isomorphism
\begin{equation}
\SM\simeq \SE_{U\vbar X}^{f\vbar_U} \Dotimes \Bigl(\SN(\ast D)\Bigr).
\end{equation}
\end{remark}

Let $\SM, f, \SN$ etc. be as in Definition \ref{def-expmero} and $r\geq0$ the generic
rank of the regular meromorphic connection $\SN$.
Then for the open subset $U\coloneq X\setminus D\subset X$ we define 
a (not necessarily homogeneous) Lagrangian cycle 
$\CCirr(\SM)$ in $T^\ast U\subset T^\ast X$ by
\begin{equation}
\CCirr(\SM)\coloneq \CCyc ( \SN |_U)+d(f|_U)= 
r\cdot \Bigl[\Set*{(x,df(x))}{x\in U}\Bigr].
\end{equation}
We call it the irregular characteristic cycle of $(\SM,f,\SN)$.
Note that if $X$ is not compact it depends not only on $\SM$ but also on
the decomposition $\SM\simeq \SE_{X\setminus Y_1\vbar X}^f\Dotimes\SN$ of $\SM$.

\begin{theorem}\label{thm-5-2-21}
Let $\SM, f, \SN$ etc. be as in Definition \ref{def-expmero} and 
$g\colon X\longrightarrow\CC$ a (local) defining holomorphic function of 
the divisor $D\subset X$. Then we have
\begin{equation}\label{eq-CCirr}
\CCyc(\SM)=\lim_{t\rightarrow+0}t\Bigl\{\CCirr(\SM)+d \log g\Bigr\}.
\end{equation}
\end{theorem}

\begin{proof}
Let $\nu\colon X^\prime\longrightarrow X$ be a projective morphism of complex
manifolds inducing an isomorphism 
$X^\prime\setminus \nu^{-1}(D)\simto X\setminus D$ such that 
$D^\prime \coloneq \nu^{-1}(D)\subset X^\prime$ is a normal crossing divisor and the
meromorphic function $f\circ\nu$ on $X^\prime$ along $D^\prime\subset X^\prime$ has no 
point indeterminacy.
Then the meromorphic connection $\bfD\nu^\ast\SM$ on $X^\prime$ has a normal form
along the normal crossing divisor $D^\prime \subset X^\prime$ and hence by 
the proof of Theorem \ref{thm-5-2-1} we obtain an equality 
\begin{equation}
\CCyc(\bfD\nu^\ast \SM)=
\lim_{t\rightarrow+0} t\Bigl\{\CCirr(\bfD\nu^\ast\SM)+ d\log(g\circ\nu)\Bigr\}.
\end{equation}
Note that $\bfD\nu^\ast\SM$ is an exponentially twisted meromorphic connection
on $X^\prime$ along $D^\prime\subset X^\prime$ and its irregular characteristic cycle
$\CCirr( \bfD\nu^\ast \SM)$ is naturally identified with 
$\CCirr( \SM)$ via the isomorphism 
$T^\ast(X^\prime\setminus D^\prime)\simeq T^\ast(X\setminus D)$.
Moreover, by the isomorphisms
\begin{equation}
Sol_X(\SM) \simeq Sol_X(\bfD\nu_\ast(\bfD\nu^\ast\SM)) 
\simeq \rmR\nu_\ast Sol_{X^\prime}(\bfD\nu^\ast\SM)
\end{equation}
we see also that the characteristic cycle $\CCyc(\SM)=\CCyc(Sol_X(\SM))$ of $\SM$ is the 
push-forward of the Lagrangian cycle 
$\CCyc(\bfD\nu^\ast\SM)=\CCyc(Sol_{X^\prime}(\bfD\nu^\ast\SM))$ 
by $\nu$ (see Kashiwara-Schapira \cite[Chapter IX]{KS90}).
Then by \cite[Proposition 2.8]{FKT26} we obtain 
the desired equality \eqref{eq-CCirr} as follows:
\begin{align}
\CCyc(\SM)&=\mu_\ast(\nu) \CCyc(\bfD\nu^\ast\SM) \\
&=\mu_\ast(\nu)\biggl[\lim_{t\rightarrow+0}t\Bigl\{
\CCirr(\bfD\nu^\ast\SM)+ d\log(g\circ\nu)
\Bigr\}\biggr] \\
&=\lim_{t\rightarrow+0}t\Bigl\{\CCirr(\SM)+d\log g\Bigr\}.
\end{align}
This completes the proof.
\end{proof}

\begin{example}
\begin{enumerate}[wide,labelwidth=!,labelindent=0pt]
\item [\rm (i)]
Let us consider the situation in Example \ref{ex-CC} (ii).
Let $g(x,y)\coloneq x$ be a defining holomorphic function of 
the divisor $D=\{x=0\}\subset X$. 
Then we can show that the coefficient of $[T_{\{0\}}^\ast X]$ 
in the limit cycle
\begin{equation}
\lim_{t\rightarrow +0} t\Bigl\{\CCirr(\SM)+ d\log g\Bigr\}
\end{equation}
is equal to $1$ as follows. For this purpose, 
for generic $(\alpha,\beta)\in(\CC^\ast)^2\subset\CC^2$ and $0<t\ll1$ 
we solve the equations  
\begin{align}
&t\Bigl\{d\varphi+ d\log g\Bigr\}(x,y) \\
&=t
\begin{Bmatrix}
\begin{pmatrix}
\displaystyle -\frac{y^k}{x^2} \\
\\
\displaystyle \frac{ky^{k-1}}{x}
\end{pmatrix}
+
\begin{pmatrix}
\displaystyle \frac{1}{x} \\
\\
0
\end{pmatrix}
\end{Bmatrix}
=
\begin{pmatrix}
\alpha \\
\\
\beta
\end{pmatrix}
\\
&\iff
\begin{dcases}
t\cdot \frac{x-y^k}{x^2} &=\alpha, \\
t\cdot \frac{ky^{k-1}}{x} &=\beta.
\end{dcases}
\label{eq-53ex1-1}
\end{align}
Then from the second equation 
in \eqref{eq-53ex1-1} 
we deduce $\displaystyle x=\frac{kt}{\beta}y^{k-1}$.
Substituting it into the first equation, we obtain
\begin{equation}
t\cdot\cfrac{\cfrac{kt}{\beta}y^{k-1}-y^k}{\cfrac{k^2t^2}{\beta^2}\cdot y^{2k-2}}
=\alpha  \qquad \iff \qquad 
\biggl(\frac{\beta^2}{k\beta-k^2\alpha y^{k-1}}\biggr)
\cdot y=t. \label{eq-53ex1-2}
\end{equation}
Since for the meromorphic function
\begin{equation}
h(y)\coloneq \frac{\beta^2}{k\beta -k^2\alpha y^{k-1}}
\end{equation}
of $y$ we have $h(0)\neq0$, the equation \eqref{eq-53ex1-2} 
of $y\in\CC$ has only one solution near the origin.
Moreover, the unique solution 
$\displaystyle (x,y)=\Bigl(\frac{kt}{\beta}y^{k-1},y\Bigr)$ of \eqref{eq-53ex1-1}
thus obtained tends to the $0=(0,0)\in X=\CC^2$ as $t\rightarrow+0$.
Now our claim is clear.
Similarly, we can show that 
\begin{align}
&\lim_{t\rightarrow +0} t\Bigl\{\CCirr(\SM)+ d\log g\Bigr\} \\
&= 1\cdot[T_X^\ast X] +2\cdot[T_D^\ast X] +1\cdot[T_{\{0\}}^\ast X]
=\CCyc(\SM).
\end{align}
\item [\rm (ii)]
Let us consider the situation in Example \ref{ex-CC} (iii).
Let $g(x,y)\coloneq x$ be a defining holomorphic function of the divisor $D=\{x=0\}
\subset X$.
Then we can show that the coefficient of $[T_{\{0\}}^\ast X]$ in the limit cycle 
\begin{equation}
\lim_{t\rightarrow +0} t\Bigl\{\CCirr(\SM)+ d\log g\Bigr\}
\end{equation}
is equal to $k>0$ as follows. For this purpose, for generic 
$(\alpha,\beta)\in(\CC^\ast)^2\subset\CC$ and $0<t\ll1$ we solve the equations
\begin{align}
&t\Bigl\{d\varphi+ d\log g\Bigr\}(x,y) \\
&=t
\begin{Bmatrix}
\begin{pmatrix}
\displaystyle \frac{-ky}{x^{k+1}} \\
\\
\displaystyle \frac{1}{x^k}
\end{pmatrix}
+
\begin{pmatrix}
\displaystyle \frac{1}{x} \\
\\
0
\end{pmatrix}
\end{Bmatrix}
=
\begin{pmatrix}
\alpha \\
\\
\beta
\end{pmatrix}
\\
&\iff
\begin{dcases}
t\cdot \frac{x^k-ky}{x^{k+1}} &=\alpha, \\
t\cdot \frac{1}{x^k} &=\beta.
\end{dcases}
\label{eq-53ex2}
\end{align}
Then the second equation of $x\in\CC$ has $k$ solutions near 
the origin $0\in\CC$. Moreover, they tend to the origin $0\in\CC$ as $t\rightarrow+0$.
We denote one of them by $\displaystyle\sqrt[\leftroot{-2}\uproot{10}k]{\frac{t}{\beta}}$ 
for short and substitute it into the first equation of \eqref{eq-53ex2}. Then we obtain 
\begin{equation}
t\cdot\cfrac{\cfrac{t}{\beta}-ky}{\sqrt[\leftroot{-2}\uproot{10}k]
{\cfrac{t}{\beta}}\cdot\cfrac{t}{\beta}}
=\alpha 
\qquad \iff \qquad 
y= \frac{1}{k\beta} 
\biggl(t-\alpha\cdot\sqrt[\leftroot{-2}\uproot{10}k]{\frac{t}{\beta}}\biggr).
\end{equation}
Since the $k$ solutions $(x,y)\in X=\CC^2$ of \eqref{eq-53ex2} thus obtained tend to 
the origin $0=(0,0)\in X=\CC^2$ as $t\rightarrow+0$, we verify our claim.
Similarly, we can show that 
\begin{align}
&\lim_{t\rightarrow +0} t\Bigl\{\CCirr(\SM)+ d\log g\Bigr\} \\
&= 1\cdot[T_X^\ast X] +(k+1)\cdot[T_D^\ast X] +k\cdot[T_{\{0\}}^\ast X]
=\CCyc(\SM).
\end{align}   
\item [\rm (iii)]
Let us consider the situation in Example \ref{ex-CC} (iv).
Let $g(x,y)\coloneq x^2-y^3$ be a defining holomorphic function 
of the divisor $D=\{x^2-y^3=0\}
\subset X$.
Then we can show that the coefficient of $[T_{\{0\}}^\ast X]$ in the limit cycle 
\begin{equation}
\lim_{t\rightarrow +0} t\Bigl\{\CCirr(\SM)+ d\log g\Bigr\}
\end{equation}
is equal to $4$ as follows. For this purpose, for generic 
$(\alpha,\beta)\in(\CC^\ast)^2\subset\CC$ and $0<t\ll1$ we solve the equations
\begin{align}
&t\Bigl\{d\varphi+ d\log g\Bigr\}(x,y) \\
&=t
\begin{Bmatrix}
\dfrac{1}{(x^2-y^3)^2}
\begin{pmatrix}
2x \\
\\
-3y^2
\end{pmatrix}
+ \dfrac{1}{x^2-y^3}
\begin{pmatrix}
2x \\
\\
-3y^2
\end{pmatrix}
\end{Bmatrix}
=
\begin{pmatrix}
-2\alpha \\
\\
3\beta
\end{pmatrix}
\\
&\iff
\begin{dcases}
tx \cdot \frac{1-(x^2-y^3)}{(x^2-y^3)^2} &=\alpha, \\
ty^2 \cdot \frac{1-(x^2-y^3)}{(x^2-y^3)^2} &=\beta.
\end{dcases}
\label{eq-53ex3-1}
\end{align}
Taking the ratio of the two equations above, we find
\begin{equation}
\frac{x}{y^2}=\frac{\alpha}{\beta} \quad
\iff \quad x=\frac{\alpha}{\beta} y^2.
\end{equation}
Substituting it into the second equation of \eqref{eq-53ex3-1}, we obtain
\begin{equation}\label{eq-53ex3-2}
\Biggl\{
\frac{(\alpha^2y-\beta^2)^2}{\beta^3-(\alpha^2\beta y^4-\beta^3y^3)}
\Biggr\}
\cdot y^4 =t.
\end{equation}
Since for the meromorphic function
\begin{equation}
h(y)\coloneq
\frac{(\alpha^2y-\beta^2)^2}{\beta^3-(\alpha^2\beta y^4-\beta^3y^3)}
\end{equation}
of $y$ we have $h(0)\neq0$, the equation \eqref{eq-53ex3-2} of $y\in\CC$
has exactly $4$ solutions near the origin $0\in\CC$.
Moreover the $4$ solutions 
$(x,y)=\Bigl(\dfrac{\alpha}{\beta}y^2,y\Bigr)$ of 
\eqref{eq-53ex3-1} thus obtained tend to the origin
$0=(0,0)\in X=\CC^2$ as $t \longrightarrow +0$. 
Now our claim is clear.
Similarly, we can show that 
\begin{align}
&\lim_{t\rightarrow +0} t\Bigl\{\CCirr(\SM)+ d\log g\Bigr\} \\
&= 1\cdot[T_X^\ast X] +2\cdot[\var{T_{D_\reg}^\ast X}] +4\cdot[T_{\{0\}}^\ast X]
=\CCyc(\SM)
\end{align}
in this case.   
\end{enumerate}
\end{example}

We can generalize Theorem \ref{thm-5-2-21} as follows.
Let $f\in\SO_X(\ast Y_1)$ be a meromorphic function on $X$ along a closed
hypersurface $Y_1\subset X$ and $\SN$ a regular holonomic $\SD_X$-modules whose 
support $Z\coloneq\supp\SN\subset X$ is irreducible such that there exists a closed
hypersurface $Y_2\subset X$ satisfying the conditions:
\begin{enumerate}[noitemsep]
\item[\rm (i)]  $Z\setminus Y_2$ is smooth and connected,
\item[\rm (ii)] On a neighborhood of the complex manifold $Z\setminus Y_2$ in $X$
the regular holonomic $\SD_X$-modules $\SN$ is a direct image of an integrable
connection $\SN_\red$ on $Z\setminus Y_2$, 
\item[\rm (iii)] $\SN \simto \SN (* Y_2)$. 
\end{enumerate}
Then we set
\begin{equation}
\SM\coloneq \SE_{X\setminus Y_1\vbar X}^f\Dotimes \SN
\quad \in\Modhol(\SD_X).
\end{equation}
Let $r\geq0$ be the rank of the integrable connection $\SN_\red$ and set 
$D\coloneq Y_1\cup Y_2\subset X$.
Then for the open subset $U\coloneq X\setminus D$ we define a 
(not necessarily homogeneous) Lagrangian cycle $\CCirr(\SM)$ in 
$T^\ast U\subset T^\ast X$ by 
\begin{equation}
\CCirr(\SM)\coloneq 
\CCyc( \SN |_U)+d(f|_U)= 
r\cdot\Bigl\{[T_{Z\cap U}^\ast U]+d(f|_U) \Bigr\}.
\end{equation}
Also for such an irregular holonomic $\SD_X$-module $\SM$ we can prove the following
result by using a resolution of singularities of $Z=\supp\SN\subset X$ 
as in the proof of Theorem \ref{thm-5-2-21}. 

\begin{theorem}\label{thm-5-2-3}
In the situation as above, let $g\colon X\longrightarrow\CC$
be a (local) defining holomorphic function of the divisor $D\subset X$.
Then we have
\begin{equation}
\CCyc(\SM)=\lim_{t\rightarrow+0}t\Bigl\{\CCirr(\SM)+d\log g\Bigr\}.
\end{equation}
\end{theorem}

\begin{definition}\label{def-expD} 
(cf. \cite{Tak22}) Let $X$ be a complex manifold.
Then we say that a holonomic $\SD_X$-module $\SM$ is an exponentially twisted holonomic
$\SD_X$-module if there exist a meromorphic function $f\in\SO_X(\ast Y)$ along a
closed hypersurface $Y\subset X$ and a regular holonomic $\SD_X$-module $\SN$ such 
that we have an isomorphism
\begin{equation}
\SM\simeq \SE_{X\setminus Y\vbar X}^f\Dotimes\SN.
\end{equation}
\end{definition}

For the exponentially twisted holonomic $\SD_X$-module $\SM$ in Definition 
\ref{def-expD} we define a (not necessarily homogeneous) Lagrangian cycle $\CCirr(\SM)$
in $T^\ast (X \setminus Y) \subset T^\ast X$ by
\begin{equation}
\CCirr(\SM)\coloneq \CCyc(\SN\vbar_{X \setminus Y}) +df.
\end{equation} 
We call it the irregular characteristic cycle of $\SM$.
Then we can prove the following result.

\begin{theorem}\label{thm-twistD}
Let $\SM, f, \SN$ etc. be as in Definition \ref{def-expD} and 
$g\colon X\longrightarrow\CC$ a (local) defining holomorphic 
function of the divisor $Y\subset X$.
Then we have
\begin{equation}
\CCyc(\SM)=\lim_{t\rightarrow+0}t\Bigl\{\CCirr(\SM)+d\log g\Bigr\}.
\end{equation}
\end{theorem}

\begin{proof}
By the additivity of the operation of taking 
characteristic cycles, we can decompose the support of $\SN$ by 
the algebraic local cohomology functors (see Kashiwara 
\cite[Section 3.4]{Kas03}) and 
reduce the problem to the case of 
$\SM\simeq\SE_{X\setminus Y_1\vbar X}^f\Dotimes \SN$ in Theorem \ref{thm-5-2-3}. 
So we use the notations in it to have $Y=Y_1$. 
For $i=1,2$ let $g_i\colon X\longrightarrow\CC$ be
the (local) defining holomorphic function of $Y_i\subset X$
so that we have $g=g_1$. 
Then by Theorem \ref{thm-5-2-3} we obtain
\begin{equation}
\CCyc(\SN\vbar_{X\setminus Y_1})
=\lim_{s\rightarrow+0}\Bigl\{\CCyc(\SN\vbar_{X\setminus(Y_1\cup Y_2)})+s d\log g_2\Bigr\}
\end{equation}
(see also Ginsburg \cite[Theorem 3.3]{Gin86}). 
First, we consider the case where $Y_1\cup Y_2\subset X$ is a normal crossing divisor,
the meromorphic function $f\in\SO_X(\ast Y_1)$ has no point of indeterminacy on 
the whole $X$ and $\SN=\SO_X(\ast Y_2)$.
The problem being local, by taking a suitable local coordinate $x=(x_1,\dots,x_n)$ of 
$X$ we may assume that 
\begin{equation}
g_1(x)=\tl{g_1}(x)\cdot \prod_{i\in I_A}x_i^{m_i}, \quad
g_2(x)=\tl{g_2}(x)\cdot \prod_{i\in I_B}x_i^{m_i^\prime}, \quad
f(x)=\prod_{i\in I_P}x_i^{-k_i}
\end{equation}
for some subsets $I_A,I_B,I_P\subset \{1,2,\dots,n\}$ and positive integers $m_i>0$
($i\in I_A$), $m_i^\prime>0$ ($i\in I_B$), $k_i>0$ ($i\in I_P$), 
where $\tl{g_1}(x),\tl{g_2}(x)\neq0$ are invertible holomorphic functions.
In this situation, our primary goal is to show 
\begin{align}
&\lim_{t\to+0} t\Bigl\{\CCirr(\SM)+d\log g_1\Bigr\} \\
&=\lim_{t\to+0} \Bigl\{\CCyc(\SO_X(\ast Y_2)\vbar_{X\setminus Y_1})+tdf+td\log g_1\Bigr\} \\
&=\lim_{t\to+0} \Bigl\{\CCyc(\SO_{X\setminus(Y_1\cup Y_2)})+tdf +td\log g_1+td\log g_2\Bigr\} \\
&=\CCyc(\SE_{X\setminus(Y_1\cup Y_2)\vbar X}^{f\vbar_{X\setminus(Y_1\cup Y_2)}}) =\CCyc(\SM).
\end{align}
For this purpose, as $td\log g_1+td\log g_2=td\log(g_1g_2)$, we may replace $g_1$ and 
$g_2$ by
\begin{align}
\Bigl(\prod_{i\in I_A}x_i^{m_i}\Bigr)\cdot 
\Bigl(\prod_{i\in I_A\cap I_B}x_i^{m_i^\prime}\Bigr)
\quad 
\text{and}
\quad
\prod_{i\in I_B\setminus I_A}x_i^{m_i^\prime}
\end{align}
respectively and assume that $I_A\cap I_B=\emptyset$ from the first.
Then we can easily see that for any $t>0$ and $s>0$ the closure of the complex Lagrangian
manifold 
\begin{align}
&\Lambda_{t,s}\coloneq \CCyc(\SO_{X\setminus(Y_1\cup Y_2)}) +t(df+d\log g_1)+sd\log g_2 \\
&=[T_{X\setminus(Y_1\cup Y_2)}^\ast(X\setminus(Y_1\cup Y_2))] +t(df+d\log g_1)+sd\log g_2
\end{align}
in $T^\ast X$ is contained in the open subset 
$T^\ast(X\setminus(Y_1\cup Y_2))\subset T^\ast X$.
Recall that we have 
\begin{align}
\CCyc(\SO_X(\ast Y_2))&=\lim_{s\to+0}\Bigl\{[T_{X\setminus Y_2}^\ast
(X\setminus Y_2)]+ sd\log g_2\Bigr\} \\
&=[T_X^\ast X] +\sum_{i\in I_B}[T_{D_i}^\ast X] 
+\sum_{\substack{i,i^\prime\in I_B \\ i<i^\prime}}[T_{D_i\cap D_{i^\prime}}^\ast X]
+\cdots \cdots ,
\end{align}
where for $1\leq i\leq n$ we set $D_i\coloneq \{x_i=0\}\subset X$.
On the other hand, by Corollary \ref{thm-CCquasi} and Theorem \ref{thm-5-2-21}, we obtain
\begin{align}
&\lim_{t\to+0}\Bigl\{[T_X^\ast X]\vbar_{X\setminus Y_1}+tdf +td\log g_1\Bigr\} \\
&=\CCyc(\SE_{X\setminus Y_1\vbar X}^f) \\
&=[T_X^\ast X]+\sum_{j\in I_A}(k_j+1)\cdot[T_{D_j}^\ast X] 
+\sum_{\substack{j,j^\prime\in I_A \\ j<j^\prime}}(k_j+k_{j^\prime}+1)\cdot[
T_{D_j\cap D_{j^\prime}}^\ast X] +\cdots \cdots.
\end{align}
Similarly, for any $i_1,i_2,\dots,i_q\in I_B$ such that $i_1<i_2<\dots<i_q$ we can 
easily show that 
\begin{align}
&\lim_{t\to+0}\Bigl\{[T_{D_{i_1}\cap\dots\cap D_{i_q}}^\ast X]\vbar_{
X\setminus Y_1}+tdf +td\log g_1\Bigr\} \\
&=\sum_{p=0}^{\abs{I_A}}\sum_{\substack{j_1,\dots,j_p\in I_A \\ j_1<\dots<j_p}}
(k_{j_1}+\dots +k_{j_p}+1)\cdot[T_{D_{i_1}\cap \cdots\cap D_{i_q}\cap 
D_{j_1}\cap \cdots \cap D_{j_p}}^\ast X].
\end{align}
Then we finish the proof by observing that the formula 
\begin{align}
&\lim_{t\to+0}\Bigl\{\CCyc(\SO_X(\ast Y_2)\vbar_{X\setminus Y_1})+tdf +td\log g_1\Bigr\} \\
&=\sum_{p=0}^{\abs{I_A}}\sum_{q=0}^{\abs{I_B}}
\sum_{\substack{j_1,\dots,j_p\in I_A \\ j_1<\dots<j_p}}
\sum_{\substack{i_1,\dots,i_q\in I_B \\ i_1<\dots<i_q}}
(k_{j_1}+\cdots+k_{j_p}+1)\cdot 
[T_{D_{i_1}\cap \cdots \cap D_{i_q}\cap D_{j_1}\cap \cdots \cap D_{j_p}}^\ast X]
\end{align}
we thus obtain coincides with that of $\CCyc(\SE_{X\setminus(Y_1\cup Y_2)\vbar 
X}^{f\vbar_{X\setminus(Y_1\cup Y_2)}})$ obtained by Corollary \ref{thm-CCquasi}. 
Note that in the above situation 
we may replace $\SN=\SO_X(\ast Y_2)$ by any meromorphic 
connection on $X$ along the normal crossing divisor 
$Y_2 \subset X$. Namely the assertion holds true for 
such meromorphic connections. 
Next, we consider the general case.
Let $\nu\colon \tl{X}\longrightarrow X$ be a proper surjective morphism of 
complex manifolds inducing an isomorphism 
$\tl{X}\setminus \nu^{-1}(Y_1\cup Y_2)\simto X\setminus (Y_1\cup Y_2)$ such that 
the proper transform $\tl{Z}$ of $Z$ in $\tl{X}$ is smooth, the meromorphic
function $f\circ \nu$ on $\tl{X}$ has no point of indeterminacy on the whole 
$\tl{X}$ and for $D_1\coloneq \nu^{-1}(Y_1)\cap \tl{Z}$, 
$D_2\coloneq \nu^{-1}(Y_2)\cap \tl{Z}$ the divisor $D_1\cup D_2\subset \tl{Z}$ in 
$\tl{Z}$ is normal crossing.
In this situation, there exists an isomorphism
\begin{equation}
\SN(\ast Y_1)\simeq \bfD \nu_\ast \bfD\nu^\ast \SN(\ast Y_1)
\end{equation}
and for the proof we may replace $\SN$ by $\SN(\ast Y_1)$.
Moreover, by Kashiwara's equivalence (see \cite[Theorem 4.30]{Kas03}) the 
holonomic $D_{\tl{X}}$-module $\SN^\prime\coloneq \bfD\nu^\ast \SN(\ast Y_1)$ 
supported by the complex submanifold $\tl{Z}\subset \tl{X}$ corresponds to a 
meromorphic connection $\tl{\SN}$ on $\tl{Z}$ along 
the normal crossing divisor $D_1\cup D_2\subset \tl{Z}$.
Then by the first part of the proof we obtain an equality 
\begin{align}
&\lim_{t\to+0}\Bigl\{\CCyc(\tl{\SN}\vbar_{\tl{Z}\setminus D_1})+td(f\circ
\nu\vbar_{\tl{Z}})+td\log(g_1\circ\nu\vbar_{\tl{Z}})\Bigr\} \\
&=\lim_{t\to+0}\Bigl\{
\CCyc(\tl{\SN}\vbar_{\tl{Z}\setminus(D_1\cup D_2)})+td(f\circ\nu\vbar_{\tl{Z}})
+td\log(g_1\circ\nu\vbar_{\tl{Z}})+td\log(g_2\circ\nu\vbar_{\tl{Z}})
\Bigr\}.
\end{align}
This implies that we have 
\begin{align}
&\lim_{t\to+0}\Bigl\{\CCyc(\SN^\prime\vbar_{\tl{X}\setminus \nu^{-1}(Y_1)})
+td(f\circ\nu )+td\log(g_1\circ\nu )\Bigr\} \\
&=\lim_{t\to+0}\Bigl\{
\CCyc(\SN^\prime\vbar_{\tl{X}\setminus\nu^{-1}(Y_1\cup Y_2)})+td(f\circ\nu
)+td\log(g_1\circ\nu )+td\log(g_2\circ\nu )
\Bigr\}.
\end{align}
Then taking the direct images of the Lagrangian cycles on the both sides by 
the proper morphism $\nu\colon\tl{X}\longrightarrow X$, by  
\cite[Proposition 2.8]{FKT26} we obtain the assertion.
Here we used the fact that for any $t>0$ and $s>0$ the closure of the support of the 
Lagrangian cycle 
\begin{equation}
\CCyc(\SN^\prime\vbar_{\tl{X}\setminus \nu^{-1}(Y_1\cup Y_2)})
+t\bigl(d(f\circ\nu)+d\log(g_1\circ\nu)\bigr)+sd\log(g_2\circ\nu)
\end{equation} 
in $T^\ast\tl{X}$ is contained in the open subset $T^\ast(\tl{X}\setminus 
\nu^{-1}(Y_1\cup Y_2))\subset T^\ast\tl{X}$ and $\nu$ induces an isomorphism 
$\tl{X}\setminus \nu^{-1}(Y_1\cup Y_2)\simto X\setminus (Y_1\cup Y_2)$.
This completes the proof. 
\end{proof}

\begin{remark}
For some $b>0$ we set $I:=(0,b)$ and $J:=[0,b)$. 
Then the complex Lagrangian submanifolds $\Lambda_{t,s} 
\subset T^*X$ $((t,s) \in I^2)$ in the proof of 
Theorem \ref{thm-twistD} form a family of Lagrangian 
analytic subsets of $T^*X$ over the set $I^2$ i.e. 
a closed subanalytic subset $A_{I^2}$ of 
$T^*X \times I^2$ such that $A_{I^2} \cap ( T^*X \times \{ (t,s) \} ) 
\subset T^*X \times \{ (t,s) \} \simeq T^*X$ is a complex Lagrangian 
analytic subset of $T^*X$ for any $(t,s) \in I^2$. 
We can easily show that it can be extended to a 
family of Lagrangian analytic subsets of $T^*X$ over $J^2$. 
Then we can prove Theorem \ref{thm-twistD} 
by a slight modification of \cite[Lemma 2.5]{FKT26}.  
\end{remark}

\begin{example}
Consider the case where $X$ is the complex vector space $\CC^3$ of dimension $3$ endowed
with the standard coordinate $z=(z_1,z_2,z_3)=(x,y,z)$. Set $Y\coloneq \{x=0\}\subset X$,
$H\coloneq \{z=0\}\subset X$ and $Z\coloneq \{x^2+y^2+z^2=0\}\subset 
X$ and let $\SN$ be a regular holonomic $\SD_X$-module such that 
\begin{equation}
\supp\SN= Z, \quad \SN(\ast H)\simeq \SN
\end{equation}
and on a neighborhood of the complex submanifold $Z\setminus H\subset X$ 
in $X=\CC^3$ it is the direct image of an integrable connection $\SN_\red$ of rank one 
on $Z\setminus H$.
Then for the meromorphic function $f(x,y,z)=\displaystyle\frac{1}{x}$ along 
$Y=\{x=0\}\subset X$ we define an exponentially twisted holonomic $\SD_X$-module $\SM$ by
\begin{equation}
\SM\coloneq \SE_{X\setminus Y\vbar X}^f\Dotimes \SN.
\end{equation}
For the defining function $g(x,y,z)=x$ of $Y\subset X$, let us calculate the limit
\begin{align}
\CCyc(\SM)&=\lim_{t\to+0}t\Bigl\{\CCirr(\SM)+d\log g\Bigr\} \\
&=\lim_{t\to+0}\Bigl\{\CCyc(\SN\vbar_{X\setminus Y})+t(df+d\log g)\Bigr\}
\end{align}
in Theorem \ref{thm-twistD}.
Set $L_\pm\coloneq \Set*{(x,y,z)\in X}{z=0,x=\pm\sqrt{-1}y}\simeq \CC$ 
(resp. $K_\pm\coloneq \{(x,y,z)\in X\mid x=0,y=\pm\sqrt{-1}z\}\simeq \CC$)
so that we have $Z\cap H=L_+ \cup L_-$ (resp. $Z\cap Y=K_+\cup K_-$).
Then it is easy to see that in the open subset $T^\ast(X\setminus Y)\subset T^\ast X$
we have
\begin{align}
\CCyc(\SN\vbar_{X\setminus Y})&=[T_{Z\setminus Y}^\ast(X\setminus Y)] 
+[T_{L_+\setminus\{0\}}^\ast(X\setminus\{0\})]
+[T_{L_-\setminus\{0\}}^\ast(X\setminus\{0\})]
\end{align}
and 
\begin{align}
\lim_{t\to+0}\Bigl\{[T_{L_\pm\setminus\{0\}}^\ast(X\setminus\{0\})]+t(df+d\log g)\Bigr\}
=[T_{L_\pm}^\ast X] + 2\cdot[T_{\{0\}}^\ast X].
\end{align}
Moreover on the open subset $T^\ast(X\setminus H)\subset T^\ast H$ we obtain 
\begin{align}
&\lim_{t\to+0}\Bigl\{[T_{Z\setminus Y}^\ast(X\setminus Y)]+t(df+d\log g)\Bigr\} \\
&=[T_{Z\setminus H}^\ast (X\setminus H)] 
+ 2\cdot[T_{{K_+\setminus\{0\}}}^\ast(X\setminus \{0\})] 
+ 2\cdot[T_{{K_-\setminus\{0\}}}^\ast(X\setminus \{0\})].
\end{align}
Hence it remains for us to calculate the multiplicity of $[T_{\{0\}}^\ast X]$ in the limit
\begin{equation}
\lim_{t\to+0}\Bigl\{[T_{Z\setminus Y}^\ast(X\setminus Y)]+t(df+d\log g)\Bigr\}.
\end{equation}
For this purpose, we parametrize the conormal bundle 
$T_{Z\setminus Y}^\ast(X\setminus Y)$ as follows:
\begin{equation}
T_{Z\setminus Y}^\ast (X\setminus Y) =
\Set*{(x,y,z; \lambda x,\lambda y, \lambda z)}{\lambda\in\CC, x\neq0, x^2+y^2+z^2=0}.
\end{equation}
Then for each point $(\alpha,\beta,\gamma)\in (\CC^\ast)^3\subset\CC^3=T_{\{0\}}^\ast X$
we solve the equations
\begin{equation}
\begin{dcases}
\lambda x-\frac{t}{x^2}+\frac{t}{x}=\alpha \\
\lambda y=\beta \\
\lambda z=\gamma      
\end{dcases}
\label{eq-lambda}
\end{equation}
for $(x,y,z)\in Z\setminus Y$.
By the second and the third equations of \eqref{eq-lambda} and using the condition
$y^2+z^2=-x^2$ we obtain 
\begin{equation}
\beta^2+\gamma^2=\lambda^2(y^2+z^2)=-\lambda^2 x^2
\end{equation}
and hence
\begin{equation}\label{eq-lamx}
\lambda x=\pm\sqrt{-(\beta^2+\gamma^2)}.
\end{equation}
Moreover by the condition $x\neq0$ we see that the first equation of 
\eqref{eq-lambda} is equivalent to 
\begin{equation}\label{eq-alphax}
\alpha x^2-\lambda x^3 -tx+t=0.
\end{equation}
Substituting \eqref{eq-lamx} into \eqref{eq-alphax} we get also 
\begin{equation}
(\alpha\mp \sqrt{-(\beta^2+\gamma^2)})\cdot x^2-tx+t=0.
\end{equation}
If $(\alpha,\beta,\gamma)\in (\CC^\ast)^3$ satisfies the condition 
$\alpha^2+\beta^2+\gamma^2\neq0$ and $0<t\ll1$, then this equation of $x\in\CC$
has two solutions in $\CC^\ast\subset \CC$ and each of them tends to 
$0\in\CC$ as $t\longrightarrow+0$.
For such a solution $x\in\CC^\ast$ the two complex numbers 
\begin{equation}
\lambda=\frac{\pm\sqrt{-(\beta^2+\gamma^2)}}{x}\neq0
\end{equation}
obtained by \eqref{eq-lamx} go to infinity as $t\longrightarrow+0$.
By the second and the third equations of \eqref{eq-lambda}, this implies that if 
$(\alpha,\beta,\gamma)\in (\CC^\ast)^3$ satisfies the condition 
$\alpha^2+\beta^2+\gamma^2\neq0$ and $0<t\ll1$ the equations \eqref{eq-lambda} 
for $(x,y,z)\in Z\setminus Y$ have exactly four solutions and each of them tends 
to the origin $0=(0,0,0)\in X=\CC^3$ as $t\longrightarrow+0$.
We thus obtain 
\begin{align}
&\lim_{t\to+0}\Bigl\{[T_{Z\setminus Y}^\ast(X\setminus Y)]+t(df+d\log g)\Bigr\} \\
&=[\overline{{T_{Z_\reg}^\ast X}}] 
+ 2\cdot[T_{{K_+\setminus\{0\}}}^\ast(X\setminus \{0\})] 
+ 2\cdot[T_{{K_-\setminus\{0\}}}^\ast(X\setminus \{0\})]
+ 4\cdot[T_{\{0\}}^\ast X] 
\end{align}   
and hence
\begin{align}\label{eq-CCN}
\CCyc(\SM)&=\lim_{t\to+0}\Bigl\{\CCyc(\SN\vbar_{X\setminus Y})+t(df+d\log g)\Bigr\} \\
&=[\overline{{T_{Z_\reg}^\ast X}}] 
+[T_{L_+\setminus\{0\}}^\ast(X\setminus\{0\})]
+[T_{L_-\setminus\{0\}}^\ast(X\setminus\{0\})] \\
&+ 2\cdot[T_{{K_+\setminus\{0\}}}^\ast(X\setminus \{0\})] 
+ 2\cdot[T_{{K_-\setminus\{0\}}}^\ast(X\setminus \{0\})]
+ 8\cdot[T_{\{0\}}^\ast X].
\end{align}
In particular, we find that the multiplicity of $[T_{\{0\}}^\ast X]$ in 
the characteristic cycle $\CCyc(\SM)$ of $\SM$ is equal to 8.
On the other hand, by Kashiwara's formula in \cite{Kas83}, for the 
Euler obstruction $\Eu_Z$ of the complex hypersurface 
$Z=\{x^2+y^2+z^2=0\}\subset X=\CC^3$ we obtain 
\begin{equation}
\Eu_Z(0)=1+(-1)^3\cdot1=1-1=0.
\end{equation}
Hence it follows from \eqref{eq-CCN} that 
\begin{equation}
\chi(\Sol_X(\SM))(0)=0+1+1+2+2-8=-2.
\end{equation}
We can verify this result also by taking the inverse image of $\SM$ by 
the blow-up $X^\prime\longrightarrow X$ of $X=\CC^3$ along the origin
$\{0\}\subset X$ as in the last part of Section \ref{sec:qNor}.
But we omit the details.
\end{example}

\appendix
\section{A supplement to the proof of Proposition \ref{prop-1}}
Let $l\geq2$. For positive integers $k_1,k_2,\dots,k_l\in\ZZ_{>0}$, we 
define a closed submanifold $T_{k_1,\dots,k_l}\subset(S^1)^l$ by 
\begin{equation}
T_{k_1,\dots,k_l}\coloneq
\{(e^{\sqrt{-1}\theta_1},\dots,e^{\sqrt{-1}\theta_l})\in(S^1)^l
\mid \theta_i\in\RR, k_1\theta_1+\dots+k_l\theta_l\in 2\pi\ZZ \}.
\end{equation}
We set $W_{k_1,\dots,k_l}\coloneq (S^1)^l\setminus T_{k_1,\dots,k_l}$.
In this appendix, for the proof of Proposition \ref{prop-1} we will compute 
the cohomology groups $H^\ast((S^1)^l; \CC_{W_{k_1,\dots,k_l}})$ of 
the sheaf $\CC_{W_{k_1,\dots,k_l}}$ on $(S^1)^l$. 
We set 
\begin{equation}
d\coloneq \gcd(k_1,\dots,k_l) \quad \in \ZZ.
\end{equation}

\subsection{The case of $d=1$}\label{subsec-A1}
In this subsection, we assume that $d=1$.
We introduce some notations.
For $1\leq m\leq l$ we set  $d_m\coloneq \gcd(k_1,\dots,k_m)$ and 
\begin{equation}
k_m^\prime \coloneq
\begin{dcases}
\frac{k_1}{d_2} & (m=1) \\
\frac{k_m}{d_m} & (m\geq2).
\end{dcases}
\end{equation}
For $2\leq m\leq l-1$ we set $d_m^\prime\coloneq \displaystyle\frac{d_m}{d_{m+1}}$.
For $2\leq m\leq l$ we fix integers $N_1^{(m)},N_2^{(m)},\dots,N_m^{(m)}\in\ZZ$ such that
\begin{equation}\label{eq-Bez}
k_1^\prime N_1^{(m)} + k_2^\prime N_2^{(m)} + \dots + k_m^\prime N_m^{(m)} = 1.
\end{equation}
(Note that the integers $k_1^\prime,\dots,k_m^\prime$ are coprime.)
Let us define a morphism $\phi\colon (S^1)^{l-1}\longrightarrow (S^1)^l$
by 
\begin{equation}
(e^{\sqrt{-1} t_1},\dots,e^{\sqrt{-1} t_{l-1}})\in(S^1)^{l-1} 
\longmapsto 
(e^{\sqrt{-1} \phi_1(t_1,\dots,t_{l-1})},\dots,e^{\sqrt{-1} 
\phi_{l}(t_1,\dots,t_{l-1})})\in(S^1)^l
\end{equation}
where we set
\begin{empheq}{alignat=8}
\phi_1(t_1,\cdots,t_{l-1})\coloneq & & k_2^\prime & t_1 & +  k_3^\prime 
N_1^{(2)} &t_2 & + k_4^\prime N_1^{(3)} & t_3 & + & \cdots & + &
k_{l-1}^\prime N_1^{(l-2)} & t_{l-2} & + k_l N_1^{(l-1)} & t_{l-1}, \notag \\
\phi_2(t_1,\cdots,t_{l-1})\coloneq & - & k_1^\prime & t_1 & + k_3^\prime 
N_2^{(2)} &t_2 & + k_4^\prime N_2^{(3)} & t_3 & + & \cdots & + &
k_{l-1}^\prime N_2^{(l-2)} & t_{l-2} & + k_l N_2^{(l-1)} & t_{l-1}, \notag \\
\phi_3(t_1,\cdots,t_{l-1})\coloneq & &  & & -d_2^\prime &
t_2 & + k_4^\prime N_3^{(3)} & t_3 & + & \cdots & + &
k_{l-1}^\prime N_3^{(l-2)} & t_{l-2} & + k_l N_3^{(l-1)} & t_{l-1}, \notag \\
\phi_4(t_1,\cdots,t_{l-1})\coloneq & & & & & & - d_3^\prime & t_3 & + & \cdots & + &
k_{l-1}^\prime N_4^{(l-2)} & t_{l-2} & + k_l N_4^{(l-1)} & t_{l-1}, \notag \\
& & & & & & & & & \vdots & & & & & \notag \\
\phi_{l-1}(t_1,\cdots,t_{l-1})\coloneq & & & & & & & & & & &
\quad \, -d_{l-2}^\prime & t_{l-2} & + k_l N_{l-1}^{(l-1)} & t_{l-1}, \notag \\
\phi_l(t_1,\cdots,t_{l-1})\coloneq & & & & & & & & & & & & & \quad 
\,\,\, -d_{l-1} & t_{l-1}. \notag
\end{empheq}   
Then we can easily check that $\phi((S^1)^{l-1})\subset T_{k_1,\dots,k_l}$.
We have the following lemma.
\begin{lemma}
The morphism $\phi\colon (S^1)^{l-1}\longrightarrow (S^1)^l$ induces 
a diffeomorphism from $(S^1)^{l-1}$ to $T_{k_1,\dots,k_l}$.
\end{lemma}

\begin{proof}
The Jacobian matrix of $\phi$ (at any point in $(S^1)^{l-1}$) is given by
\begin{equation}
\begin{pmatrix}
k_2^\prime & k_3^\prime N_1^{(2)} & \cdots & k_{l-1}^\prime N_1^{(l-2)} & k_l N_1^{(l-1)} \\
-k_1^\prime & k_3^\prime N_2^{(2)} & \cdots & k_{l-1}^\prime N_2^{(l-2)} & k_l N_2^{(l-1)} \\
& -d_2^\prime & \cdots & k_{l-1}^\prime N_3^{(l-2)} & k_l N_3^{(l-1)} \\
&& \ddots & \vdots & \vdots \\
&\text{\huge{0}}&& -d_{l-2}^\prime & k_l N_{l-1}^{(l-1)} \\
&&&& -d_{l-1}  
\end{pmatrix}.
\end{equation}
Since the rank of it is equal to $l-1$, the morphism $\phi\colon 
(S^1)^{l-1}\longrightarrow (S^1)^l$ is an immersion.
Hence it suffices to show that $\phi$ induces a bijection between 
$(S^1)^{l-1}$ and $T_{k_1,\dots,k_l}$.
Let us first treat the case of $l=2$.
In this case, we have $k_1^\prime=k_1$ and $k_2^\prime=k_2$.
If $e^{\sqrt{-1} t_1}, e^{\sqrt{-1} t_1^\prime}\in S^1$ $(t_1,t_1^\prime\in\RR)$ 
satisfy $\phi(e^{\sqrt{-1} t_1})=\phi(e^{\sqrt{-1} t_1^\prime})$, then 
by the definition of $\phi$ we have $k_2(t_1-t_1^\prime), k_1(t_1-t_1^\prime)\in 2\pi\ZZ$.
By the assumption that $d=\gcd(k_1,k_2)=1$, we obtain 
$t_1-t_1^\prime\in 2\pi\ZZ$ and hence $\phi$ is injective.
If $(e^{\sqrt{-1} \theta_1},e^{\sqrt{-1} \theta_2})\in T_{k_1,k_2}$ $(\subset(S^1)^2)$ 
$(0\leq\theta_1, \theta_2< 2\pi)$ is given,
then there exists an integer $M\in\ZZ$ such that $k_1\theta_1+ k_2\theta_2=2\pi M$.
Since $d=1$,  there exists an integer $b_2\in\ZZ$ such that $k_2 b_2 +M\in k_1\ZZ$.
We set $t_1\coloneq \displaystyle -\frac{\theta_2+2\pi b_2}{k_1}$.
Then we can show that $k_2t_1-\theta_1, -k_1t_1-\theta_2\in 2\pi\ZZ$.
Thus we obtain $\phi(e^{\sqrt{-1} t_1})=(e^{\sqrt{-1} \theta_1},e^{\sqrt{-1} \theta_2})$.
Therefore, $\phi$ induces a bijection between $S^1$ and $T_{k_1,k_2}$.
Next, let us consider the case of $l\geq3$.
If $(e^{\sqrt{-1} t_1},\dots,e^{\sqrt{-1} t_{l-1}}), (e^{\sqrt{-1} t_1^\prime},
\dots,e^{\sqrt{-1} t_{l-1}^\prime})\in (S^1)^{l-1}$ $(t_i,t_i^\prime\in\RR)$ satisfy 
$\phi\bigl((e^{\sqrt{-1} t_1},\dots,e^{\sqrt{-1} t_{l-1}})\bigr)=
\phi\bigl((e^{\sqrt{-1} t_1^\prime},\dots,e^{\sqrt{-1} t_{l-1}^\prime})
\bigr)$, then by the definition of $\phi$ we have
\begin{empheq}[left=\empheqlbrace]{alignat=7}
k_2^\prime\tl{t_1} & +k_3^\prime N_1^{(2)}\tl{t_2} &+ \cdots &
 +k_{l-1}^\prime N_1^{(l-2)}\tl{t_{l-2}} & +k_l N_1^{(l-1)}\tl{t_{l-1}} & 
\in 2\pi\ZZ \tag{$B_{1}$}\\
-k_1^\prime\tl{t_1} & +k_3^\prime N_2^{(2)}\tl{t_2} &+ \cdots & 
+k_{l-1}^\prime N_2^{(l-2)}\tl{t_{l-2}} & +k_l N_2^{(l-1)}
\tl{t_{l-1}} & \in 2\pi\ZZ \tag{$B_{2}$}\\
& \qquad -d_2^\prime\tl{t_2} &+ \cdots & +k_{l-1}^\prime N_3^{(l-2)}\tl{t_{l-2}} 
& +k_l N_3^{(l-1)}\tl{t_{l-1}} & \in 2\pi\ZZ \tag{$B_{3}$}\\
&& & \vdots & \notag \\
&&& \qquad \quad -d_{l-2}^\prime\tl{t_{l-2}} & +k_l N_{l-1}^{(l-1)}\tl{t_{l-1}} 
& \in 2\pi\ZZ \tag{$B_{l-1}$}\\
&&&& -d_{l-1}\tl{t_{l-1}}  & \in 2\pi\ZZ \tag{$B_{l}$}
\end{empheq}
where for each $1\leq m\leq l-1$ we set $\tl{t_m}\coloneq t_m-t_m^\prime$.
It follows from $k_1\times(B_1)+k_2\times(B_2)+\dots+k_{l-1}\times(B_{l-1})$ and (\ref{eq-Bez}) that
\begin{equation}
k_l d_{l-1} \tl{t_{l-1}}\in 2\pi(k_1\ZZ+\dots+k_l\ZZ).
\end{equation}
Since $k_1\ZZ+\dots+k_l\ZZ= d_{l-1}\ZZ$, we get $k_l\tl{t_{l-1}}\in 2\pi\ZZ$.
From $(B_l)$, we have $d_{l-1}\tl{t_{l-1}}\in 2\pi\ZZ$.
By the assumption that $d=1$, the integers $d_{l-1}$ and $k_l$ are coprime, and hence 
\begin{equation}
\tl{t_{l-1}}=t_{l-1}-t_{l-1}^\prime\in 2\pi\ZZ.
\end{equation} 
Then, it follows from $(B_1),\dots,(B_{l-1})$ that
\begin{empheq}[left=\empheqlbrace]{alignat=7}
k_2^\prime\tl{t_1} & +k_3^\prime N_1^{(2)}\tl{t_2} &+ \cdots & 
+k_{l-1}^\prime N_1^{(l-2)}\tl{t_{l-2}} & \in 2\pi\ZZ \\
-k_1^\prime\tl{t_1} & +k_3^\prime N_2^{(2)}\tl{t_2} &+ \cdots & 
+k_{l-1}^\prime N_2^{(l-2)}\tl{t_{l-2}} & \in 2\pi\ZZ \\
& \qquad -d_2^\prime\tl{t_2} &+ \cdots & +k_{l-1}^\prime 
N_3^{(l-2)}\tl{t_{l-2}} & \in 2\pi\ZZ \\
&& \vdots   &  \notag \\
&&& \qquad -d_{l-2}^\prime\tl{t_{l-2}}  & \in 2\pi\ZZ. 
\end{empheq}
By repeating the same procedure, we obtain 
\begin{equation}
t_1-t_1^\prime,t_2-t_2^\prime,\dots,t_{l-2}-t_{l-2}^\prime \in 2\pi\ZZ.
\end{equation} 
This implies that $\phi\colon (S^1)^{l-1}\longrightarrow (S^1)^l$ is injective.
If $(e^{\sqrt{-1} \theta_1},\dots,e^{\sqrt{-1} \theta_l})\in T_{k_1,
\dots,k_l}$ $(\subset(S^1)^l)$ $(0\leq\theta_i<2\pi)$ is given, 
then there exists an integer $M\in\ZZ$ such that 
\begin{equation}\label{eq-M}
k_1\theta_1+k_2\theta_2+\dots+k_l\theta_l= 2\pi M.
\end{equation}
Since $d=\gcd(k_1,\dots,k_l)=1$, there exist integers $b_2,\dots,b_l\in\ZZ$ such that 
\begin{equation}\label{eq-b}
k_2 b_2 + k_3 b_3 +\dots+ k_l b_l + M \in k_1\ZZ.
\end{equation}
We take real numbers $t_1,\dots,t_{l-1}\in\RR$ such that
\begin{empheq}[left=\empheqlbrace]{alignat=10}
-k_1^\prime t_1 & +k_3^\prime N_2^{(2)}t_2  & + \cdots & +
k_{l-1}^\prime N_2^{(l-2)}t_{l-2} & +k_l N_2^{(l-1)}t_{l-1} &=
\theta_2+2\pi b_2 \tag{$C_{2}$}\\
& \qquad -d_2^\prime t_2  & + \cdots & +k_{l-1}^\prime N_3^{(l-2)}
t_{l-2} & +k_l N_3^{(l-1)}t_{l-1} &=\theta_3+2\pi b_3 \tag{$C_{3}$}\\
&&& \vdots &&  \notag \\
&&& \qquad \quad -d_{l-2}^\prime t_{l-2} & +k_l N_{l-1}^{(l-1)}
t_{l-1} &=\theta_{l-1}+2\pi b_{l-1} \tag{$C_{l-1}$}\\
&&&& \qquad -d_{l-1}t_{l-1} &=\theta_l+2\pi b_l \tag{$C_{l}$}
\end{empheq}
By $k_2\times(C_2)+k_3\times(C_3)+\dots+k_l\times(C_l)$ and (\ref{eq-Bez}), we have 
\begin{equation}\label{eq-k}
\begin{split}
&-k_1(k_2^\prime t_1+k_3^\prime N_1^{(2)}t_2+\dots+k_{l-1}^\prime 
N_1^{(l-2)}t_{l-2}+k_l N_1^{(l-1)}t_{l-1}) \\
&=k_2\theta_2+\dots+k_l\theta_l+2\pi(k_2b_2+\dots+k_lb_l).
\end{split}
\end{equation}
It follows from (\ref{eq-M}), (\ref{eq-b}) and (\ref{eq-k}) that 
\begin{equation}
k_2^\prime t_1+k_3^\prime N_1^{(2)}t_2+\dots+k_{l-1}^\prime 
N_1^{(l-2)}t_{l-2}+k_l N_1^{(l-1)}t_{l-1}-\theta_1\in 2\pi\ZZ.
\end{equation}
This implies that $\phi\bigl((e^{\sqrt{-1} t_1},\dots,
e^{\sqrt{-1} t_{l-1}})\bigr)=(e^{\sqrt{-1} \theta_1},\dots,e^{\sqrt{-1} \theta_l})$.
Therefore, $\phi$ induces a bijection between $(S^1)^{l-1}$ and $T_{k_1,\dots,k_l}$.
This completes the proof. 
\end{proof}

From now on, for an integer $j\in\ZZ$ and a smooth manifold $M$ 
denote by $H_\dR^j(M)$ the $j$-th de Rham cohomology group of $M$.
\begin{lemma}\label{lem-d1rank}
For $j\in\ZZ$ let $H^j\phi^\ast\colon H_\dR^j((S^1)^l)\longrightarrow 
H_\dR^j((S^1)^{l-1})$ be the morphism induced by
$\phi\colon (S^1)^{l-1}\longrightarrow (S^1)^l$.
Then we have
\begin{equation}
\rank H^j\phi^\ast=
\begin{dcases}
\binom{l-1}{j} & (0\leq j\leq l-1) \\
\\
\quad \,\, 0 & \mathrm{(otherwise)}.
\end{dcases}
\end{equation} 
\end{lemma}

\begin{proof}
For the polar coordinates $(\theta_1,\dots,\theta_l)$ of $(S^1)^l
=\{(e^{\sqrt{-1} \theta_1},\dots,e^{\sqrt{-1} \theta_l})\}$ and \\
$(t_1,\dots,t_{l-1})$ of $(S^1)^{l-1}=\{(e^{\sqrt{-1} t_1},
\dots,e^{\sqrt{-1} t_{l-1}})\}$, the de Rham cohomology classes 
$[d\theta_1],\dots,[d\theta_l]\in H_\dR^1((S^1)^l)$ and $[dt_1],
\dots,[dt_{l-1}]\in H_\dR^1((S^1)^{l-1})$ are the bases of $H_\dR^1
((S^1)^l)$ and $H_\dR^1((S^1)^{l-1})$, respectively.
From the definition of $\phi$, the matrix representation of 
$H^1\phi^\ast\colon H_\dR^1((S^1)^l)\longrightarrow 
H_\dR^1((S^1)^{l-1})$ with respect to the above basis is
\begin{equation}
\begin{pmatrix}
k_2^\prime & k_3^\prime N_1^{(2)} & \cdots & k_{l-1}^\prime N_1^{(l-2)} & k_l N_1^{(l-1)} \\
-k_1^\prime & k_3^\prime N_2^{(2)} & \cdots & k_{l-1}^\prime N_2^{(l-2)} & k_l N_2^{(l-1)} \\
& -d_2^\prime & \cdots & k_{l-1}^\prime N_3^{(l-2)} & k_l N_3^{(l-1)} \\
&& \ddots & \vdots & \vdots \\
&\text{\huge{0}}&& -d_{l-2}^\prime & k_l N_{l-1}^{(l-1)} \\
&&&& -d_{l-1}  
\end{pmatrix}.
\end{equation}
Therefore, we obtain $\rank H^1\phi^\ast=l-1$. 
Namely $H^j\phi^\ast$ is surjective. 
By the K\"{u}nneth formula, for $j\in\ZZ$ there exists a commutative diagram
\begin{equation}
\vcenter{
\xymatrix@M=7pt{
H_\dR^j((S^1)^l) \ar[r]^-{\sim} \ar@{->}[d]_-{H^j\phi^\ast} &
\bigwedge^j H_\dR^1((S^1)^l) \ar@{->}[d]^-{\bigwedge^j H^1\phi^\ast} \\
H_\dR^j((S^1)^{l-1}) \ar[r]^-{\sim} &
\bigwedge^j H_\dR^1((S^1)^{l-1})
}}\end{equation}   
where $\bigwedge^j H_\dR^1((S^1)^l)$ (resp. $\bigwedge^j H_\dR^1((S^1)^{l-1}), 
\bigwedge^j H^1\phi^\ast$) is the $j$-th exterior power of $H_\dR^1((S^1)^l)$ 
(resp. $H_\dR^1((S^1)^{l-1}), H^1\phi^\ast$).
Thus we have 
\begin{align}
\rank H^j\phi^\ast&=\rank(\textstyle\bigwedge^j H^1\phi^\ast) \\
&=
\begin{dcases}
\binom{l-1}{j} & (0\leq j\leq l-1) \\
\\
\quad \,\, 0 & \mathrm{(otherwise)}.   
\end{dcases}
\end{align}
This completes the proof.
\end{proof}

\subsection{The computation of $H^\ast((S^1)^l; \CC_{W_{k_1,\dots,k_l}})$}
In this subsection, we treat the general case where $d=\gcd(k_1,\dots,k_l)$ is not necessarily equal to 1.
For $1\leq m\leq l$ we set $\displaystyle\tl{k_m}\coloneq\frac{k_m}{d}$.
For $0\leq j\leq d-1$ we define a closed submanifold 
$T_{\tl{k_1},\dots,\tl{k_l}}^{(j)}\subset (S^1)^l$ by
\begin{equation}
T_{\tl{k_1},\dots,\tl{k_l}}^{(j)}\coloneq
\Set*{(e^{\sqrt{-1}\theta_1},\dots,e^{\sqrt{-1}\theta_l})\in(S^1)^l
} {\theta_i\in\RR, \tl{k_1}\theta_1+\dots+\tl{k_l}\theta_l\in \frac{2j\pi}{d} + 2\pi\ZZ  }.
\end{equation}
Then we have 
\begin{equation}
T_{k_1,\dots,k_l}=\bigsqcup_{j=0}^{d-1} T_{\tl{k_1},\dots,\tl{k_l}}^{(j)}.
\end{equation}
Since $\gcd(\tl{k_1},\dots,\tl{k_l})=1$, as in Section 
\ref{subsec-A1}, we can define a morphism $\phi\colon (S^1)^{l-1}\longrightarrow 
(S^1)^l$ which induces a diffeomorphism from $(S^1)^{l-1}$ to $T_{\tl{k_1},\dots,
\tl{k_l}}=T_{\tl{k_1},\dots,\tl{k_l}}^{(0)}$.
For $0\leq j\leq d-1$ we define a morphism $\psi_j\colon(S^1)^l\longrightarrow (S^1)^l$ by
\begin{equation}
(e^{\sqrt{-1}\theta_1},e^{\sqrt{-1}\theta_2},\dots,e^{\sqrt{-1}\theta_l})\in(S^1)^l \longmapsto
(e^{\sqrt{-1}(\theta_1+\frac{2j\pi }{d \tl{k_1}})},
e^{\sqrt{-1}\theta_2},\dots,e^{\sqrt{-1}\theta_l})\in(S^1)^l.
\end{equation}
Note that for each $0\leq j\leq d-1$ the morphism $\psi_j$ induces a 
diffeomorphism from $T_{\tl{k_1},\dots,\tl{k_l}}^{(0)}$ to $T_{\tl{k_1},\dots,\tl{k_l}}^{(j)}$.
Then the morphisms $\psi_j\circ\phi\colon (S^1)^{l-1}\longrightarrow (S^1)^l$ 
($0\leq j\leq d-1$) induce the one 
$\Phi\colon\displaystyle\bigsqcup_{d}(S^1)^{l-1}\longrightarrow(S^1)^l$.
The following lemma is clear.

\begin{lemma}\label{lem-diffeo}
In the above situation, the morphism 
$\Phi\colon\displaystyle\bigsqcup_{d}(S^1)^{l-1}\longrightarrow(S^1)^l$ 
induces a diffeomorphism from $\displaystyle\bigsqcup_{d}(S^1)^{l-1}$ to $T_{k_1,\dots,k_l}$.
\end{lemma}

Similarly to Lemma \ref{lem-d1rank}, we have the following lemma.
\begin{lemma}\label{lem-drank}
As in Lemma \ref{lem-d1rank}, for $j\in\ZZ$ we define the morphism  $H^j\Phi^\ast\colon 
H_\dR^j((S^1)^l)\longrightarrow H_\dR^j(\displaystyle\bigsqcup_d(S^1)^{l-1})$ 
induced by
$\Phi\colon \displaystyle\bigsqcup_d(S^1)^{l-1}\longrightarrow (S^1)^l$.
Then we have
\begin{equation}
\rank H^j\Phi^\ast=
\begin{dcases}
\binom{l-1}{j} & (0\leq j\leq l-1) \\
\\
\quad \,\, 0 & \mathrm{(otherwise)}.
\end{dcases}
\end{equation} 
\end{lemma}

Finally, we compute the cohomology groups $H^\ast((S^1)^l; \CC_{W_{k_1,\dots,k_l}})$ as follows.
\begin{proposition}\label{prop-appA}
We have isomorphisms
\begin{equation}
H^j((S^1)^l;\CC_{W_{k_1,\dots,k_l}})\simeq
\begin{dcases}
\CC^{d\cdot \binom{l-1}{j-1}} & (1\leq j\leq l) \\
\\
\quad \, 0 & \mathrm{(otherwise)}.
\end{dcases}
\end{equation}
\end{proposition}

\begin{proof}
Recall that $W_{k_1,\dots,k_l}=(S^1)^l\setminus T_{k_1,\dots,k_l}$.
We have an exact sequence
\begin{equation}
0\longrightarrow
\CC_{W_{k_1,\dots,k_l}} \longrightarrow
\CC_{(S^1)^l} \longrightarrow
\CC_{T_{k_1,\dots,k_l}} \longrightarrow
0.
\end{equation}
For $j\in\ZZ$, let $\alpha_j\colon H^j((S^1)^l;\CC_{(S^1)^l})
\longrightarrow H^j((S^1)^l;\CC_{T_{k_1,\dots,k_l}})$ be the linear map 
induced by the morphism $\CC_{(S^1)^l} \longrightarrow\CC_{T_{k_1,\dots,k_l}}$.
Then for each $j\in\ZZ$ we have
\begin{equation}\label{eq-AH1}
\begin{split}
&\dim H^j((S^1)^l;\CC_{W_{k_1,\dots,k_l}}) \\
&=\dim H^j((S^1)^l;\CC_{(S^1)^l}) -\rank\alpha_j
+\dim H^j((S^1)^l;\CC_{T_{k_1,\dots,k_l}}) -\rank\alpha_{j-1}.
\end{split}
\end{equation}
From Lemma \ref{lem-diffeo}, there is a commutative diagram
\begin{equation}
\vcenter{
\xymatrix@M=5pt{
\displaystyle\bigsqcup_d(S^1)^{l-1} \ar[d]^-[@!-90]{\sim} \ar@{^{(}->}[r]^-{\Phi} & 
(S^1)^l & \\
T_{k_1,\dots,k_l}. \ar@{^{(}->}[ur] & &  
}}
\end{equation}
Thus it follows from the de Rham Theorem that $\rank\alpha_j =\rank H^j\Phi^\ast$ 
for each $j\in\ZZ$.
By Lemma \ref{lem-drank}, we obtain
\begin{equation}\label{eq-AH2}
\rank\alpha_j=
\begin{dcases}
\binom{l-1}{j} & (0\leq j\leq l-1) \\
\\
\quad \,\, 0 & \mathrm{(otherwise)}.
\end{dcases}
\end{equation}
On the other hand, since $T_{k_1,\dots,k_l}$ is diffeomorphic to 
$\displaystyle\bigsqcup_d(S^1)^{l-1}$, we have isomorphisms
\begin{equation}\label{eq-AH3}
H^j((S^1)^l;\CC_{T_{k_1,\dots,k_l}})\simeq
\begin{dcases}
\CC^{d\cdot \binom{l-1}{j}} & (0\leq j\leq l-1) \\
\\
\quad \, 0 & \mathrm{(otherwise)}.
\end{dcases}
\end{equation}   
We also have isomorphisms
\begin{equation}\label{eq-AH4}
H^j((S^1)^l;\CC_{(S^1)^l})\simeq
\begin{dcases}
\CC^{\binom{l}{j}} & (0\leq j\leq l) \\
\\
\quad \! \! 0 & \mathrm{(otherwise)}.
\end{dcases}
\end{equation}   
Now the assertion immediately follows from (\ref{eq-AH1}), (\ref{eq-AH2}), 
(\ref{eq-AH3}) and (\ref{eq-AH4}).
\end{proof}

\end{document}